\documentclass[a4paper,10pt]{article}
\makeatletter
\newcommand*{\rom}[1]{\expandafter\@slowromancap\romannumeral #1@}
\makeatother

\usepackage{lipsum,relsize,graphicx,amsmath,amsthm,amssymb,mathtools}
\usepackage[margin=0.84in]{geometry}
\usepackage{hyperref}
\usepackage{authblk}

\DeclarePairedDelimiter\floor{\lfloor}{\rfloor}
\usepackage{tocloft}

\newtheorem{theorem}{Theorem}[section]
\newtheorem{lemma}[theorem]{Lemma}
\newtheorem{corollary}[theorem]{Corollary}
\newtheorem{definition}{Definition}[section]
\newtheorem{proposition}[theorem]{Proposition}

\begin{document}

\setcounter{tocdepth}{1}

\title{Completed Group Algebras of Abelian-by-procyclic Groups}

\author{Adam Jones}

\date{\today}

\maketitle

\begin{abstract}

\noindent Let $p$ be prime, let $G$ be a $p$-valuable group, isomorphic to $\mathbb{Z}_p^d\rtimes\mathbb{Z}_p$, and let $k$ be a field of characteristic $p$. We will prove that all faithful prime ideals of the completed group algebra $kG$ are controlled by $Z(G)$, and a complete decomposition for $Spec(kG)$ will follow. The principal technique we employ will be to study the convergence of Mahler expansions for inner automorphisms.

\end{abstract}

\tableofcontents

\setcounter{section}{0}
\section{Introduction}

Let $p$ be a prime, $k$ be a field of characteristic $p$, and let $G$ be a compact $p$-adic Lie group. The \emph{completed group algebra}, or \emph{Iwasawa algebra} of $G$ with respect to $k$ is defined as:

\begin{equation}
kG:=\underset{N\trianglelefteq_o G}{\varprojlim} k[G/N].
\end{equation}

\noindent We aim to improve our understanding of the prime ideal structure of completed group algebras, which would have profound consequences for the representation theory of compact $p$-adic Lie groups.\\

\subsection{Background} 

In \cite[Section 6]{survey}, Ardakov and Brown ask a number of questions regarding the two-sided ideal structure of $kG$. Several of these have now been answered or partially answered, but a number of them remain open. We hope that this paper will take important steps towards providing an answer to these questions.\\

\noindent First recall some important definitions: 

\begin{definition}

Let $I$ be a right ideal of $kG$:\\

\noindent 1. We say that $I$ is \emph{faithful} if for all $g\in G$, $g-1\in I$ if and only if $g=1$, i.e. $G\to\frac{kG}{I},g\mapsto g+I$ is injective.\\

\noindent 2. We say that $H\leq_c G$ \emph{controls} $I$ if $I=(I\cap kH)kG$.\\

\noindent Define the \emph{controller subgroup} of $I$ by $I^{\chi}:=\bigcap\{U\leq_o G: U$ controls $I\}$, and denote by $Spec^f(kG)$ the set of all faithful prime ideals of $kG$.

\end{definition}

\noindent Also recall the following useful result \cite[Theorem A]{controller}:

\begin{theorem}

Let $I$ be a right ideal of $kG$. Then $I^{\chi}\trianglelefteq_c G$, and $H\leq_c G$ controls $I$ if and only if $I^{\chi}\subseteq H$.

\end{theorem}

\noindent We will assume throughout that $G$ is $p$-valuable, i.e. carries a complete $p$-valuation $\omega:G\to\mathbb{R}\cup\{\infty\}$ in the sense of \cite[\rom{3} 2.1.2]{Lazard}.\\

\noindent\textbf{\underline{Notation:}} we say $N\trianglelefteq_c^i G$ to mean that $N$ is a closed, isolated, normal subgroup of $G$.\\

\noindent One of the main open problems in the subject concerns proving a control theorem for faithful prime ideals in $kG$. To illustrate the importance of this problem, consider the following result (\cite[Theorem A]{nilpotent}):

\begin{theorem}\label{decomposition}

Let $G$ be a $p$-valuable group such that for every $N\trianglelefteq_c^i G$, $P\in Spec^f(k\frac{G}{N})$, $P$ is controlled by the centre of $\frac{G}{N}$, i.e. $P^{\chi}\subseteq Z(\frac{G}{N})$.\\

\noindent Then every prime ideal $P$ of $kG$ is completely prime, i.e. $\frac{kG}{P}$ is a domain, and we have a bijection:

\begin{center}
$Spec(kG)\leftrightarrow \underset{N\trianglelefteq_c^i G}{\coprod}{Spec^f(kZ(\frac{G}{N}))}$
\end{center}

\end{theorem}

\noindent\textbf{\underline{Note:}} This is stronger than the statement given in \cite{nilpotent}, but it has an identical proof.\\

This Theorem is the strongest result we have to date concerning ideal classification for $kG$. Using it, we reduce the problem of classifying ideals in the non-commutative ring $kG$ to the simpler problem of classifying ideals in the \emph{commutative strata} $kZ(\frac{G}{N})$.\\

\noindent The result gives a positive answer to \cite[Question N]{survey} for all groups $G$ for which we have the appropriate control theorem for faithful primes in $kG$, i.e. for all groups $G$ such that $P^{\chi}\subseteq Z(\frac{G}{N})$ for $P$ faithful. 

It follows from the following result (\cite[Theorem 8.4]{nilpotent}) that this is true if $G$ is nilpotent.

\begin{theorem}\label{nilpotence}

Let $G$ be a $p$-valuable group, $P\in Spec^f(kG)$, and suppose that $P$ is controlled by a nilpotent subgroup of $G$. Then $P$ is controlled by the centre of $G$.

\end{theorem}

\noindent We want to extend this result to more general groups, in particular when $G$ is solvable, and thus give a positive answer to \cite[Question O]{survey}.

\subsection{Aims}

In this paper, we will consider certain classes of metabelian groups, i.e. groups whose commutator subgroup is abelian.\\

\noindent A $p$-valuable group $G$ is \emph{abelian-by-procyclic} if it is isomorphic to $\mathbb{Z}_p^d\rtimes\mathbb{Z}_p$ for some $d\in\mathbb{N}$. The completed group algebras of these groups have the form of skew-power series rings $R[[x;\sigma,\delta]]$, for $R$ a commutative power series ring.\\ 

\noindent Our main result establishes the control condition in Theorem \ref{decomposition} for groups of this form:

\begin{theorem}\label{A}

Let $G$ be a $p$-valuable, abelian-by-procyclic group. Then every faithful prime ideal of $kG$ is controlled by $Z(G)$.

\end{theorem}

\noindent Using Lemma \ref{Lie-invariant} below, we see that if $G$ is an abelian-by-procyclic group, and $N$ is a closed, isolated normal subgroup, then $\frac{G}{N}$ is also abelian-by-procyclic, and hence by the theorem, every faithful prime ideal of $k\frac{G}{N}$ is controlled by the centre. Hence we can apply Theorem \ref{decomposition} to get a complete decomposition for $Spec(kG)$.\\

\noindent In particular, if $G$ is any solvable, $p$-valuable group of rank 2 or 3, then $G$ is abelian-by-procyclic, so again, we can completely determine $Spec(kG)$.\\

\noindent Most of the work in this paper will go towards proving the following theorem.

\begin{theorem}\label{B}

Let $G$ be a non-abelian, $p$-valuable, abelian-by-procyclic group. If $P\in Spec^f(kG)$ then $P$ is controlled by a proper, open subgroup of $G$.

\end{theorem}

\vspace{0.1in}

\noindent This result is actually all we need to deduce Theorem \ref{A}:\\

\noindent\emph{Proof of Theorem \ref{A}.} Recall from \cite[Definition 5.5]{nilpotent} that a prime ideal $P$ of $kG$ is \emph{non-splitting} if for all $U\leq_o G$ controlling $P$, $P\cap kU$ is prime in $kU$.\\

\noindent Fix $G$ a $p$-valuable, abelian-by-procyclic group, and suppose that $P\in Spec^f(G)$ is non-splitting.\\

\noindent Consider the controller subgroup $P^{\chi}$ of $P$. Using \cite[Proposition 5.5]{nilpotent} and the non-splitting property, we see that $Q:=P\cap kP^{\chi}$ is a faithful prime ideal of $kP^{\chi}$. We know that $P^{\chi}$ is a closed, normal subgroup of $G$, so it follows from Lemma \ref{Lie-invariant} below that $P^{\chi}$ is abelian-by-procyclic.\\

\noindent If $P^{\chi}$ is abelian, then it follows from Theorem \ref{nilpotence} that $P^{\chi}$ is central in $G$, i.e. $P^{\chi}\subseteq Z(G)$ and $P$ is controlled by $Z(G)$ as required. So assume for contradiction that $P^{\chi}$ is non-abelian:\\

\noindent Applying Theorem \ref{B} gives that $Q$ is controlled by a proper open subgroup $U$ of $kP^{\chi}$, i.e. $Q=(Q\cap kU)kP^{\chi}$. But $P$ is controlled by $P^{\chi}$ by \cite[Theorem A]{controller}, so: 

\begin{center}
$P=(P\cap kP^{\chi})kG=QkG=(Q\cap kU)kP^{\chi}kG=(P\cap kU)kG$
\end{center} 

\noindent Hence $P$ is controlled by $U$, which is a proper subgroup of $P^{\chi}$ -- contradiction.\\

\noindent So, we conclude that any faithful, non-splitting prime ideal of $kG$ is controlled by $Z(G)$. Now suppose that $I$ is a faithful, \emph{virtually non-splitting} right ideal of $kG$, i.e. $I=PkG$ for some open subgroup $U$ of $G$, $P$ a faithful, non-splitting prime ideal of $kU$.\\

\noindent Using Lemma \ref{Lie-invariant} below, we see that $U$ is $p$-valuable, abelian-by-procyclic, so by the above discussion, $P$ is controlled by $Z(U)$, and in fact, $Z(U)=Z(G)\cap U$ by \cite[Lemma 8.4]{nilpotent}. Therefore, since $I\cap kU=P$ by \cite[Lemma 5.1(ii)]{nilpotent}: 

\begin{center}
$I=PkG=(P\cap kZ(U))kUkG=(I\cap kU\cap kZ(G))kG=(I\cap kZ(G))kG$
\end{center}

\noindent So since $G$ is $p$-valuable and every faithful, virtually non-splitting right ideal of $kG$ is controlled by $Z(G)$, it follows from \cite[Theorem 5.8, Corollary 5.8]{nilpotent} that every faithful prime ideal of $kG$ is controlled by $Z(G)$ as required.\qed\\

\noindent So the remainder of our work will be to prove Theorem \ref{B}, this proof will be given at the end of section 6.

\subsection{Outline}

Throughout, we will fix $G=\mathbb{Z}_p^{d}\rtimes\mathbb{Z}_p$ a non-abelian, $p$-valuable, abelian-by-procyclic group, $P$ a faithful prime ideal of $kG$, and let $\tau:kG\to Q(\frac{kG}{P})$ be the natural map. To prove a control theorem for $P$, we will follow a similar approach to the method used in \cite{nilpotent}.\\ 

\noindent The most important notion we will need is the concept of the \emph{Mahler expansion} of an automorphism $\varphi$ of $G$, introduced in \cite[Chapter 6]{nilpotent}, and which we will recap in section 2.

In our case, we will take $\varphi$ to be the automorphism defining the action of $\mathbb{Z}_p$ on $\mathbb{Z}_p^d$.\\

\noindent In section 2, we will see how we can use the Mahler expansion of $\varphi^{p^m}$ to deduce an expansion of the form:

\begin{center}
$0=q_1^{p^m}\tau\partial_1(y)+....+q_{d}^{p^m}\tau\partial_{d}(y)+O(q^{p^m})$
\end{center}

\noindent Where $y\in P$ is arbitrary, $\partial_i:kG\to kG$ is a derivation, $q_i,q\in Q(\frac{kG}{P})$, we call the $q$'s \emph{Mahler approximations}. We want to examine convergence of this expression as $m\rightarrow\infty$ to deduce that $\tau\partial_i(P)=0$ for some $i\leq d$, from which a control theorem should follow.\\

\noindent Also, for an appropriate polynomial $f$, we will see that we have a second useful expansion:

\begin{center}
$0=f(q_1)^{p^m}\tau\partial_1(y)+....+f(q_{d})^{p^m}\tau\partial_{d}(y)+O(f(q)^{p^m})$
\end{center}

\noindent The idea is to compare the growth of the Mahler approximations with $m$ so that we can scale this expression and get that the higher order terms tend to zero, and the lower order terms converge to something non-zero.

One possiblity would be to divide out the lower order terms, which was the approach used in \cite{nilpotent}. But an important difference in our situation is that we cannot be sure that the elements $f(q_i)$ are \emph{regular}, i.e. that dividing them out of the expression will not affect convergence of the higher order terms.\\

\noindent In section 3, we will recap how to use a filtration on $kG$ to construct an appropriate filtration $v$ on $Q(\frac{kG}{P})$, and introduce the notion of a \emph{growth rate function}, which allows us to examine the growth with $m$ of our elements $f(q)^{p^m}$ with respect to $v$.\\

\noindent We will also introduce the notion of a \emph{growth preserving polynomial} (GPP), and show how a control theorem follows from the existence of such a polynomial $f$ satisfying certain $v$-regularity conditions.

To prove that a GPP $f$ satisfies these conditions, it is only really necessary to prove that $f(q)$ is central, non-nilpotent of minimal degree inside gr $\frac{kG}{P}$.\\

\noindent Ensuring the elements $f(q)$ are central and non-nilpotent in gr $\frac{kG}{P}$ can usually be done, provided $Q(\frac{kG}{P})$ is not a CSA -- a case we deal with separately in section 4 using a technique involving diagonalisation of the Mahler approximations. 

Ensuring these elements have minimal degree, however, is more of a problem when using the standard filtration on $kG$.\\

\noindent In section 5, we define a new filtration on $kG$ which ensures that we can find a GPP satisfying the required conditions, and we construct such a polynomial in section 6. This will allow us to complete our analysis and prove Theorem \ref{B}.\\

\noindent\textbf{\underline{Acknowledgements:}} I would like to thank Konstantin Ardakov for supervising me in my research, and for many helpful conversations and suggestions. I would also like to thank EPSRC for funding this research.

\section{Preliminaries}

\noindent Throughout, we will use the notation $(.,.)$ to denote the group commutator, i.e. $(g,h)=ghg^{-1}h^{-1}$.

\subsection{Abelian-by-procyclic groups}

Fix $G$ a compact $p$-adic Lie group carrying a $p$-valuation $\omega$. Since $G$ is compact, it follows that $G$ has finite rank $d\in\mathbb{N}$ by \cite[\rom{3} 2.2.6]{Lazard}, i.e. there exists a finite \emph{ordered basis} $\underline{g}=\{g_1,....,g_d\}$ such that for all $g\in G$, $g=g_1^{\alpha_1}....g_d^{\alpha_d}$ for some $\alpha_i\in\mathbb{Z}_p$ unique, and $\omega(g)=\min\{v_p(\alpha_i)+\omega(g_i):i=1,....,d\}$.\\

\noindent Hence $G$ is homeomorphic to $\mathbb{Z}_p^d$, and furthermore, it follows from \cite[\rom{1} 2.3.17]{Lazard} that $kG\cong k[[b_1,\cdots,b_d]]$ as $k$-vector spaces, where $b_i=g_i-1$ for each $i$.\\

\noindent Recall the definition of a $p$-saturated group \cite[\rom{3}.2.1.6]{Lazard}, and recall that the category of $p$-saturated groups is isomorphic to the category of saturated $\mathbb{Z}_p$-Lie algebras via the $\log$ and $\exp$ functors \cite[\rom{4}.3.2.6]{Lazard}. 

This means that for $G$ a $p$-saturated group, $\log(G)$ is a free $\mathbb{Z}_p$-Lie subalgebra of the $\mathbb{Q}_p$-Lie algebra $\mathfrak{L}(G)$ of $G$.\\

\noindent Also, recall that any $p$-valuable group $G$ can be embedded as an open subgroup into a $p$-saturable group $Sat(G)$, and hence $Sat(G)^{p^t}\subseteq G$ for some $t\in\mathbb{N}$ -- see \cite[\rom{3}.3.3.1]{Lazard} for full details.

\begin{definition}

A compact $p$-adic Lie group $G$ is \emph{abelian-by-procyclic} if there exists $H\trianglelefteq_c^{i} G$ such that $H$ is abelian and $\frac{G}{H}\cong\mathbb{Z}_p$. For $G$ non-abelian, $H$ is unique, and we call it the \emph{principal subgroup} of $G$; it follows easily that $Z(G)\subseteq H$ and $(G,G)\subseteq H$, so $G$ is metabelian.

\end{definition}

\noindent If $G$ is non-abelian, abelian-by-procyclic with principal subgroup $H$, then $G\cong H\rtimes\mathbb{Z}_p$. So if we assume that $G$ is $p$-valuable, we have that $G\cong\mathbb{Z}_p^d\rtimes\mathbb{Z}_p$.\\

\noindent Moreover, we can choose an element $X\in G\backslash H$ such that for every $g\in G$, $g=hX^{\alpha}$ for some unique $h\in H$, $\alpha\in\mathbb{Z}_p$. We call $X$ a \emph{procyclic element} in $G$.

\begin{lemma}\label{Lie-invariant}

Let $G$ be a non-abelian $p$-valuable abelian-by-procyclic group with principal subgroup $H$, and let $C$ be a closed subgroup of $G$. Then $C$ is abelian-by-procyclic.

Furthermore, if $N$ is a closed, isolated normal subgroup of $G$, then $\frac{G}{N}$ is $p$-valuable, abelian-by-procyclic.

\end{lemma}

\begin{proof}

Let $H':=C\cap H$, then clearly $H'$ is a closed, isolated, abelian normal subgroup of $C$.\\

\noindent Now, $\frac{C}{H'}=\frac{C}{H\cap C}\cong\frac{CH}{H}\leq\frac{G}{H}\cong\mathbb{Z}_p$, hence $\frac{C}{H'}$ is isomorphic to a closed subgroup of $\mathbb{Z}_p$, i.e. it is isomorphic to either 0 or $\mathbb{Z}_p$.\\

\noindent If $\frac{C}{H'}=0$ then $C=H'=H\cap C$ so $C\subseteq H$. Hence $C$ is abelian, and therefore abelian-by-procyclic.\\

\noindent If $\frac{C}{H'}\cong\mathbb{Z}_p$ then $C$ is abelian-by-procyclic by definition.\\

\noindent For $N\trianglelefteq_o^i G$, it follows from \cite[\rom{4}.3.4.2]{Lazard} that $\frac{G}{N}$ is $p$-valuable, and clearly $\frac{HN}{N}$ is a closed, abelian normal subgroup of $G$.\\

\noindent Consider the natural surjection of groups $\mathbb{Z}_p\cong\frac{G}{H}\twoheadrightarrow{}\frac{G}{HN}\cong\frac{G/N}{HN/N}$:\\ 

\noindent If the kernel of this map is zero, then it is an isomorphism, so $\frac{G/N}{HN/N}\cong\mathbb{Z}_p$, so $\frac{G}{N}$ is abelian-by-procyclic by definition.\\

\noindent If the kernel is non-zero, then $\frac{G}{HN}$ is a non-trivial quotient of $\mathbb{Z}_p$, and hence is finite, giving that $HN$ is open in $G$. So $\frac{HN}{N}$ is abelian and open in $\frac{G}{N}$, and it follows that $\frac{G}{N}$ is abelian, and hence abelian-by-procyclic.\end{proof}

\vspace{0.15in}

\noindent Now, fix $G$ a non-abelian $p$-valuable, abelian-by-procyclic group with principal subgroup $H$, procyclic element $X$. Let $\varphi\in Inn(G)$ be conjugation by $X$, then it is clear that $\varphi\neq id$.\\

\noindent Recall from \cite[Definition 4.8, Proposition 4.9]{nilpotent} the definition of $z(\varphi^{p^m}):H\to Sat(H)$, $h\mapsto \lim_{n\rightarrow\infty}{(\varphi^{p^{m+n}}(h)h^{-1})^{p^{-n}}}$.\\

\noindent Using \cite[\rom{4}. 3.2.3]{Lazard} it follows that $z(\varphi^{p^m})(h)=\exp([p^m\log(X),\log(h)])$, and hence $z(\varphi^{p^m})(h)=z(\varphi)(h)^{p^m}$ for all $m$.\\

\noindent For each $m\in\mathbb{N}$, define $u_m:=z(\varphi^{p^m})$, then it is clear that $u_m(h)=u_0(h)^{p^m}$ for all $g\in G$. So fix $m_1\in\mathbb{N}$ such that $u_0(h)^{p^{m_1}}\in H$ for all $g\in G$, and let $u:=u_{m_1}:H\to H$.\\

\noindent We call $m_1$ the \emph{initial power}, we may choose this to be as high as we need.

\subsection{Filtrations}

Let $R$ be any ring, we assume that all rings are unital.

\begin{definition}\label{filt}

A \emph{filtration} of $R$ is a map $w:R\to\mathbb{Z}\cup\{\infty\}$ such that for all $x,y\in R$, $w(x+y)\geq\min\{w(x),w(y)\}$, $w(xy)\geq w(x)+w(y)$, $w(1)=0$, $w(0)=\infty$. The filtration is \emph{separated} if $w(x)=\infty$ implies that $x=0$ for all $x\in R$.

\end{definition}

\noindent If $R$ carries a filtration $w$, then there is an induced topology on $R$ with the subgroups $F_nR:=\{r\in R:w(r)\geq n\}$ forming a basis for the neighbourhoods of the identity. This topology is Hausdorff if and only if the filtration is separated.\\

\noindent Recall from \cite[Ch.\rom{2} Definition 2.1.1]{LVO} that a filtration is \emph{Zariskian} if $F_1R\subseteq J(F_0R)$ and the \emph{Rees ring} $\widetilde{R}:=\underset{n\in\mathbb{Z}}{\oplus}{F_nR}$ is Noetherian.\\

\noindent Zariskian filtrations can only be defined on Noetherian rings, and by \cite[Ch.\rom{2} Theorem 2.1.2]{LVO} we see that a Zariskian filtration is separated.\\

\noindent If $R$ carries a filtration $w$, then define the \emph{associated graded ring} of $R$ to be 

\begin{center}
gr $R:=\underset{n\in\mathbb{Z}}{\oplus}{\frac{F_nR}{F_{n+1}R}}$
\end{center}

\noindent This is a graded ring with multiplication given by $(r+F_{n+1}R)\cdot (s+F_{m+1}R)=(rs+F_{n+m+1}R)$.\\

\noindent\textbf{\underline{Notation:}} For $r\in R$ with $w(r)=n$ we denote by gr$(r):=r+F_{n+1}R\in$ gr $R$.\\

\noindent We say that a filtration $w$ is \emph{positive} if $w(r)\geq 0$ for all $r\in R$.

\begin{definition}\label{valuation}

If $R$ carries a filtration $v$ and $x\in R\backslash\{0\}$, we say that $x$ is \emph{$v$-regular} if $v(xy)=v(x)+v(y)$ for all $y\in R$, i.e. gr$(x)$ is not a zero divisor in gr $R$. If all non-zero $x$ are $v$-regular we say that $v$ is a \emph{valuation}.

\end{definition}

\noindent Note that if $x$ is $v$-regular and a unit then $x^{-1}$ is $v$-regular and $v(x^{-1})=-v(x)$.\\

\noindent\textbf{\underline{Examples:}} 1. If $I$ is an ideal of $R$, the \emph{$I$-adic} filtration on $R$ is given by $F_0R=R$ and $F_nR=I^n$ for all $n>0$. If $I=\pi R$ for some normal element $\pi\in R$, we call this the \emph{$\pi$-adic filtration}.\\

\noindent 2. If $R$ carries a filtration $w$, then $w$ extends naturally to $M_k(R)$ via $w((a_{i,j}))=\min\{w(a_{i,j}):i,j=1,\cdots,k\}$ -- the \emph{standard matrix filtration}.\\

\noindent 3. If $R$ carries a filtration $w$ and $I\trianglelefteq R$, we define the \emph{quotient filtration} $\overline{w}:\frac{R}{I}\to\mathbb{Z}\cup \{\infty\}, r+I\mapsto\sup\{w(r+y):y\in I\}$. In this case, gr$_{\overline{w}}$ $\frac{R}{I}=\frac{gr R}{gr I}$.\\

\noindent 4. Let $(G,\omega)$ be a $p$-valuable group with ordered basis $\underline{g}=\{g_1,\cdots,g_d\}$. We say that $\omega$ is an \emph{abelian} $p$-valuation if $\omega(G)\subseteq\frac{1}{e}\mathbb{Z}$ for some $e\in\mathbb{Z}$, and $\omega((g,h))>\omega(g)+\omega(h)$ for all $g,h\in G$.

If $\omega$ is abelian, then by \cite[Corollary 6.2]{nilpotent}, $kG\cong k[[b_1,\cdots,b_d]]$ carries an associated Zariskian filtration given by 

\begin{center}
$w(\underset{\alpha\in\mathbb{N}^d}{\sum}{\lambda_{\alpha}b_1^{\alpha_1}\cdots b_d^{\alpha_d}})=\inf\{\sum_{i=1}^{d}{e\alpha_i\omega(g_i)}:\lambda_{\alpha}\neq 0\}\in\mathbb{Z}\cup\{\infty\}$.
\end{center}

\noindent We call $w$ the \emph{Lazard filtration} associated with $w$. It induces the natural complete, Hausdorff topology on $kG$, and gr $kG\cong k[X_1,\cdots,X_d]$, where $X_i=$ gr$(b_i)=$ gr$(g_i-1)$. 

Since gr $kG$ is a domain, it follows that $w$ is a valuation.\\

\noindent Now, recall from \cite[Definition 1.5.8]{McConnell} the definition of a \emph{crossed product}, $R\ast F$, of a ring $R$ with a group $F$. That is $R\ast F$ is a ring containing $R$, free as an $R$-module with basis $\{\overline{g}:g\in F\}\subseteq R^{\times}$, where $\overline{1_F}=1_R$, such that $R\overline{g_1}=\overline{g_1}R$ and $(\overline{g_1}R)(\overline{g_2}R)=\overline{g_1g_2}R$ for all $g_1,g_2\in F$.\\

\noindent Also recall from \cite{Passman} that given a crossed product $S=R\ast F$, we can define the \emph{action} $\sigma:F\to Aut(R)$ and the \emph{twist} $\gamma:F\times F\to R^{\times}$ such that for all $g,g_1,g_2\in F$, $r\in R$:\\ 

\noindent $\overline{g}r=\sigma(g)(r)\overline{g}$ and $\overline{g_1}$ $\overline{g_2}=\gamma(g_1,g_2)\overline{g_1g_2}$.

\begin{proposition}\label{crossed product}

Let $R$ be a ring with a complete, positive, Zariskian valuation $w:R\to\mathbb{N}\cup\{\infty\}$, let $F$ be a finite group, and let $S=R\ast F$ be a crossed product with action $\sigma$ and twist $\gamma$. Suppose that $w(\sigma(g)(r))=w(r)$ for all $g\in F$, $r\in R$.\\

\noindent Then $w$ extends to a complete, positive, Zariskian filtration $w':S\to \mathbb{N}\cup\{\infty\}$ defined by $w'(\sum_{g\in F}{r_g \overline{g}})=\min\{w(r_g):g\in G\}$, and gr$_{w'}$ $S \cong ($gr$_w$ $R)\ast F$.

\end{proposition}

\begin{proof}

From the definition it is clear that $w'(s_1+s_2)\geq\min\{w'(s_1),w'(s_2)\}$. So to prove that $w$ defines a ring filtration, it remains to check that $w'(s_1s_2)\geq w'(s_1)+w'(s_2)$.\\

\noindent In fact, using the additive property, we only need to prove that $w'(r\overline{g}s\overline{h})\geq w'(r\overline{g})+w'(s\overline{h})$ for all $r,s\in R$, $g,h\in F$. We will in fact show that equality holds here:\\

\noindent $w'(r\overline{g}s\overline{h})=w'(r\sigma(g)(s)\overline{g} \overline{h})=w'(r\sigma(g)(s)\gamma(g,h)\overline{gh})=w(r\sigma(g)(s)\gamma(g,h))$\\

\noindent $=w(r)+w(\sigma(g)(s))+w(\gamma(g,h))$ (by the valuation property)\\

\noindent $=w(r)+w(s)$.\\

\noindent The last equality follows because $w(\sigma(g)(s))=w(s)$ by assumption, and since $R$ is positively filtered and $\gamma(g,h)$ is a unit in $R$, it must have value zero. Clearly $w(r)+w(s)=w'(r\overline{g})+w'(s\overline{h})$ so we are done.\\

\noindent Hence $w'$ is a well-defined ring filtration, clearly $w'(r)=w(r)$ for all $r\in R$, and $w'(\overline{g})=0$ for all $g\in F$. We can define $\theta:$ gr$_w$ $R\to $ gr$_{w'}$ $S, r+F_{n+1}R\mapsto r+F_{n+1}S$, which is a well defined, injective ring homomorphism.\\

\noindent Given $s\in S$, $s=\underset{g\in F}{\sum}{r_g \overline{g}}$, so let $A_s:=\{g\in F:w(r_g)=w'(s)\}$. Then: \\

gr$(s)=\underset{g\in A_s}{\sum}{r_g \overline{g}}+F_{w'(s)+1}S=\underset{g\in A_s}{\sum}{(r_g+F_{w'(s)+1}S)(\overline{g}+F_1S)}=\underset{g\in A_s}{\sum}\theta($gr$(r_g))$gr$(\overline{g})$.\\

\noindent Hence gr$_{w'}$ $S$ is finitely generated over $\theta($gr $R)$ by $\{$gr$(\overline{g}):g\in F\}$. This set forms a basis, hence gr$_{w'}$ $S$ is free over $\theta($gr $R)$, and it is clear that each gr$(\overline{g})$ is a unit in gr $S$, and they are in bijection with the elements of $F$.\\

\noindent Therefore gr $S$ is Noetherian, and clearly $R\ast F$ is complete with respect to $w'$. Hence $w'$ is Zariskian by \cite[Chapter \rom{2}, Theorem 2.1.2]{LVO}.\\

\noindent Finally, gr$(r\overline{g})$gr$(s\overline{h})=$gr$(r\overline{g}s\overline{h})$ since $w'(r\overline{g}s\overline{h})=w'(r\overline{g})+w'(s\overline{h})$, so it is readily checked that $($gr $R)$gr$(\overline{g})$gr$(\overline{h})=(($gr $R)($gr$(\overline{g}))(($gr $R)($gr$(\overline{h}))$, and clearly $($gr $R)($gr$(\overline{g}))=($gr$(\overline{g}))($gr $R)$.\\

\noindent Therefore gr$_{w'}$ $S=($gr$_w$ $R)\ast F$.\end{proof}

\noindent This result will be useful to us later, because for any $p$-valuable group $G$, $U\trianglelefteq_o G$, $kG\cong kU\ast\frac{G}{U}$.

\subsection{Mahler expansions}

\noindent We will now recap the notion of the Mahler expansion of an automorphism.\\

\noindent Let $G$ be a compact $p$-adic Lie group. Define $C^{\infty}=C^{\infty}(G,k)=\{f:G\to k:f$ locally constant$\}$, and for each $U\leq_o G$, define $C^{\infty U}:=\{f\in C^{\infty}:f$ constant on the cosets of $U\}$.\\

\noindent Clearly $C^{\infty}$ is a $k$-algebra, $C^{\infty U}$ is a subalgebra, and recall from \cite[Lemma 2.9]{controller} that there is an action $\gamma:C^{\infty}\to End_k(kG)$ of $C^{\infty}$ on $kG$. Also recall the following result (\cite[Proposition 2.8]{controller}):

\begin{proposition}\label{control}

Let $I$ be a right ideal of $kG$, $U$ an open subgroup of $kG$. Then $I$ is controlled by $U$ if and only if $I$ is a $C^{\infty U}$-submodule of $kG$, i.e. if and only if for all $g\in G$, $\gamma(\delta_{Ug})(I)\subseteq I$, where $\delta_{Ug}$ is the characteristic function of the coset $Ug$.

\end{proposition}

\noindent Now assume that $G$ is $p$-valuable, with $p$-valuation $\omega$, and let $\underline{g}=\{g_1,\cdots,g_d\}$ be an ordered basis for $(G,\omega)$. For each $\alpha\in\mathbb{N}^d$ define $i_{\underline{g}}^{(\alpha)}:G\to k,\underline{g}^{\beta}\mapsto\binom{\beta}{\alpha}$. Then $i_{\underline{g}}^{(\alpha)}\in C^{\infty}$, so let $\partial_{\underline{g}}^{(\alpha)}:=\gamma(i_{\underline{g}}^{(\alpha)})$ -- the \emph{$\alpha$-quantized divided power} with respect to $\underline{g}$.

Also, for each $i=1,\cdots,d$, we define $\partial_i:=\partial_{\underline{g}}^{(e_i)}$, where $e_i$ is the standard $i$'th basis vector, these are $k$-linear derivations of $kG$.

\begin{proposition}\label{Frattini}

Suppose that $I$ is a right ideal of $kG$ and $\partial_j(I)\subseteq I$ for some $j\in\{1,\cdots,d\}$. Then $I$ is controlled by a proper open subgroup of $G$.

\end{proposition}

\begin{proof}

Recall from \cite[Lemma 7.13]{nilpotent} that if $V$ is an open normal subgroup of $G$ with ordered basis 

\noindent $\{g_1,\cdots,g_{s-1},g_s^p,\cdots,g_r^p,g_{r+1},\cdots,g_d\}$ for some $1\leq s\leq r\leq d$, then for each $g\in G$, $\gamma(\delta_{Vg})$ can be expressed as a polynomial in $\partial_s,\cdots,\partial_r$.

So it follows from Proposition \ref{control} that if $\partial_i(I)\subseteq I$ for all $i=s,\cdots,r$, then $I$ is controlled by $V$.\\

\noindent Since we know that $\partial_j(I)\subseteq I$, it remains to show that we can find a proper, open normal subgroup $U$ of $G$ with ordered basis $\{g_1,\cdots,g_{j-1},g_j^p,g_{j+1},\cdots,g_d\}$, and it will follow that $U$ controls $I$.\\

\noindent\textbf{\underline{Notation:}} Given $d$ variables $x_1,\cdots,x_d$, we will write $\underline{x}$ to denote the set $\{x_1,\cdots,x_d\}$, $\underline{x}_{j,p}$ to denote the same set, but with $x_j$ replaced by $x_j^p$, and $\underline{x}_{j}$ to denote the set with $x_j$ removed altogether. We write $\underline{x}^{\alpha}$ to denote $x_1^{\alpha_1}\cdots x_d^{\alpha_d}$.\\

\noindent Let $U$ be the subgroup of $G$ generated topologically by the set $\underline{g}_{j,p}$. It is clear that this subgroup contains $G^p$, and hence it is open in $G$. Let us suppose, for contradiction, that it contains an element $u=\underline{g}^{\alpha}$ where $\alpha\in\mathbb{Z}_p^d$ and $p\nmid\alpha_j$.\\

\noindent Recall from \cite[Definition 1.8]{DDMS} that the \emph{Frattini subgroup} $\phi(G)$ of $G$ is defined as the intersection of all maximal open subgroups of $G$, and since $G$ is a pro-$p$ group, it follows from \cite[Proposition 1.13]{DDMS} that $\phi(G)$ contains $G^p$ and $[G,G]$.

Hence $\phi(G)$ is an open normal subgroup of $G$, and $\frac{G}{\phi(G)}$ is abelian.\\

\noindent Since $\underline{g}_{j,p}$ generates $U$, it is clear that $\underline{g\phi(G)}_{p,j}$ generates $\frac{U\phi(G)}{\phi(G)}$. Therefore $u\phi(G)=\underline{g\phi(G)}_{j,p}^{\beta}$ for some $\beta\in\mathbb{Z}_p^d$.

But we know that $u=\underline{g}^{\alpha}$, so $u\phi(G)=\underline{g\phi(G)}^{\alpha}$. So since $\frac{G}{\phi(G)}$ is abelian, it follows that $g_j^{\alpha_j-p\beta_j}\phi(G)=\underline{g\phi(G)}_j^{\gamma}$ for some $\gamma\in\mathbb{Z}_p^d$.\\

\noindent But since $p\nmid\alpha_j$, $\alpha_j-p\beta_j$ is a $p$-adic unit, and hence $g_j\phi(G)=\underline{g\phi(G)}_j^{\delta}$, where $\delta_i=\gamma_i(\alpha_i-p\beta_i)^{-1}$. Therefore $\frac{G}{\phi(G)}$ is generated by $\underline{g\phi(G)}_j=\frac{\underline{g}_j\phi(G)}{\phi(G)}$.\\

\noindent It follows from \cite[Proposition 1.9]{DDMS} that $G$ is generated by $\underline{g}_j$, which has size $d-1$, and this is a contradiction since the rank $d$ of $G$ is the minimal cardinality of a generating set.\\

\noindent Therefore, every $u\in U$ has the form $\underline{g}_{j,p}^{\beta}$ for some $\beta\in\mathbb{Z}_p^d$, i.e $\{g_1,\cdots,g_j^p,\cdots,g_d\}$ is an ordered basis for $U$.\\

\noindent Finally, $U$ is maximal, so it contains $\phi(G)\supseteq [G,G]$, and hence it is normal in $G$ as required.\end{proof}

\noindent Now, given $\varphi\in Inn(G)$, clearly $\varphi$ extends to a $k$-linear endomorphism of $kG$, and using Mahler's theorem, we can express $\varphi^{p^m}$ as:

\begin{equation}\label{Mahler1}
\varphi^{p^m}=\sum_{\alpha\in\mathbb{N^d}}{\big\langle \varphi^{p^m},\partial_{\underline{g}}^{(\alpha)}}\big\rangle\partial_{\underline{g}}^{(\alpha)}=id+c_{1,m}\partial_1+....+c_{d,m}\partial_d+....
\end{equation}

\noindent Where $\big\langle \varphi^{p^m},\partial_{\underline{g}}^{(\alpha)}\big\rangle\in kG$ is the $\alpha$-\emph{Mahler coefficient} of $\varphi^{p^m}$ (see \cite[Corollary 6.6]{nilpotent} for full details).\\

\noindent Given $P\in Spec^f(kG)$, our general approach is to choose $\varphi\neq id$, and use analysis of (\ref{Mahler1}) to obtain a sequence of endomorphisms preserving $P$, converging pointwise in $m$ to an expression involving only $\partial_1,\cdots,\partial_d$, which will imply a control theorem using Proposition \ref{Frattini}.

However, it is more convenient for us to reduce modulo $P$, whence we can pass to the ring of quotients $Q(\frac{kG}{P})$ and divide out by anything regular mod $P$.\\ 

\noindent So let $\tau:kG\to Q(\frac{kG}{P})$ be the natural map, then inside End$_k(kG,Q(\frac{kG}{P}))$ our expression becomes:

\begin{equation}\label{Mahler2}
\tau\varphi^{p^m}-\tau=\sum_{\alpha\in\mathbb{N^d}}{\tau(\big\langle \varphi^{p^m},\partial_{\underline{g}}^{(\alpha)}}\big\rangle)\tau\partial_{\underline{g}}^{(\alpha)}
\end{equation}

\noindent Since we want to analyse convergence of this expression as $m\rightarrow\infty$, we need to define a certain well behaved filtration $v$ on $Q(\frac{kG}{P})$, which we call a \emph{non-commutative valuation}. We will describe the construction of $v$ in the next section.\\

\noindent In our case, we take $G$ to be non-abelian, $p$-valuable, abelian-by-procyclic group, with principal subgroup $H$, procyclic element $X$, and we take $\varphi$ to be conjugation by $X$, which is a non-trivial inner automorphism.\\

\noindent It is clear that $\varphi|_H$ is trivial modulo centre since $H$ is abelian. Therefore it follows from the proof of \cite[Lemma 6.7]{nilpotent} that if $\underline{g}=\{h_1,\cdots,h_d,X\}$ is an ordered basis for $(G,\omega)$, where $\{h_1,\cdots,h_d\}$ is some ordered basis for $H$, then for any $m\in\mathbb{N}$ $\alpha\in\mathbb{N}^{d+1}$, we have:

\begin{center}
$\big\langle \varphi^{p^m},\partial_{\underline{g}}^{(\alpha)}\big\rangle=(\varphi^{p^m}(h_1)h_1^{-1}-1)^{\alpha_1}\cdots(\varphi^{p^m}(h_d)h_d^{-1}-1)^{\alpha_d}(\varphi^{p^m}(X)X^{-1}-1)^{\alpha_{d+1}}$
\end{center}

\noindent And since $\varphi(X)=X$, it is clear that this is 0 if $\alpha_{d+1}\neq 0$. Hence we may assume that $\alpha\in\mathbb{N}^d$, and:

\begin{center}
$\big\langle \varphi^{p^m},\partial_{\underline{g}}^{(\alpha)}\big\rangle=(\varphi^{p^m}(h_1)h_1^{-1}-1)^{\alpha_1}\cdots(\varphi^{p^m}(h_d)h_d^{-1}-1)^{\alpha_d}$
\end{center}

\noindent Recall our function $u:H\to H$ defined in section 2.1, we can now use $u$ to approximate the Mahler coefficients inside $Q(\frac{kG}{P})$. Setting $q_i=\tau(u(h_i)-1)$ for $i=1,...,d$, it follows from the proof of \cite[Proposition 7.7]{nilpotent} that we can derive the following expression from (\ref{Mahler2}):

\begin{equation}\label{Mahler3}
\tau\varphi^{p^m}-\tau=q_1^{p^m}\tau\partial_1+....+q_{d}^{p^m}\tau\partial_{d}+\sum_{\vert\alpha\vert\geq 2}{\underline{q}^{\alpha p^m}\tau\partial_{\underline{g}}^{(\alpha)}}+\varepsilon_m
\end{equation}

\noindent Where $\underline{q}^{\alpha p^m}=q_1^{\alpha_1}\cdots q_d^{\alpha_d}$, $\varepsilon_m=\underset{\alpha\in\mathbb{N}^d}{\sum}{(\big\langle \varphi^{p^m},\partial_{\underline{g}}^{(\alpha)}\big\rangle-\underline{q}^{\alpha p^m})\tau\partial_{\underline{g}}^{(\alpha)}}\in$ End$_k(kG,Q)$, and there exists $t\in\mathbb{N}$ such that $v(\varepsilon_m(r))>p^{2m-t}$ for all $r\in kG$.\\

\noindent Since $\varphi(P)=P$ it is clear that the left hand side of this expression annihilates $P$. So take any $y\in P$ and apply it to both sides of (\ref{Mahler3}) and we obtain:

\begin{equation}\label{Mahler}
0=q_1^{p^m}\tau\partial_1(y)+....+q_{d}^{p^m}\tau\partial_{d}(y)+O(q^{p^m})
\end{equation}

\noindent Where $q\in Q$ with $v(q^{p^m})\geq \underset{i\leq d}{\min}\{2v(q_i^{p^m})\}$ for all $m$.\\

\noindent Furthermore, let $f(x)=a_0x+a_1x^p+a_2x^{p^2}+\cdots+a_nx^{p^n}$ be a polynomial, where $a_i\in\tau(kH)$ for each $i$.\\

\noindent Then for each $m\in\mathbb{N}$, $i=0,\cdots,n$, consider expression (\ref{Mahler}) above, with $m$ replaced by $m+i$, and multiply by $a_i^{p^m}$ to obtain:

\begin{center}
$0=(a_iq_1^{p^i})^{p^m}\tau\partial_1(y)+....+(a_iq_{d}^{p^i})^{p^m}\tau\partial_{d}(y)+O((a_iq^{p^i})^{p^m})$
\end{center}

\noindent Sum all these expressions as $i$ ranges from $0$ to $n$ to obtain:

\begin{equation}\label{MahlerPol}
0=f(q_1)^{p^m}\tau\partial_1(y)+....+f(q_{d})^{p^m}\tau\partial_{d}(y)+O(f(q)^{p^m})
\end{equation}

\noindent In the next section, we will see how after ensuring certain conditions on $f$, we can use this expression to deduce a control theorem.

\section{Non-commutative Valuations}

In this section, we fix $Q$ a simple Artinian ring. First, recall from \cite{nilpotent} the definition of a non-commutative valuation.

\begin{definition}
A \emph{non-commutative valuation} on $Q$ is a Zariskian filtration $v:Q\to\mathbb{Z}\cup\{\infty\}$ such that if $\widehat{Q}$ is the completion of $Q$ with respect to $v$, then $\widehat{Q}\cong M_k(Q(D))$ for some complete non-commutative DVR $D$, and $v$ is induced by the natural $J(D)$-adic filtration.
\end{definition}

\noindent It follows from this definition that if $q\in Q$ with $v(q)\geq 0$, then $q$ is $v$-regular if and only if $q$ is normal in $F_0\widehat{Q}$.\\

\noindent We want to construct a non-commutative valuation $v$ on $Q(\frac{kG}{P})$ which we can use to analyse our Mahler expansion (\ref{MahlerPol}).

\subsection{Construction}

Recall that \cite[Theorem C]{nilpotent} gives that if $R$ is a prime ring carrying a Zariskian filtration such that gr $R$ is commutative, then we can construct a non-commutative valuation on $Q(R)$.\\

\noindent The main theorem of this section generalises this result.\\

\noindent Let $R$ be a prime ring with a positive Zariskian filtration $w:R\to\mathbb{N}$ $\cup$ $\{\infty\}$ such that gr$_w R$ is finitely generated over a central, graded, Noetherian subring $A$, and we will assume that the positive part $A_{>0}$ of $A$ is not nilpotent, and hence we may fix a minimal prime ideal $\mathfrak{q}$ of $A$ with $\mathfrak{q}\not\supseteq A_{> 0}$. Define:

\begin{center}
$T=\{X\in A\backslash\mathfrak{q}: X$ is homogeneous $\}$.
\end{center} 

\noindent Then $T$ is central, and hence localisable in gr $R$, and the left and right localisations agree.

\begin{lemma}\label{homogeneous localsiation}

Let $\mathfrak{q}':=T^{-1}\mathfrak{q}$, then $\mathfrak{q}'$ is a nilpotent ideal of $T^{-1}A$ and:\\

\noindent i. There exists $Z\in T$, homogeneous of positive degree, such that $\frac{T^{-1}A}{\mathfrak{q}'}\cong (\frac{T^{-1}A}{\mathfrak{q}'})_0[\bar{Z},\bar{Z}^{-1}]$, where $\bar{Z}:=Z+\mathfrak{q}$.\\

\noindent ii. The quotient $\frac{(T^{-1}A)_{\geq 0}}{Z(T^{-1}A)_{\geq 0}}$ is Artinian, and $T^{-1}A$ is gr-Artinian, i.e. every descending chain of graded ideals terminates.

\end{lemma}

\begin{proof}

Since $A$ is a graded, commutative, Noetherian ring, this is identical to the proof of \cite[Proposition 3.2]{nilpotent}.\end{proof}

\noindent Since gr $R$ is finitely generated over $A$, it follows that $T^{-1}$gr $R$ is finitely generated over $T^{-1}A$. So using this lemma, we see that $T^{-1}$gr $R$ is gr-Artinian.\\

\vspace{0.1in}

\noindent Let $S:=\{r\in R: $ gr$(r)\in T\}$, then since $w$ is Zariskian, $S$ is localisable by \cite[Corollary 2.2]{Ore sets}, and $S^{-1}R$ carries a Zariskian filtration $w'$ such that gr$_{w'}$ $S^{-1}R\cong T^{-1}$gr $R$, and if $r\in R$ then $w'(r)\geq w(r)$, and equality holds if $r\in S$.

Furthermore, $w'$ satisfies $w'(s^{-1}r)=w'(r)-w(s)$ for all $r\in $ gr $R$, $s\in S$.\\

\noindent Now, since $R$ is prime, the proof of \cite[Lemma 3.3]{nilpotent} shows that $S^{-1}R=Q(R)$, so let $Q'$ be the completion of $Q(R)$ with respect to $w'$.\\

\noindent Let $U:=F_0Q'$, which is Noetherian by \cite[Ch.\rom{2} Lemma 2.1.4]{LVO} then gr$_{w'}$ $U\cong (T^{-1}$gr $R)_{\geq 0}$, and since gr $Q'=T^{-1}$gr $R$ is gr-Artinian, $Q'$ is Artinian.

\begin{lemma}\label{microlocalisation}

There exists a regular, normal element $z\in J(U)\cap Q'^{\times}$ such that $\frac{U}{zU}$ has Krull dimension 1 on both sides, and for all $n\in\mathbb{Z}$, $F_{nw'(z)}Q'=z^nU$, hence the $z$-adic filtration on $Q'$ is topologically equivalent to $w'$.

\end{lemma}

\begin{proof}

Recall the element $Z\in T^{-1}A$ from Lemma \ref{homogeneous localsiation}($i$), then we can choose an element $z\in U$ such that gr$_{w'}(z)=Z$. Since $Z$ is a unit in $T^{-1}A$, we can in fact choose $z$ to be regular and normal in $U$. Then since $w'$ is Zariskian and $Z$ has positive degree, $z\in F_1Q'\subseteq J(U)$.\\
 
\noindent Furthermore, since $\bar{Z}=Z+\mathfrak{q}'$ is a unit in $\frac{T^{-1}A}{\mathfrak{q}'}$ and $\mathfrak{q}'$ is nilpotent, it follows that $Z$ is a unit in $T^{-1}A$, and hence in $T^{-1}$gr $R=$ gr $Q'$. 

So it follows that $z$ is not a zero divisor in $Q'$, and hence is a unit since $Q'$ is artinian. Also, for all $u\in Q'$, $w'(zuz^{-1})=w'(u)$ since $Z$ is central in gr $Q'$, hence $z$ is normal in $U$.\\ 

\noindent Since $(T^{-1}$gr $R)_{\geq 0}$ is finitely generated over $(T^{-1}A)_{\geq 0}$, it follows that $\frac{\text{gr }U}{Z \text{gr }U}$ is finitely generated over the image of $\frac{(T^{-1}A)_{\geq 0}}{Z(T^{-1}A)_{\geq 0}}\to \frac{\text{gr }U}{Z \text{gr }U}$.

This image is gr-Artinian by Lemma \ref{homogeneous localsiation}($ii$) and hence $\frac{\text{gr }U}{Z \text{gr }U}$ it is also gr-Artinian.\\

\noindent Therefore $\frac{U}{zU}$ is Artinian, and the proof of \cite[Proposition 3.4]{nilpotent} gives us that $U$ has Krull dimension at most 1 on both sides, and that $F_{nw'(z)} Q'=z^{n} U$ for all $n\in\mathbb{Z}$.\end{proof}

\vspace{0.1in}

\noindent So, after passing to a simple quotient $\widehat{Q}$ of $Q'$, and letting $V:=\widehat{Q}_{\geq 0}$ be the image of $U$ in $\widehat{Q}$, then since $Q(R)$ is simple, it follows that the map $Q(R)\to\widehat{Q}$ is injective, and the image is dense with respect to the quotient filtration.\\

\noindent Now, choose a maximal order $\mathcal{O}$ in $\widehat{Q}$, which is equivalent to $V$ in the sense of \cite[Definition 1.9]{McConnell}. Such an order exists by \cite[Theorem 3.11]{nilpotent}, and it is Noetherian.

Furthermore, let $z\in J(U)$ be the regular, normal element from Lemma \ref{microlocalisation}, and let $\overline{z}\in J(V)$ be the image of $z$ in $V$, then $\mathcal{O}\subseteq \overline{z}^{-r} V$ for some $r\in\mathbb{N}$ by \cite[Proposition 3.7]{nilpotent}.\\

\noindent It follows from \cite[Theorem 3.6]{nilpotent} that $\mathcal{O}\cong M_n(D)$ for some complete non-commutative DVR $D$, and hence $\widehat{Q}\cong M_n(Q(D))$. So let $v$ be the $J(\mathcal{O})$-adic filtration, i.e. the filtration induced from the valuation on $D$. Then $v$ is topologically equivalent to the $\overline{z}$-adic filtration on $\widehat{Q}$.\\

\noindent It is clear from the definition that the restriction of $v$ to $Q(R)$ is a non-commutative valuation, and the proof of \cite[Theorem C]{nilpotent} shows that $(R,w)\to (Q(R),v)$ is continuous.\\

\noindent Note that our construction depends on a choice of minimal prime ideal $\mathfrak{q}$ of $A$. So altogether, we have proved the following:

\begin{theorem}\label{filtration}

Let $R$ be a prime ring with a Zariskian filtration $w:R\to\mathbb{N}\cup\{\infty\}$ such that gr$_w$ $R$ is finitely generated over a central, graded, Noetherian subring $A$, and the positive part $A_{> 0}$ of $A$ is not nilpotent.\\

\noindent Then for every minimal prime ideal $\mathfrak{q}$ of $A$ with $q\not\supseteq A_{> 0}$, there exists a corresponding non-commutative valuation $v_{\mathfrak{q}}$ on $Q(R)$ such that the inclusion $(R,w)\to (Q(R),v_{\mathfrak{q}})$ is continuous.

\end{theorem}

\noindent In particular, if $P$ is a prime ideal of $kG$, then $R=\frac{kG}{P}$ carries a natural Zariskian filtration, given by the quotient of the Lazard filtration on $kG$, and gr $R\cong\frac{\text{gr }kG}{\text{gr }P}$ is commutative, and if $P\neq J(kG)$ then $($gr $R)_{\geq 0}$ is not nilpotent by \cite[Lemma 7.2]{nilpotent}. 

Hence we may apply Theorem \ref{filtration} to obtain a non-commutative valuation $v$ on $Q(\frac{kG}{P})$ such that the natural map $\tau:(kG,w)\to(Q(\frac{kG}{P}),v)$ is continuous.

\subsection{Properties}

We will now explore some important properties of the non-commutative valuation $v_{\mathfrak{q}}$ on $Q(R)$ that we have constructed.\\

\noindent So again, we have that gr $R$ is finitely generated over $A$, and $\mathfrak{q}$ is a minimal prime ideal of $A$, not containing $A_{>0}$. Recall first the data that we used in the construction of $v_{\mathfrak{q}}$:

\begin{itemize}

\item $w'$ -- a Zariskian filtration on $Q(R)$ such that $w'(r)\geq w(r)$ for all $r\in R$, with equality if gr$_w(r)\in A\backslash\mathfrak{q}$. Moreover, if gr$_w(r)\in A\backslash\mathfrak{q}$ then $r$ is $w'$-regular.

\item $Q'$ -- the completion of $Q(R)$ with respect to $w'$.

\item $U$ -- the positive part of $Q'$, a Noetherian ring.

\item $z$ -- a regular, normal element of $J(U)$ such that $z^nU=F_{nw'(z)}Q'$ for all $n\in\mathbb{Z}$.

\item $v_{z,U}$ -- the $z$-adic filtration on $Q'$, topologically equivalent to $w'$.

\item $\widehat{Q}$ -- a simple quotient of $Q'$.

\item $V$ -- the positive part of $\widehat{Q}$, which is the image of $U$ in $\widehat{Q}$.

\item $\overline{z}$ -- the image of $z$ in $V$.

\item $v_{\overline{z},V}$ -- the $\overline{z}$-adic filtration on $\widehat{Q}$, topologically equivalent to the quotient filtration.

\item $\mathcal{O}$ -- a maximal order in $\widehat{Q}$, equivalent to $V$, satisfying $\mathcal{O}\subseteq \overline{z}^{-r}V$ for some $r\geq 0$.

\item $v_{\overline{z},\mathcal{O}}$ -- the $\overline{z}$-adic filtration on $\mathcal{O}$.

\item $v_{\mathfrak{q}}$ -- the $J(\mathcal{O})$-adic filtration on $\widehat{Q}$, topologically equivalent to $v_{\overline{z},\mathcal{O}}$.

\end{itemize}

\noindent From now on, we will assume further that $R$ is an $\mathbb{F}_p$-algebra.

\begin{lemma}\label{normal}

Given $r\in R$ such that gr$(r)\in A\backslash\mathfrak{q}$, we have:\\

\noindent i. $r$ is normal in $U$, a unit in $Q'$ and for any $u\in U$, $w'(rur^{-1}-u)>w'(u)$.\\

\noindent ii $v_{\overline{z},V}(r)= v_{z,U}(r)$.

\end{lemma}

\begin{proof}

$i$. Since $r\in S=\{s\in R:$ gr$(s)\in A\backslash\mathfrak{q}\}$ and $Q(R)=S^{-1}R$, $r$ is a unit in $Q'$, and we know that $w'(r)=w(r)$. Given $u\in U$, we want to prove that $rur^{-1}\in U$, thus showing that $r$ is normal in $U$.\\ 

\noindent We know that $U=F_0Q'$ is the completion of the positive part $F_0Q(R)$ of $Q(R)$ by definition, and we may assume that $u$ lies in $Q(R)$, i.e. $u=s^{-1}t$ for some $s\in S$, $t\in R$, and $w'(u)=w'(t)-w(s)\geq 0$.\\

\noindent But gr$(r),$ gr$(s)\notin\mathfrak{q}$, and hence gr$(r)$gr$(s)\neq 0$, which means that $w(rs)=w(r)+w(s)$. Therefore $w'(r^{-1}ur)=w'((sr)^{-1}tr)=w'(tr)-w(rs)\geq w'(t)+w'(r)-w(r)-w(s)=w'(t)-w'(s)\geq 0$, and so $r^{-1}ur\in U$ as required.\\

\noindent Furthermore, since gr$(r)\in A$ is central in $gr$ $R$, $w'(ru-ur)>w'(u)+w'(r)$, and thus $w'(rur^{-1}-u)=w'((ru-ur)r^{-1})\geq w'(ru-ur)-w'(r)>w'(u)+w'(r)-w'(r)=w'(u)$.\\

\noindent $ii$. Let $t:=v_{z,U}(r)$.\\

\noindent So $r\in z^tU\backslash z^{t+1}U=F_{tw'(z)}Q'\backslash F_{(t+1)w'(z)}Q'$, and hence $w'(r)=tw'(z)+j$ for some $0\leq j<w'(z)$.

Since gr$_{w'}(r)\in A\backslash\mathfrak{q}$, we have that $w'(r^{-1})=-w'(r)=-tw'(z)-j$.\\

\noindent Let $\overline{r}$ be the image of $r$ in $\widehat{Q}$. Then since $r\in z^tU$, it is clear that $\overline{r}\in\overline{z}^tV$, hence $v_{\overline{z},V}(r)\geq t$, so it remains to prove that $v_{\overline{z},V}(r)\leq t$.\\

\noindent Suppose that $\overline{r}\in\overline{z}^{t+1}V$, i.e. $r-z^{t+1}u$ maps to zero in $\widehat{Q}$ for some $u\in U$, and hence $z^{-t}r-zu=z^{-t}(r-z^{t+1}u)$ also maps to zero.\\

\noindent Let $a=z^{-t}r$, $b=-zu$. Then $w'(b)\geq w'(u)+w'(z)\geq w'(z)$, $w'(a^{-1})=w'(r^{-1})+tw'(z)=-tw'(z)-j+tw'(z)=-j$, so $w'(a^{-1}b)\geq w'(z)-j>w'(z)-w'(z)=0$, and therefore $(a^{-1}b)^n\rightarrow 0$ as $n\rightarrow\infty$.\\

\noindent So by completeness of $Q'$, the series $\underset{n\geq 0}{\sum}{(-1)^{n}(a^{-1}b)^na^{-1}}$ converges in $Q'$, and the limit is the inverse of $a+b$, hence $a+b=z^{-t}r-zu$ is a unit in $Q'$.\\

\noindent Therefore a unit in $Q'$ maps to zero in $\widehat{Q}$ -- contradiction.\\

\noindent Hence $\overline{r}\notin\overline{z}^{t+1}V$, so $v_{\overline{z},V}(r)\leq t$ as required.\end{proof}

\begin{proposition}\label{regular}

Let $u\in U$ be regular and normal, then $u$ is a unit in $Q'$. Furthermore, if $w'(uau^{-1}-a)>w'(a)$ for all $a\in Q'$, then setting $\overline{u}$ as the image of $u$ in $V$, we have that for sufficiently high $m\in\mathbb{N}$, $\overline{u}^{p^m}$ is $v_{\mathfrak{q}}$-regular.

\end{proposition}

\begin{proof}

Since $u$ is regular in $U$, it is not a zero divisor, so it follows that $u$ is not a zero divisor in $Q'$, and hence a unit since $Q'$ is artinian.\\

\noindent Since $u$ is normal in $U$, i.e. $uU=Uu$, it follows that $\overline{u}V=V\overline{u}$, so $\overline{u}$ is normal in $V$. We want to prove that for $m$ sufficiently high, $\overline{u}^{p^m}$ is normal in $\mathcal{O}=F_0\widehat{Q}$, and it will follow that it is $v_{\mathfrak{q}}$-regular.\\

\noindent We know that $w'(uau^{-1}-a)>w'(a)$ for all $a\in Q'$, so let $\theta:Q'\to Q'$ be the conjugation action of $u$, then $(\theta-id)(F_nQ')\subseteq F_{n+1}Q'$ for all $n\in\mathbb{Z}$.

Therefore, for all $k\in\mathbb{N}$, $(\theta-id)^k(F_nQ')\subseteq F_{n+k}Q'$.\\ 

\noindent Since $Q'$ is an $\mathbb{F}_p$-algebra, it follows that $(\theta^{p^m}-id)(F_nQ')=(\theta-id)^{p^m}(F_nQ')\subseteq F_{n+p^m}Q'$, and clearly $\theta^{p^m}$ is conjugation by $\overline{u}^{p^m}$.\\

\noindent So fix $k\in\mathbb{N}$ such that $p^k\geq w'(z)$. Then we know that $z^nU=F_{nw'(z)}U$, so $(\theta^{p^k}-id)(z^nU)\subseteq F_{nw'(z)+p^k}U\subseteq F_{(n+1)w'(z)}U=z^{n+1}U$.\\

\noindent Hence we have that for all $a\in Q'$, $v_{z,U}(u^{p^k}au^{-p^k}-a)>v_{z,U}(a)$, and it follows immediately that $v_{\overline{z},V}(\overline{u}^{p^k}a\overline{u}^{-p^k}-a)>v_{\overline{z},V}(a)$ for all $a\in\widehat{Q}$.\\

\noindent For convenience, let $v:=\overline{u}^{p^k}\in V$. We know that $v_{\overline{z},V}(vav^{-1}-a)>v_{\overline{z},V}(a)$ for all $a\in\widehat{Q}$, and we want to prove that $v^{p^m}$ is normal in $\mathcal{O}$ for $m$ sufficently high.\\

\noindent Let $I=\{v\in V:qv\in V$ for all $q\in \mathcal{O}\}$, then $I$ is a two-sided ideal of $V$, and since $\mathcal{O}\subseteq \overline{z}^{-r}V$, we have that $\overline{z}^rV\subseteq I$.\\

\noindent Let $\psi:\widehat{Q}\to\widehat{Q}$ be conjugation by $v$, then we know that $(\psi-id)(\overline{z}^nV)\subseteq \overline{z}^{n+1}V$, and hence $(\psi-id)^s(V)\subseteq \overline{z}^sV$ for all $s$.

Choose $m\in\mathbb{N}$ such that $p^m\geq r$, then $(\psi^{p^m}-id)(V)=(\psi-id)^{p^m}(V)\subseteq \overline{z}^{p^m}V\subseteq \overline{z}^rV\subseteq I$.\\

\noindent Therefore, for all $a\in V$, $v^{p^m}av^{-p^m}-a\in I$, and in particular, for all $a\in I$, $v^{p^m}av^{-p^m}\in I$, so $v^{p^m}Iv^{-p^m}\subseteq I$.

So set $b:=v^{p^m}$, then it follows from Noetherianity of $V$ that $bIb^{-1}=I$.\\

\noindent Finally, consider the subring $\mathcal{O}':=b^{-1}\mathcal{O}b$ of $\widehat{Q}$ containing $V$, then since $\mathcal{O}$ is a maximal order equivalent to $V$, it follows immediately that $\mathcal{O}'$ is equivalent to $V$, and that $\mathcal{O}'$ is also maximal.\\

\noindent Given $c\in \mathcal{O}'$, $c=b^{-1}qb$ for some $q\in \mathcal{O}$. So given $x\in I$, $cx=b^{-1}qbx=b^{-1}qbxb^{-1}b\in b^{-1}Ib=I$, so $c\in\mathcal{O}_l(I)=\{q\in\widehat{Q}:qI\subseteq I\}$, and hence $\mathcal{O}'\subseteq \mathcal{O}_l(I)$.\\

\noindent But $\mathcal{O}_l(I)$ is a maximal order in $\widehat{Q}$, equivalent to $V$ by \cite[Lemma 1.12]{McConnell}, and this order contains $\mathcal{O}$ by the definition of $I$.

So since $\mathcal{O}$ and $\mathcal{O}'$ are maximal orders and are both contained in $\mathcal{O}_l(I)$, it follows that $\mathcal{O}_l(I)=\mathcal{O}=\mathcal{O}'=b\mathcal{O}b^{-1}$.\\

\noindent Therefore  $b=v^{p^m}=\overline{u}^{p^{m+k}}$ is normal in $\mathcal{O}$ as required.\end{proof}

\noindent In particular, it is clear that $z\in U$ satisfies the property that $w'(zaz^{-1}-a)>w'(a)$ for all $a\in Q'$, thus $\overline{z}^{p^m}$ is normal in $\mathcal{O}$ for large $m$.\\

\noindent The next result will be very useful to us later when we want to compare values of elements in $Q(\frac{kG}{P})$ based on their values in $kG$.

\begin{theorem}\label{comparison}

Given $r\in R$ such that gr$_w(r)\in A\backslash\mathfrak{q}$, there exists $m\in\mathbb{N}$ such that $r^{p^m}$ is $v_{\mathfrak{q}}$-regular inside $\widehat{Q}$. Also, if $s\in R$ with $w(s)>w(r)$ then for sufficiently high $m$, $v_{\mathfrak{q}}(s^{p^m})>v_{\mathfrak{q}}(r^{p^m})$.

Moreover, if $w(s)=w(r)$ and gr$_w(s)\in\mathfrak{q}$ then we also have that $v_{\mathfrak{q}}(s^{p^m})>v_{\mathfrak{q}}(r^{p^m})$ for sufficiently high $m$.

\end{theorem}

\begin{proof}

Since gr$_w(r)\in A\backslash\mathfrak{q}$, it follows from Lemma \ref{normal}($i$) that $r$ is normal and regular in $U$, and $w'(rur^{-1}-u)>w'(u)$ for all $u\in U$. So for $m\in\mathbb{N}$ sufficiently high, $r^{p^m}$ is $v_{\mathfrak{q}}$-regular by Proposition \ref{regular}.\\

\noindent Note that since gr$_w(r)\in A\backslash\mathfrak{q}$, we have that $w'(r)=w(r)$. In fact, since gr$_w(r)$ is not nilpotent, we actually have that $w'(r^n)=w(r^n)=nw(r)$ for all $n\in\mathbb{N}$. So if $w(s)>w(r)$, then for any $n$, $w'(s^n)\geq nw(s)>nw(r)=w'(r^n)$.

Moreover, if $w(s)=w(r)$ and gr$_w(s)\in\mathfrak{q}$, then since $\mathfrak{q}'=T^{-1}\mathfrak{q}$ is nilpotent by Lemma \ref{homogeneous localsiation}, it follows that for $n$ sufficiently high, $w'(s^n)>nw'(s)$, and hence $w'(s^n)> nw(s)=nw(r)=w(r^n)=w'(r^n)$.\\

\noindent So, in either case, after replacing $r$ and $s$ by high $p$'th powers of $r$ and $s$ if necessary, we may assume that $w'(s)>w'(r)$, i.e. $w'(s)\geq w'(r)+1$.\\

\noindent It follows that for every $K>0$, we can find $m\in\mathbb{N}$ such that $w'(s^{p^m})\geq w'(r^{p^m})+K$. First we will prove the same result for $v_{z,U}$:\\

\noindent Given $K>0$, let $N=w'(z)(K+1)$, so that $K=\frac{1}{w'(z)}N-1$, then choose $m$ such that $w'(s^{p^m})\geq w'(r^{p^m})+N$, and let $l:=v_{z,U}(s^{p^m})$, $t:=v_{z,U}(r^{p^m})$.\\

\noindent So $s^{p^m}\in z^lU\backslash z^{l+1}U=F_{lw'(z)}Q'\backslash F_{(l+1)w'(z)}Q'$, and $r^{p^m}\in z^tU\backslash z^{t+1}U=F_{lw'(z)}Q'\backslash F_{(l+1)w'(z)}Q'$.\\

\noindent Hence $(l+1)w'(z)\geq w'(s^{p^m})\geq lw'(z)$ and $(t+1)w'(z)\geq w'(r^{p^m})\geq tw'(z)$.\\

\noindent Therefore, $v_{z,U}(s^{p^m})=l=\frac{1}{w'(z)}((l+1)w'(z))-1\geq\frac{1}{w'(z)}w'(s^{p^m})-1$\\

\noindent $\geq\frac{1}{w'(z)}(w'(r^{p^m})+N)-1\geq\frac{1}{w'(z)}(tw'(z)+N)-1=t+\frac{1}{w'(z)}N-1=t+K=v_{z,U}(r^{p^m})+K$ as required.\\

\noindent Now, since gr$_w(r)\in A\backslash\mathfrak{q}$, we have that $v_{\overline{z},V}(r^{p^m})=v_{z,U}(r^{p^m})$ for all $m$ by Lemma \ref{normal}($ii$).\\

\noindent Therefore, since $v_{\overline{z},V}(s^{p^m})\geq v_{z,U}(s^{p^m})$ for all $m$, it follows that for every $K>0$, there exists $m\in\mathbb{N}$ such that $v_{\overline{z},V}(s^{p^m})\geq v_{\overline{z},V}(r^{p^m})+K$.\\

\noindent Now we will consider $v_{\overline{z},\mathcal{O}}$, the $\overline{z}$-adic filtration on $\widehat{Q}$.\\ 

\noindent Recall that $V\subseteq\mathcal{O}\subseteq\overline{z}^{-r}V$, and thus $z^nV\subseteq z^n\mathcal{O}\subseteq z^{n-r}V$ for all $n$. Hence 

\noindent $v_{\overline{z},V}(v)-r\leq v_{\overline{z},\mathcal{O}}(v)\leq v_{\overline{z},V}(v)$ for all $v\in V$.\\

\noindent For any $K>0$, choose $m$ such that $v_{\overline{z},V}(s^{p^m})\geq v_{\overline{z},V}(r^{p^m})+K+r$. Then:\\

$v_{\overline{z},\mathcal{O}}(s^{p^m})\geq v_{\overline{z},V}(s^{p^m})-r\geq v_{\overline{z},V}(r^{p^m})+K+r-r\geq v_{\overline{z},\mathcal{O}}(r^{p^m})+K$.\\

\noindent Now, using Proposition \ref{regular}, we know that we can find $k\in\mathbb{N}$ such that $x:=\overline{z}^{p^k}$ is normal in $\mathcal{O}$, i.e. $x\mathcal{O}=\mathcal{O}x$ is a two-sided ideal of $\mathcal{O}$.

Then since $\mathcal{O}\cong M_n(D)$ for some non-commutative DVR $D$, it follows that $x\mathcal{O}=J(\mathcal{O})^a$ for some $a\in\mathbb{N}$, and $x^m\mathcal{O}=J(\mathcal{O})^{am}$ for all $m$.\\

\noindent So, choose $m\in\mathbb{N}$ such that $r^{p^m}$ is $v_{\mathfrak{q}}$-regular, $v_{\mathfrak{q}}(r^{p^m})\geq a$ and $v_{\overline{z},\mathcal{O}}(s^{p^m})\geq v_{\overline{z},\mathcal{O}}(r^{p^m})+p^k$.\\

\noindent Then suppose that $v_{\mathfrak{q}}(r^{p^m})=n$, i.e. $r^{p^m}\in J(\mathcal{O})^n\backslash J(\mathcal{O})^{n+1}$ and $n\geq a$.\\

\noindent We have that $n=qa+t$ for some $q,t\in\mathbb{N}$, $0\leq t<a$, so $q\geq 1$ and $qa\leq n<n+1\leq (q+1)a$. Therefore:\\

$r^{p^m}\in J(\mathcal{O})^n\subseteq J(\mathcal{O})^{qa}=x^q\mathcal{O}=z^{p^kq}\mathcal{O}$, and so $v_{\overline{z},\mathcal{O}}(r^{p^m})\geq p^kq$.\\

\noindent Hence $v_{\overline{z},\mathcal{O}}(s^{p^m})\geq v_{\overline{z},\mathcal{O}}(r^{p^m})+p^k\geq p^kq+p^k=p^k(q+1)$, so $s^{p^m}\in z^{p^k(q+1)}\mathcal{O}=x^{q+1}\mathcal{O}=J(\mathcal{O})^{a(q+1)}\subseteq J(\mathcal{O})^{n+1}$.\\

\noindent Therefore $v_{\mathfrak{q}}(s^{p^m})\geq n+1>n=v_{\mathfrak{q}}(r^{p^m})$.\\

\noindent Furthermore, for all $l\in\mathbb{N}$, $v_{\mathfrak{q}}(s^{p^{m+l
}})\geq p^lv_{\mathfrak{q}}(s^{p^m})>p^lv_{\mathfrak{q}}(r^{p^m})=v_{\mathfrak{q}}(r^{p^{m+l}})$ -- the last equality holds since $r^{p^m}$ is $v_{\mathfrak{q}}$-regular.

Hence $v_{\mathfrak{q}}(s^{p^n})>v_{\mathfrak{q}}(r^{p^n})$ for all sufficiently high $n$ as required.\end{proof}

\subsection{Growth Rates}

Since we are usually only interested in convergence, it makes sense to consider the growth of elements values as they are raised to high powers.

\begin{definition}\label{growth rate}

Let $Q$ be a ring with a filtration $v:Q\to\mathbb{Z}\cup\{\infty\}$. Define $\rho:Q\to\mathbb{R}\cup\{\infty\},x\to\underset{n\rightarrow\infty}{\lim}{\frac{v(x^n)}{n}}$. This is the \emph{growth rate function} of $Q$ with respect to $v$, the proof of \cite[Lemma 1]{Bergmann} shows that this is well defined.

\end{definition}

\begin{lemma}\label{growth properties}

Let $Q$ be a ring with a filtration $v:Q\to\mathbb{Z}\cup\{\infty\}$, and let $\rho$ be the corresponding growth rate function. Then for all $x,y\in Q$:\\

\noindent i. $\rho(x^n)=n\rho(x)$ for all $n\in\mathbb{N}$.

\noindent ii. If $x$ and $y$ commute then $\rho(x+y)\geq\min\{\rho(x),\rho(y)\}$ and $\rho(xy)\geq\rho(x)+\rho(y)$.

\noindent iii. $\rho(x)\geq v(x)$ and $\rho$ is invariant under conjugation.

\noindent iv. If $Q$ is simple and artinian and $v$ is separated, then $\rho(x)=\infty$ if and only if $x$ is nilpotent.

\noindent v. If $x$ is $v$-regular and commutes with $y$, then $\rho(x)=v(x)$ and $\rho(xy)=v(x)+\rho(y)$.

\end{lemma}

\begin{proof}

\noindent $i$ and $ii$ are given by the proof of \cite[Lemma 1]{Bergmann}.\\

\noindent $iii$. For each $n\in\mathbb{N}$, $\frac{v(x^n)}{n}\geq\frac{nv(x)}{n}=v(x)$, and so $\rho(x)\geq v(x)$.\\

\noindent Given $u\in Q^{\times}$, $\rho(uxu^{-1})=\underset{n\rightarrow\infty}{\lim}{\frac{v((uxu^{-1})^n)}{n}}=\underset{n\rightarrow\infty}{\lim}{\frac{v(u(x^n)u^{-1})}{n}}\geq\underset{n\rightarrow\infty}{\lim}{\frac{v(u)+v(x^n)+v(u^{-1})}{n}}=\underset{n\rightarrow\infty}{\lim}{\frac{v(x^n)}{n}}+\underset{n\rightarrow\infty}{\lim}{\frac{v(u)}{n}}+\underset{n\rightarrow\infty}{\lim}{\frac{v(u^{-1})}{n}}=\rho(x)$.\\

\noindent Hence $\rho(x)=\rho(u^{-1}uxu^{-1}u)\geq\rho(uxu^{-1})\geq\rho(x)$ -- forcing equality. Therefore $\rho$ is invariant under conjugation.\\

\noindent $iv$. Clearly if $x$ is nilpotent then $\rho(x)=\infty$.\\

\noindent First suppose that $x$ is a unit, then for any $y\in Q$, $v(y)=v(x^{-1}xy)\geq v(x^{-1})+v(xy)$, and so $v(xy)\leq v(y)-v(x^{-1})$. It follows using induction that for all $n\in\mathbb{N}$, $v(x^ny)\leq v(y)-nv(x^{-1})$, and hence $\frac{v(x^ny)}{n}\leq \frac{v(y)}{n}-v(x^{-1})$.\\

\noindent Taking $y=1$, it follows easily that $\rho(x)\leq -v(x^{-1})$, and since $v$ is separated, this is less than $\infty$.\\

\noindent Now, since $Q$ is simple and artinian, we have that $Q\cong M_l(K)$ for some division ring $K$, $l\in\mathbb{N}$. So applying Fitting's Lemma \cite[section 3.4]{Fitting}, we can find a unit $u\in Q^{\times}$ such that $uxu^{-1}$ has standard Fitting block form $\left(\begin{array}{cc} A & 0\\
0 & B\end{array}\right)$, where $A$ and $B$ are square matrices over $K$, possibly empty, $A$ is invertible and $B$ is nilpotent.\\

\noindent If $x$ is not nilpotent then $uxu^{-1}$ is not nilpotent, and hence $A$ is non-empty. Therefore, $\rho(uxu^{-1})=\rho(A)$, and since $A$ is invertible, $\rho(A)<\infty$. So by part $iii$, $\rho(x)=\rho(uxu^{-1})<\infty$.\\

\noindent $v$. Since $x$ is $v$-regular, $v(x^n)=nv(x)$ for all $n$, so clearly $\rho(x)=v(x)$. Also, $\rho(xy)=\underset{n\rightarrow\infty}{\lim}{\frac{v((xy)^n)}{n}}=\underset{n\rightarrow\infty}{\lim}{\frac{v(x^ny^n)}{n}}=\underset{n\rightarrow\infty}{\lim}{\frac{v(x^n)+v(y^n)}{n}}=\underset{n\rightarrow\infty}{\lim}{\frac{v((x^n)}{n}}+\underset{n\rightarrow\infty}{\lim}{\frac{v(y^n)}{n}}=\rho(x)+\rho(y)=v(x)+v(y)$.\end{proof}

\vspace{0.1in}

\noindent So let $G$ be a non-abelian $p$-valuable, abelian-by-procyclic group with principal subgroup $H$, procyclic element $X$, and let $P$ be a faithful prime ideal of $kG$. Fix $v$ a non-commutative valuation on $Q(\frac{kG}{P})$ such that the natural map $\tau:kG\to Q(\frac{kG}{P})$ is continuous, and let $\rho$ be the growth rate function of $v$.\\

Recall the function $u=z(\varphi^{p^{m_1}}):H\to H$ defined in section 2.1, where $\varphi$ is conjugation by $X$. Define $\lambda:=\inf\{\rho(\tau(u(h)-1)):h\in G\}$. Since we know that for all $h\in H$, $u(h)=u_0(h)^{p^{m_1}}$ and $v(\tau(h-1)^{p^m}))\geq 1$ for sufficiently high $m$, we may choose $m_1$ such that $v(\tau(u(h)-1))\geq 1$, and hence $\lambda\geq 1$.

\begin{lemma}\label{faithful}

$\lambda<\infty$ and $\lambda$ is attained as the growth rate of $\tau(u(h)-1)$ for some $h\in H$.

\end{lemma}

\begin{proof}

Suppose that $\lambda=\infty$, i.e. $\rho(\tau(u(h)-1))=\infty$ for all $h\in G$.\\

\noindent By Lemma \ref{growth properties}($iv$), it follows that $\tau(u(h)-1)$ is nilpotent for all $h\in H$, and hence $u(h)^{p^m}-1\in P$ for sufficiently high $m$.

So since $P$ is faithful, it follows that $u(h)^{p^m}=u_0(h)^{p^{m_1+m}}=1$, and hence $u_0(h)=1$ by the torsionfree property, and this holds for all $h\in H$.\\

\noindent But in this case, $1=u_0(h)=\exp([\log(X),\log(h)])$, so $[\log(X),\log(h)]=0$ in $\log(Sat(G))$, and it follows that $\log(Sat(G))$ is abelian, and hence $G$ is abelian -- contradiction.\\

\noindent Therefore $\lambda$ is finite, and it is clear that for any basis $\{h_1,\cdots,h_d\}$ for $H$, $\lambda=\min\{\rho(\tau(u(h_i)-1)):i=1,\cdots,d\}$, hence it is attained.\end{proof}

\subsection{Growth Preserving Polynomials}

We will now see the first example of proving a control theorem using our Mahler expansions.\\

\noindent Consider a polynomial of the form $f(t)=a_0t+a_1t^p+\cdots+a_rt^{p^r}$, where $a_i\in\tau(kH)$ and $r\geq 0$, we call $r$ the \emph{$p$-degree} of $f$.

Then $f:\tau(kH)\to\tau(kH)$ is an $\mathbb{F}_p$-linear map.\\

\noindent Recall the definition of $\lambda\geq 1$ from section 3.3, which is implicit in the following definition.

\begin{definition}\label{GPP}

We say that $f(t)=a_0t+a_1t^p+\cdots+a_rt^{p^r}$ is a \emph{growth preserving polynomial}, or \emph{GPP}, if:\\ 

\noindent i. $\rho(f(q))\geq p^{r}\lambda$ for all $q\in\tau(kH)$ with $\rho(q)\geq\lambda$.\\

\noindent ii. $\rho(f(q))> p^{r}\lambda$ for all $q\in\tau(kH)$ with $\rho(q)>\lambda$.\\

\noindent We say that a GPP $f$ is \emph{trivial} if for all $q=\tau(u(h)-1)$, $\rho(f(q))>p^r\lambda$.\\

\noindent Furthermore, $f$ is a \emph{special} GPP if $f$ is not trivial, and for any $q=\tau(u(h)-1)$ with $\rho(f(q))=p^{r}\lambda$, we have that $f(q)^{p^k}$ is $v$-regular for sufficiently high $k$.

\end{definition}

\noindent\textbf{\underline{Example:}} $f(t)=t$ is clearly a GPP. In general it need not be special, and it is not trivial, because if $\rho(q)=\lambda$ then $\rho(f(q))=\rho(q)=\lambda$.

\vspace{0.2in}

\noindent For any growth preserving polynomial $f(t)$, define $K_f:=\{h\in H:\rho(f(\tau(u(h)-1)))>p^{r}\lambda\}$.

\begin{lemma}\label{sub}

If $f(t)$ is a GPP, then $K_f$ is an open subgroup of $H$ containing $H^p$. Moreover, $K_f=H$ if and only if $f$ is trivial.

\end{lemma}

\begin{proof}

Firstly, it is clear from the definition that $K_f=H$ if and only if $\rho(f(\tau(u(h)-1)))>p^r\lambda$ for all $h\in H$, i.e. if and only if $f$ is trivial.\\

\noindent Given $h,h'\in H$, let $q=\tau(u(h)-1)$, $q'=\tau(u(h')-1)$. Then

\begin{center}
$\tau(u(hh')-1)=\tau((u(h)-1)(u(h')-1)+(u(h)-1)+(u(h')-1)))=qq'+q+q'$
\end{center} 

\noindent Therefore $f(\tau(u(hh')-1))=f(qq')+f(q)+f(q)$ using $\mathbb{F}_p$-linearity of $f$. But $\rho(q),\rho(q')\geq\lambda$ by the definition of $\lambda$, so $\rho(qq')\geq 2\lambda>\lambda$, so by the definition of a GPP, $\rho(f(qq'))>p^r\lambda$.\\

\noindent We know that $\rho(f(q)),\rho(f(q'))> p^r\lambda$, thus $\rho(f(q_1q_2)+f(q_1)+f(q_2))>p^r\lambda$. Therefore, $\rho(f(\tau(u(hh')-1)))>p^r\lambda$ and $hh'\in K_f$ as required.\\

\noindent Also, $\tau(u(h^{-1})-1)=-\tau(u(h^{-1}))\tau(u(h)-1)=-\tau(u(h^{-1})-1)\tau(u(h)-1)-\tau(u(h)-1)$, so by the same argument it follows that $h^{-1}$ in $K_f$, and $K_f$ is a subgroup of $H$.\\

\noindent Finally, for any $h\in H$ $\rho(\tau(u(h^p)-1))=\rho(\tau(u(h)-1)^p)\geq p\lambda>\lambda$, hence $\rho(f(\tau(u(h^p)-1)))>p^r\lambda$ and $h^p\in K_f$. Therefore $K_f$ contains $H^p$ and $K_f$ is an open subgroup of $H$.\end{proof}

\vspace{0.1in}

\noindent In particular, for $f(t)=t$, let $K:=K_f=\{h\in H:\rho(\tau(u(h)-1))>\lambda\}$. Then $K$ is a proper open subgroup of $H$ containing $H^p$.

\vspace{0.2in}

\noindent Given a polynomial $f(t)=a_0t+a_1t^p+\cdots+a_rt^{p^r}$, and a basis $\{h_1,\cdots,h_d\}$ for $H$, let $q_i=\tau(u(h_i)-1)$, and recall our expression (\ref{MahlerPol}) from section 2:

\begin{center}
$0=f(q_1)^{p^m}\tau\partial_1(y)+....+f(q_{d})^{p^m}\tau\partial_{d}(y)+O(f(q)^{p^m})$
\end{center}

\noindent Where $v(q^{p^m})\geq 2v(q_i^{m})$ for all $m$, and hence $\rho(q)>\rho(q_i)\geq\lambda$. Then if $f$ is a GPP, it follows that $\rho(f(q))>p^r\lambda$, and hence $v(f(q)^{p^m})>p^{m+r}\lambda$ for $m>>0$.\\

\noindent For the rest of this section, fix a non-trivial GPP $f$ of $p$ degree $r$. Then $K_f$ is a proper open subgroup of $H$ containing $H^p$ by Lemma \ref{sub}, so fix a basis $\{h_1,\cdots,h_d\}$ for $H$ such that $\{h_1^p,\cdots,h_t^p,h_{t+1},\cdots,h_d\}$ is an ordered basis for $K_f$.\\

\noindent Set $q_i:=\tau(u(h_i)-1)$ for each $i$, so that $\rho(f(q_i))=p^r\lambda$ for $i\leq t$ and $\rho(f(q_i))>p^r\lambda$ for $i>t$, and define:

\noindent $S_m:=\left(\begin{array}{cccccc} f(q_1)^{p^m} & f(q_2)^{p^m} & . & . & f(q_t)^{p^m}\\
f(q_1)^{p^{m+1}} & f(q_2)^{p^{m+1}} & . & . & f(q_t)^{p^{m+1}}\\
. & . & . & . & .\\
. & . & . & . & .\\
f(q_1)^{p^{m+t-1}} & f(q_2)^{p^{m+t-1}} & . & . & f(q_t)^{p^{m+t-1}}\end{array}\right)$, $\underline{\partial}:=\left(\begin{array}{c}\tau\partial_1(y)\\
\tau\partial_2(y)\\
.\\
.\\
.\\
\tau\partial_t(y)\end{array}\right)$\\

\noindent Then we can rewrite our expression as:

\begin{equation}
0=S_m\underline{\partial}+\left(\begin{array}{c} O(f(q)^{p^m})\\
O(f(q)^{p^{m+1}})\\
.\\
.\\
.\\
O(f(q)^{p^{m+t-1}})\end{array}\right)
\end{equation}

\noindent And multiplying by $adj(S_m)$ gives:

\begin{equation}\label{Mahler''}
0=det(S_m)\underline{\partial}+adj(S_m)\left(\begin{array}{c} O(f(q)^{p^m})\\
O(f(q)^{p^{m+1}})\\
.\\
.\\
.\\
O(f(q)^{p^{m+t-1}})\end{array}\right)
\end{equation}

\begin{lemma}\label{adjoint}

Suppose that $f$ is special. Then for each $i,j\leq t$, the $(i,j)$-entry of $adj(S_m)$ has value at least $\frac{p^t-1}{p-1}p^{m+r}\lambda-p^{m+r+j-1}\lambda$ for sufficiently high $m$.

\end{lemma}

\begin{proof}

By definition, the $(i,j)$-entry of $adj(S_m)$ is (up to sign) the determinant of the matrix $(S_m)_{i,j}$ obtained by removing the $j$'th row and $i$'th column of $S_m$. This determinant is a sum of elements of the form:

\begin{center}
$f(q_{k_1})^{p^m}f(q_{k_2})^{p^{m+1}}\cdots \widehat{f(q_{k_j})^{p^{m+j-1}}}\cdots f(q_{k_t})^{p^{m+t-1}}$
\end{center}

\noindent For $k_i\leq t$, where the hat indicates that the $j$'th term in this product is omitted.\\

\noindent Since $f$ is special, $f(q_{k_s})^{p^m}$ is $v$-regular for $m>>0$. So since $\rho(f(q_{k_s}))=p^r\lambda$, it follows that $f(q_{k_s})$ is $v$-regular of value $p^{m+r}\lambda$.\\

\noindent Therefore, this $(i,j)$-entry has value at least $(1+p+\cdots+\widehat{p^{j-1}}+\cdots+p^{t-1})p^{m+r}\lambda$

\noindent $=\frac{p^t-1}{p-1}p^{m+r}\lambda-p^{m+r+j-1}\lambda$.\end{proof}

\vspace{0.1in}

\begin{lemma}\label{delta}

Let $\Delta:=\underset{\alpha\in\mathbb{P}^{t-1}\mathbb{F}_p}{\prod}{(\alpha_1f(q_1)+\cdots+\alpha_tf(q_t))}$.

\noindent Then there exists $\delta\in\tau(kH)$, which is a product of length $\frac{p^t-1}{p-1}$ in elements of the form $f(\tau(u(h)-1))$, $h\in H\backslash K_f$, such that $\rho(\Delta-\delta)>\frac{p^t-1}{p-1}p^r\lambda$.

\end{lemma}

\begin{proof}

For each $\alpha\in\mathbb{F}_p^t\backslash\{0\}$, we have that $\alpha_1f(q_1)+\cdots+\alpha_tf(q_t)=f(\alpha_1q_1+\cdots+\alpha_tq_t)$ using linearity of $f$.\\

\noindent Using expansions inside $kH$, we see that $\alpha_1q_1+\cdots+\alpha_{t}q_{t}=\tau(u(h_1^{\alpha_1}\cdots h_{t}^{\alpha_{t}})-1)+O(q_jq_k)$, and hence:\\

$f(\alpha_1q_1+\cdots+\alpha_{t}q_t)=f(\tau(u(h_1^{\alpha_1}\cdots h_t^{\alpha_t})-1))+O(f(q_jq_k))$.\\

\noindent So setting $h_{i,\alpha}:=u(h_1^{\alpha_1}\cdots h_{t}^{\alpha_{t}})$. So since $\alpha_i\neq 0$ for some $i$, it follows from $\mathbb{F}_p$-linear independence of $h_1,\cdots,h_t$ modulo $K_f$, that $\rho(f(\tau(u(h_{i,\alpha})-1)))=p^r\lambda$, and hence $h_{i,\alpha}\in H\backslash K_f$.\\

\noindent Set $\delta:=\underset{\alpha\in\mathbb{P}^{t-1}\mathbb{F}_p}{\prod}{f(\tau(u(h_{i,\alpha})-1))}$. Then $\Delta=\underset{\alpha\in\mathbb{P}^{t-1}\mathbb{F}_p}{\prod}{(\alpha_1f(q_1)+\cdots+\alpha_tf(q_t))}$\\

\noindent $=\underset{\alpha\in\mathbb{P}^{t-1}\mathbb{F}_p}{\prod}{f(\tau(u(h_{i,\alpha})-1))}+O(f(q_jq_k))=\delta+\epsilon$\\

\noindent Where $\epsilon$ is a sum of products over all $i$, $\alpha$ in $f(\tau(u(h_{i,\alpha})-1))$ and $O(f(q_jq_k))$, with each product containing at least one $O(f(q_jq_k))$. \\

\noindent Since the length of each of these products is $\frac{p^{t}-1}{p-1}$, and each term has growth rate at least $p^r\lambda$, with one or more having growth rate strictly greater than $p^r\lambda$, it follows that $\rho(\epsilon)>\frac{p^t-1}{p-1}p^r\lambda$ as required.\end{proof}

\begin{theorem}\label{special-GPP}

Suppose that there exists a special growth preserving polynomial $f$. Then $P$ is controlled by a proper open subgroup of $G$.

\end{theorem}

\begin{proof}

\noindent Since $f$ is special, we have that for some $k>0$ $f(q_i)^{p^k}$ is $v$-regular for each $i\leq t$, and thus $v(f(q_i)^{p^k})=\rho(f(q_i)^{p^k})=p^{r+k}\lambda$ and for all $m\geq k$, $v(f(q_i)^{p^m})=p^{r+m}\lambda$.\\

\noindent Also, since $\rho(f(q))>p^r\lambda$, we can choose $c>0$ such that $\rho(f(q))>p^r\lambda+c$, and hence $v(f(q)^{p^m})>p^{m+r}\lambda+p^mc$ for sufficiently high $m$.\\

\noindent Consider our Mahler expansion (\ref{Mahler''}):

\begin{center}
$0=det(S_m)\underline{\partial}+adj(S_m)\left(\begin{array}{c} O(f(q)^{p^m})\\
O(f(q)^{p^{m+1}})\\
.\\
.\\
.\\
O(f(q)^{p^{m+t-1}})\end{array}\right)$
\end{center}

\noindent We will analyse this expression to prove that $\tau\partial_i(P)=0$ for all $i\leq t$, and it will follow from Proposition \ref{Frattini} that $P$ is controlled by a proper open subgroup of $G$ as required.\\ 

\noindent Consider the $i$'th entry of the vector 

\begin{center}
$adj(S_m)\left(\begin{array}{c} O(f(q)^{p^m})\\
O(f(q)^{p^{m+1}})\\
.\\
.\\
.\\
O(f(q)^{p^{m+t-1}})\end{array}\right)$
\end{center} 

\noindent This has the form $adj(S_m)_{i,1}O(f(q)^{p^m})+adj(S_m)_{i,2}O(f(q)^{p^{m+1}})+\cdots+adj(S_m)_{i,t}O(f(q)^{p^{m+t-1}})$.

\vspace{0.2in}

\noindent By Lemma \ref{adjoint}, we know that $v(adj(S_m)_{i,j})\geq \frac{p^t-1}{p-1}p^{m+r}\lambda-p^{m+r+j-1}\lambda$ for $m>>0$, and hence:\\ 

\noindent $v(adj(S_m)_{i,j}O(f(q)^{p^{m+j-1}}))\geq v(adj(S_m)_{i,j})+v(O(f(q)^{p^{m+j-1}}))$\\

\noindent $\geq \frac{p^t-1}{p-1}p^{m+r}\lambda-p^{m+r+j-1}\lambda+p^{m+r+j-1}\lambda+p^{m+j-1}c=\frac{p^t-1}{p-1}p^{m+r}\lambda+p^{m+j-1}c$ for each $j$.\\

\noindent Hence this $i$'th entry has value at least $\frac{p^t-1}{p-1}p^{m+r}\lambda+p^mc$.\\

\noindent Therefore, the $i$'th entry of our expression (\ref{Mahler''}) has the form $0=det(S_m)\tau\partial_i(y)+\epsilon_{i,m}$, where $v(\epsilon_{i,m})\geq \frac{p^t-1}{p-1}p^{m+r}\lambda+p^mc$.\\

\noindent Now, take $\Delta:=det(S_0)$, and it follows that $det(S_m)=\Delta^{p^m}$ for each $m$. Also, using \cite[Lemma 1.1($ii$)]{chevalley} we see that:\\

$\Delta=\beta\cdot\underset{\alpha\in\mathbb{P}^{t-1}\mathbb{F}_p}{\prod}{(\alpha_1f(q_1)+\cdots+\alpha_tf(q_t))}$ for some $\beta\in\mathbb{F}_p$.\\

\noindent Therefore, by Lemma \ref{delta}, we can find an element $\delta\in\tau(kH)$, which up to scalar multiple is a product of length $\frac{p^t-1}{p-1}$ in elements of the form $f(\tau(u(h)-1))$, with $\rho(f(\tau(u(h)-1)))=p^r\lambda$ for each $h$, such that $\rho(\Delta-\delta)>\frac{p^t-1}{p-1}p^{r}\lambda$. 

Hence we can find $c'>0$ such that for all $m>>0$, $v((\Delta-\delta)^{p^m})\geq \frac{p^t-1}{p-1}p^{m+r}\lambda+p^mc'$.\\

\noindent Therefore, $0=det(S_m)\tau\partial_i(y)+\epsilon_{i,m}=\Delta^{p^m}\tau\partial_i(y)+\epsilon_{i,m}=\delta^{p^m}\tau\partial_i(y)+(\Delta-\delta)^{p^m}\tau\partial_i(y)+\epsilon_{i,m}$.\\

\noindent Again, since $f$ is a special GPP, if $h\in H$ and $\rho(f(\tau(u(h)-1)))=p^r\lambda$, it follows that $f(\tau(u(h)-1))^{p^k}$ is $v$-regular for $k>>0$.\\

\noindent So since $\delta$ is a product of $\frac{p^t-1}{p-1}$ elements of the form $f(\tau(u(h)-1))$ of growth rate $p^r\lambda$, it follows that for some $k\in\mathbb{N}$, $\delta^{p^k}$ is $v$-regular of value $\frac{p^t-1}{p-1}p^{r+k}\lambda$.\\

\noindent Therefore for all $m\geq k$, $\delta^{p^m}$ is $v$-regular of value $\frac{p^t-1}{p-1}p^{r+m}\lambda$, and dividing out by $\delta^{p^m}$ gives that for each $i=1,\cdots,t$:

\begin{center}
$0=\tau\partial_i(y)+\delta^{-p^m}(\Delta-\delta)^{p^m}\tau\partial_i(y)+\delta^{-p^m}\epsilon_{i,m}$.
\end{center}

\noindent But for $m>>0$, $v(\delta^{-p^m}(\Delta-\delta)^{p^m})\geq \frac{p^t-1}{p-1}p^{m+r}\lambda+p^mc'-\frac{p^t-1}{p-1}p^{m+r}\lambda=p^mc'$, and $v(\delta^{-p^m}\epsilon_{i,m})\geq \frac{p^t-1}{p-1}p^{m+r}\lambda+p^mc-\frac{p^t-1}{p-1}p^{m+r}\lambda=p^mc$, hence the right hand side of this expression converges to $\tau\partial_i(y)$.\\

\noindent Therefore, $\tau\partial_i(y)=0$, and since this holds for all $y\in P$, we have that $\tau\partial_i(P)=0$ for each $i=1,\cdots,t$.\end{proof}

\noindent So to prove Theorem \ref{B}, it remains only to prove the existence of a special growth preserving polynomial. We will construct such a polynomial in section 6.

\section{Central Simple algebras}

Suppose that $G$ is an abelian-by-procyclic group, and $P$ is a faithful prime ideal of $kG$. Since $Q:=Q(\frac{kG}{P})$ is simple, its centre is a field. Recall that the proof of Theorem \ref{B} will split into two parts, when $Q$ is finite dimensional over its centre, and when it is not.

In this section, we deal with the former case. So we will assume throughout that $Q$ is finitely generated over its centre, i.e. is a central simple algebra.

\subsection{Isometric embedding}

\noindent Fix a non-commutative valuation $v$ on $Q$. Then by definition, the completion $\widehat{Q}$ of $Q$ with respect to $v$ is isomorphic to $M_n(Q(D))$ for some complete, non-commutative DVR $D$.

But since $Q$ is finite dimensional over its centre, the same property holds for $\widehat{Q}$, and hence $Q(D)$ is finite dimensional over its centre.\\ 

\noindent Let $F:=Z(Q(D))$, $s:= $dim$_F(Q(D))$, $R:=F\cap D$. Then $F$ is a field, $R$ is a commutative DVR, and $Q(D)\cong F^s$. Let $\pi\in R$ be a uniformiser, and suppose that $v(\pi)=t>0$.

\begin{lemma}\label{technical'}

Let $\{y_1,\cdots,y_s\}$ be an $F$-basis for $Q(D)$ with $0\leq v(y_i)\leq t$ for all $i$. Then there exists $l\in\mathbb{N}$ such that if $v(r_1y_1+\cdots+r_sy_s)\geq l$ then $v(r_i)>0$ for some $i$ with $r_i\neq 0$.

\end{lemma}

\begin{proof}

First, note that the field $F$ is complete with respect to the non-archimedian valuation $v$, and $Q(D)\cong F^s$ carries two filtrations as a $F$-vector space, which both restrict to $v$ on $F$. One is the natural valuation $v$ on $Q(D)$, the other is given by $v_0(r_1y_1+....+r_sy_s)=\min\{v(r_i):i=1,...,s\}$.\\

\noindent But it follows from \cite[Proposition 2.27]{norm} that any two norms on $F^s$ are topologically equivalent. Hence $v$ and $v_0$ induce the same topology on $Q(D)$.

It is easy to see that any subspace of $F^s$ is closed with respect to $v_0$, and hence also with respect to $v$.\\

\noindent So, suppose for contradiction that for each $m\in\mathbb{N}$, there exist $r_{i,m}\in K$, not all zero, with $v(r_{i,m})\leq 0$ if $r_{i,m}\neq 0$, and $v(r_{1,m}y_1+....+r_{s,m}y_s)\geq m$.\\

\noindent Then there exists $i$ such that $r_{i,m}\neq 0$ for infinitely many $m$, and we can assume without loss of generality that $i=1$. So from now on, we assume that $r_{1,m}\neq 0$ for all $m$.\\

\noindent Dividing out by $r_{1,m}$ gives that $v(y_1+t_{2,m}y_2+....+t_{s,m}y_s)\geq m-v(r_{1,m})\geq m$, where $t_{i,m}=r_{1,m}^{-1}r_{i,m}\in K$.\\

\noindent Hence $\underset{m\rightarrow\infty}{\lim}{(y_1+t_{2,m}y_2+....+t_{s,m}y_s)}=0$, and thus $\underset{n\rightarrow\infty}{\lim}{(t_{2,m}y_2+....+t_{s,m}y_s)}$ exists and equals $-y_1$. But since Span$_F\{y_2,....,y_s\}$ is closed in $F^n$, this means that $y_1\in$ Span$_F\{y_2,.....,y_s\}$ -- contradiction.\end{proof}

\begin{proposition}\label{integral basis}

There exists a basis $\{x_1,....,x_s\}\subseteq D$ for $Q(D)$ over $F$ such that $D=Rx_1\oplus....\oplus R x_s$.

\end{proposition}

\begin{proof}

It is clear that we can find an $F$-basis $\{y_1,....,y_s\}\subseteq D$ for $Q(D)$ such that $v(y_i)<t$ for all $i$, just by rescaling elements of some arbitrary basis.

Therefore, by Lemma \ref{technical'}, there exists $l\in\mathbb{N}$ such that if $v(r_1y_1+\cdots+r_st_s)\geq l$ then $v(r_i)>0$ for some $i$ with $r_i\neq 0$.\\

\noindent Choose $m\in\mathbb{N}$ such that $tm>l$. Then given $x\in D\backslash \{0\}$, $v(x)\geq 0$ so $v(\pi^mx)\geq tm>l$. So if $\pi^mx=r_1y_1+....+r_sy_s$ then $v(r_i)\geq 0$ for some $r_i\neq 0$.\\

\noindent It follows from an easy inductive argument that $\pi^{ms}x\in Ry_1\oplus.....\oplus Ry_s$, and hence $\pi^{ms}D\subseteq Ry_1\oplus.....\oplus Ry_s$.\\

\noindent But $Ry_1\oplus.....\oplus Ry_s$ is a free $R$-module, and $R$ is a commutative PID, hence any $R$-submodule is also free. Hence $\pi^{ms}D=R(\pi^{ms}x_1)\oplus....\oplus R(\pi^{ms}x_e)$ for some $x_i\in D$.\\

\noindent It follows easily that $e=s$, $\{x_1,....,x_s\}$ is an $F$-basis for $Q(D)$, and $D=Rx_1\oplus....\oplus Rx_s$. \end{proof}

\vspace{0.1in}

\noindent Now, we restrict our non-commutative valuation $v$ on $\widehat{Q}\cong M_n(Q(D))$ to $Q(D)$, and by definition this is the natural $J(D)$-adic valuation.

Using Proposition \ref{integral basis}, we fix a basis $\{x_1,\cdots,x_s\}\subseteq D$ for $Q(D)$ over $F$, with $D=Rx_1\oplus Rx_2\oplus\cdots\oplus R x_s$.

\begin{proposition}\label{embedding}

Let $F'$ be any finite extension of $F$; then $v$ extends to $F'$. Let $v'$ be the standard matrix filtration of $M_s(F')$ with respect to $v$.

Then there is a continuous embedding of $F$-algebras $\phi: Q(D)\xhookrightarrow{} M_s(F')$ such that 

\begin{center}
$v'(\phi(x))\leq v(x)<v'(\phi(x))+2t$ for all $x\in Q(D)$.
\end{center}

\noindent Hence applying the functor $M_n$ to $\phi$ gives us a continuous embedding $M_n(\phi):\widehat{Q}\hookrightarrow{} M_{ns}(F')$ such that $v'(M_n(\phi)(x))\leq v(x)<v'(M_n(\phi)(x))+2t$ for all $x\in \widehat{Q}$.

\end{proposition}

\begin{proof}

It is clear that the embedding $F\to F'$ is an isometry, so it suffices to prove the result for $F'=F$.\\

\noindent Again, define $v_0:Q(D)\to\mathbb{Z}$ $\cup$ $\{\infty\},\sum_{i=1}^{s}{r_ix_i}\mapsto \min\{v(r_i): i=1,....,s\}$, it is readily checked that this is a separated filtration of $F$-vector spaces, and clearly $v(x)\geq v_0(x)$ for all $x\in Q(D)$.\\

\noindent Then if $x=r_1x_1+....+r_sx_s\in Q(D)$ with $0\leq v(x)<t$, then $v(r_i)\geq 0$ for all $i$ because $D$ is an $R$-lattice by Lemma \ref{integral basis}. Since $v(x_i)\geq 0$ for all $i$, $v(r_j)<t$ for some $j$, so since $r_j\in R$, this means that $v(r_j)=0$, and hence $v_0(x)=0$.\\

\noindent So if $x\in Q(D)$ with $v(x)=l$, then $at\leq l < (a+1)t$ where $a=\floor{\frac{l}{t}}$, and hence $0\leq v(\pi^{-a}x)<l$, so $v_0(\pi^{-a}x)=0$.

Thus $\pi^{-a}x=r_1x_1+....+r_sx_s$ with $v(r_i)\geq 0$ for all $i$, $v(r_j)=0$ for some $i$, and hence $v_0(x)=ta$, so $v(x)< v_0(x)+t$.\\

\noindent So it follows that $v_0(x)\leq v(x)< v_0(x)+t$ for all $x\in Q(D)$, in particular $0\leq v(x_i)<t$ for all $i$. Hence the identity map $(Q(D),v)\to (Q(D),v_0)$ is bounded.\\

\noindent Now, consider the map $\phi:Q(D)\to $ End$_F(Q(D)), x\mapsto (Q(D)\to Q(D), d\mapsto x\cdot d)$, this is an injective $F$-algebra homomorphism.\\

\noindent Also, End$_F(Q(D))$ carries a natural filtration of $F$-algebras given by 

\begin{center}
$v'(\psi):=\min\{v_0(\psi(x_i)):i=1,...,s\}$ for each $\psi\in $ End$_F(Q(D))$.
\end{center} 

\noindent Using the isomorphism End$_F(Q(D))\cong M_s(F)$, this is just the standard matrix filtration, and it is readily seen that

\begin{center}
$v'(\psi)=\inf\{v_0(\psi(d)):d\in Q(D), 0\leq v'(d)<t\}$ for all $\psi\in $ End$_F(Q(D))$.
\end{center}

\noindent So if $v_0(x)=r$ then since $v_0(1)=0$ and $v_0(\phi(x)(1))=v_0(x\cdot 1)=r$, it follows that $v'(\phi(x))\leq r$. But for all $i=1,...,s$:\\ 

$v_0(x\cdot x_i)>v(x\cdot x_i)-t\geq r+v(x_i)-t\geq r-t$, hence $v'(\phi(x))\geq r-t$.\\ 

\noindent Therefore $v'(\phi(x))\leq v_0(x)\leq v'(\phi(x))+t$ for all $x$, so $\phi$ is bounded, and hence continuous.\\

\noindent Finally, since for all $x\in Q(D)$, $v_0(x)\leq v(x)\leq v_0(x)+t$ and $v'(\phi(x))\leq v_0(x)\leq v'(\phi(x))+t$, it follows that $v'(\phi(x))\leq v(x)\leq v'(\phi(x))+2t$. \end{proof}

\vspace{0.1in}

\noindent Recall from Definition \ref{growth rate} the growth rate function $\rho'$ of $M_{ns}(F')$ with respect to $v'$. Then using Proposition \ref{embedding}, we see that for all $x\in\widehat{Q}$:\\ 

\noindent $\rho(x)=\underset{n\rightarrow\infty}{\lim}{\frac{v(x^n)}{n}}\leq \underset{n\rightarrow\infty}{\lim}{\frac{v'(x^n)+2t}{n}}=\rho'(x)\leq\rho(x)$ -- forcing equality.\\

\noindent Therefore $\rho'=\rho$ when restricted to $\widehat{Q}$.

\subsection{Diagonalisation}

\noindent Again, recall our Mahler expansion (\ref{Mahler}):

\begin{center}
$0=q_1^{p^m}\tau\partial_1(y)+....+q_{d}^{p^m}\tau\partial_{d}(y)+O(q^{p^m})$
\end{center}

\noindent Where each $q_i=\tau(u(h_i)-1)$ for some basis $\{h_1,\cdots,h_d\}$ for $H$, and $\rho(q)>\rho(q_i)$ for each $i$.\\

\noindent We may embed $Q(D)$ continuously into $M_s(F')$ for any finite extension $F'$ of $F=Z(Q(D))$ by Proposition \ref{embedding}, and since each $q_i$ is a square matrix over $Q(D)$, by choosing $F'$ appropriately, we may ensure that they can be reduced to Jordan normal form inside $M_{ns}(F')$.\\

\noindent But since $F'$ has characteristic $p$, after raising to sufficiently high $p$'th powers, a Jordan block becomes diagonal. So we may choose $m_0\in\mathbb{N}$ such that $q_i^{p^{m_0}}$ is diagonalisable for each $i$.\\

\noindent But $q_1,\cdots,q_d$ commute, and it is well known that commuting matrices can be simultaneously diagonalised. Hence there exists $a\in M_{ns}(F')$ invertible such that $aq_i^{p^{m_0}}a^{-1}$ is diagonal for each $i$.\\

\noindent So, let $t_i:=aq_ia^{-1}$, then after multiplying (\ref{Mahler}) on the left by $a$, we get:

\begin{center}
$0=t_1^{p^m}a\tau\partial_1(y)+....+t_{d}^{p^m}a\tau\partial_{d}(y)+O(aq^{p^m})$
\end{center}

\noindent Note that since $t_i^{p^{m_0}}$ is diagonal for each $i$, $\rho'(t_i^{p^{m_0}})=v'(t_i^{p^{m_0}})$, and $\rho'(t_i)=\rho'(q_i)$ since growth rates are invariant under conjugation by Lemma \ref{growth properties}($iii$). Since $\rho'=\rho$ on $\widehat{Q}$, it follows that $\rho(q_i^{p^{m_0}})=v'(t_i^{p^{m_0}})$.

Moreover, $v'(t_i^{p^m})=\rho(q_i^{p^m})=p^m\rho(q_i)$ for each $m\geq m_0$.\\

\noindent Also, $q_i=\tau(u(h_i)-1)=\tau(u_0(h_i)-1)^{p^{m_1}}$, so after replacing $m_1$ by $m_1+m_0$ we may ensure that each $t_i$ is diagonal, and hence $v'(t_i)=\rho(q_i)$.\\

\noindent Now, let $K:=\{h\in H:\rho(\tau(u(h)-1))>\lambda\}$. Then since $id$ is a non-trivial GPP and $K=K_{id}$, it follows from Lemma \ref{sub} that $K$ is a proper open subgroup of $H$ containing $H^p$.\\

\noindent For the rest of this section, fix a basis $\{h_1,\cdots,h_d\}$ for $H$ such that $\{h_1^p,\cdots,h_r^p,h_{r+1},\cdots,h_d\}$ is a basis for $K$, $q_i=\tau(u(h_i)-1)$, $t_i=aq_ia^{-1}$ as above.\\

\noindent Then it follows that for all $i\leq r$, $v'(t_i)=\rho(q_i)=\lambda$, and for $i>r$, $v'(t_i)>\lambda$, so we have:

\begin{equation}\label{Mahlerf}
0=t_1^{p^m}a\tau\partial_1(y)+....+t_{r}^{p^m}a\tau\partial_{r}(y)+O(aq^{p^m})
\end{equation}

\noindent Where $\rho(q)>\lambda$.\\

\begin{definition}\label{independent}

Given $c_1,\cdots,c_m\in M_{ns}(F')$ with $v'(c_i)=\mu$ for some $i$, we say that $c_1,\cdots,c_m$ are \emph{$\mathbb{F}_p$-linearly independent modulo $\mu^+$} if for any $\alpha_1,\cdots,\alpha_m\in\mathbb{F}_p$, $v'(\alpha_1c_1+\cdots+\alpha_mc_m)>\mu$ if and only if $\alpha_i=0$ for all $i$.

\end{definition}

\begin{lemma}\label{technical}

$t_1,\cdots,t_r$ are $\mathbb{F}_p$-linearly dependent modulo $\lambda^+$.

\end{lemma}

\begin{proof}

Suppose, for contradiction, that $v'(\alpha_1t_1+\cdots+\alpha_rt_r)>\lambda$ for some $\alpha_i\in\mathbb{F}_p$, not all zero, then using Lemma \ref{growth properties}($iii$) we see that \\

\noindent $\rho(\alpha_1q_1+\cdots+\alpha_rq_r)=\rho(a(\alpha_1q_1+\cdots+\alpha_rq_r)a^{-1})=\rho(\alpha_1t_1+\cdots+\alpha_rt_r)=v'(\alpha_1t_1+\cdots+\alpha_rt_r)>\lambda$.\\

\noindent But since $q_i=\tau(u(h_i)-1)$ for each $i$, we can see using expansions in $kG$ that $\alpha_1q_1+\cdots+\alpha_rq_r=\tau(u(h_1^{\alpha_1}\cdots h_r^{\alpha_r})-1)+O(q_iq_j)$, and clearly $\rho(O(q_iq_j))>\lambda$, and hence $\rho(\tau(u(h_1^{\alpha_1}\cdots h_r^{\alpha_r})-1))>\lambda$.\\

\noindent But $K=\{g\in G:\rho(\tau(u(g)-1))>\lambda\}=\langle h_1^p,\cdots,h_r^p,h_{r+1},\cdots,h_d,X\rangle$, so since $p$ does not divide every $\alpha_i$, it follows that $h_1^{\alpha_1}\cdots h_r^{\alpha_r}\notin K$, and hence $\rho(\tau(u(h_1^{\alpha_1}\cdots h_r^{\alpha_r})-1))=\lambda$ -- contradiction.\end{proof}

\noindent\textbf{\underline{Notation:}} For each $i=1,\cdots,ns$, denote by $e_i$ the diagonal matrix with 1 in the $i$'th diagonal position, 0 elsewhere.

\begin{proposition}\label{linear-indep}

Suppose $d_1,\cdots,d_r\in M_{ns}(F')$ are diagonal, $v'(d_i)=\lambda$ for each $i$, and suppose that for all $m\in\mathbb{N}$ we have:

\begin{center}
$0=d_1^{p^m}a_1+\cdots+d_r^{p^m}a_r+O(aq^{p^m})$
\end{center}

\noindent Where $a_i,a,q\in M_{ns}(F')$, $\rho(q)>\lambda$.\\

\noindent Suppose further that for some $j\in\{1,\cdots,ns\}$, the $j$'th entries of $d_1,\cdots,d_r$ are $\mathbb{F}_p$-linearly independent modulo $\lambda^+$. Then $e_ja_i=0$ for all $i=1,\cdots,r$.

\end{proposition}

\begin{proof}

Firstly, since $d_{1,j},\cdots,d_{r,j}$ are $\mathbb{F}_p$-linearly independent modulo $\lambda^+$, it follows immediately that $e_jd_1,\cdots,e_jd_r$ are $\mathbb{F}_p$-linearly independent modulo $\lambda^+$. And:

\begin{center}
$0=(e_jd_1)^{p^m}a_1+\cdots+(e_jd_r)^{p^m}a_r+O(e_jaq^{p^m})$
\end{center}

\noindent For convenience, set $d_i':=e_jd_i$, and in a similar vein to the proof of Theorem \ref{special-GPP}, define the following matrices:

\noindent $D_m:=\left(\begin{array}{cccccc} d_1'^{p^m} & d_2'^{p^m} & . & . & d_r'^{p^m}\\
d_1'^{p^{m+1}} & d_2'^{p^{m+1}} & . & . & d_r'^{p^{m+1}}\\
. & . & . & . & .\\
. & . & . & . & .\\
d_1'^{p^{m+t-1}} & d_2'^{p^{m+t-1}} & . & . & d_r'^{p^{m+t-1}}\end{array}\right)$, $\underline{a}:=\left(\begin{array}{c}a_1\\
a_2\\
.\\
.\\
.\\
a_r\end{array}\right)$\\

\noindent Then we can rewrite our expression as:

\begin{center}
$0=D_m\underline{a}+\left(\begin{array}{c} O(e_jaq^{p^m})\\
O(e_jaq^{p^{m+1}})\\
.\\
.\\
.\\
O(e_jaq^{p^{m+t-1}})\end{array}\right)$
\end{center}

\noindent And multiplying by $adj(D_m)$ gives:

\begin{equation}\label{Mahlerf'}
0=det(D_m)\underline{a}+adj(D_m)\left(\begin{array}{c} O(e_jq^{p^m})\\
O(e_jaq^{p^{m+1}})\\
.\\
.\\
.\\
O(e_jaq^{p^{m+t-1}})\end{array}\right)
\end{equation}

\noindent And the proof of Lemma \ref{adjoint} shows that the $(i,j)$-entry of $adj(D_m)$ has value at least $\frac{p^r-1}{p-1}p^{m}\lambda-p^{m+j-1}\lambda$.\\

\noindent Since $\rho(q)>\lambda$, fix $c>0$ such that $\rho(q)>\lambda+c$, and hence $v'(e_jaq^{p^m})\geq p^m\lambda+p^mc+v(a)$ for all sufficiently high $m$. Then we see that the $i$'th entry of the vector 

\begin{center}
$adj(D_m)\left(\begin{array}{c} O(e_jaq^{p^m})\\
O(e_jaq^{p^{m+1}})\\
.\\
.\\
.\\
O(e_jaq^{p^{m+t-1}})\end{array}\right)$
\end{center} 

\noindent has value at least $\frac{p^r-1}{p-1}p^{m}\lambda+p^mc+v(a)$ for $m>>0$.\\

\noindent So examining the $i$'th entry of our expression (\ref{Mahlerf'}) gives that $0=det(D_m)a_i+\epsilon_{i,m}$, where $v'(\epsilon_{i,m})\geq \frac{p^t-1}{p-1}p^{m+r}\lambda+p^mc+v(a)$.\\

\noindent Let $\Delta:=det(D_0)$, then $det(D_m)=\Delta^{p^m}$ for all $m\in\mathbb{N}$, and using \cite[Lemma 1.1($ii$)]{chevalley} we see that 

\begin{center}
$\Delta=\beta\cdot\underset{\alpha\in\mathbb{P}^{r-1}\mathbb{F}_p}{\prod}{(\alpha_1d_1'+\cdots+\alpha_rd_r')}$ for some $\beta\in\mathbb{F}_p$
\end{center}

\noindent Since $d_1',\cdots,d_r'$ are $\mathbb{F}_p$-linearly independent modulo $\lambda^+$, each term in this product has value $\lambda$, and moreover is a diagonal matrix, with only the $j$'th diagonal entry non-zero.\\

\noindent Let $\delta$ be the $j$'th diagonal entry of $\Delta$. Then $\delta\in F'$, $\delta^{-1}\Delta=e_j$, and  $v(\delta)=\underset{\alpha\in\mathbb{P}^{r-1}\mathbb{F}_p}{\sum}\lambda=\frac{p^r-1}{p-1}\lambda$. So:

\begin{center}
$0=\delta^{-p^m}\Delta^{p^m}a_i+\delta^{-p^m}\epsilon_{i,m}=e_ja_i+\delta^{-p^m}\epsilon_{i,m}$
\end{center}

\noindent and $v'(\delta^{-p^m}\epsilon_{i,m})=v'(\epsilon_{i,m})-p^m\frac{p^r-1}{p-1}\lambda\geq\frac{p^r-1}{p-1}p^m\lambda+p^mc+v(a)-\frac{p^m-1}{p-1}p^m\lambda=v(a)+p^mc\rightarrow\infty$ as $m\rightarrow\infty$.

Hence $\delta^{-p^m}\epsilon_{i,m}\rightarrow 0$ and $e_ja_i=0$ as required.\end{proof}

\subsection{Linear Dependence}

Consider again the maps $\partial_1,\cdots,\partial_r:kG\to kG$. These are $k$-linear derivations of $kG$, and we want to prove that $\partial_i(P)=0$ for all $i$.

\begin{lemma}\label{artinian}

Let $\delta:kG\to kG$ be any $k$-linear derivation of $kG$. Then if $c\tau\delta(P)=0$ for some $0\neq c\in M_{ns}(F')$ then $\tau\delta(P)=0$

\end{lemma}

\begin{proof}

Let $I=\{a\in M_{ns}(F'):a\tau\delta(P)=0\}$, then it is clear that $I$ is a left ideal of $M_{ns}(F')$, and $I\neq 0$ since $0\neq c\in I$. We want to prove that $1\in I$, and hence $\tau\delta(P)=0$.\\ 

\noindent We will first prove that $I$ is right $\widehat{Q}$-invariant:\\

\noindent Given $r\in kG$, $y\in P$, $\delta(ry)=r\delta(y)+\delta(r)y$ since $\delta$ is a derivation. So $\tau\delta(ry)=\tau(r)\tau\delta(y)+\tau\delta(r)\tau(y)=\tau(r)\tau\delta(y)$.

Therefore, for any $a\in I$, $a\tau(r)\tau\delta(y)=a\tau\delta(ry)=0$ since $ry\in P$. Thus $a\tau(r)\in I$.\\

\noindent It follows that $I$ is right $\frac{kG}{P}$-invariant.\\

\noindent Given $s\in kG$, regular mod $P$ (i.e. $\tau(s)$ is a unit in $Q(\frac{kG}{P})$), we have that $I\tau(s)\subseteq I$. Hence we have a descending chain $I\supseteq I\tau(s)\supseteq I\tau(s)^2\supseteq\cdots$ of right ideals of $M_{ns}(F')$.

So since $M_{ns}(F')$ is artinian, it follows that $I\tau(s)^n=I\tau(s)^{n+1}$ for some $n\in\mathbb{N}$, so dividing out by $\tau(s)^{n+1}$ gives that $I\tau(s)^{-1}=I$.\\

\noindent Therefore, $I$ is right $Q(\frac{kG}{P})$-invariant, and passing to the completion gives that it is right $\widehat{Q}$-invariant as required.\\

\noindent This means that $I\cap\widehat{Q}$ is a two sided ideal of the simple ring $\widehat{Q}\cong M_n(Q(D))$. We will prove that $I\cap\widehat{Q}\neq 0$, and it will follow that $I\cap\widehat{Q}=\widehat{Q}$ and thus $1\in I$.\\

\noindent We know that $\widehat{Q}\cong M_n(Q(D))$ and $Q(D)\xhookrightarrow{} M_s(F')$. Since $Q(D)$ is a division ring, we must have that $M_s(F')$ is free as a right $Q(D)$-module, so let $\{x_1,\cdots,x_t\}$ be a basis for $M_{s}(F')$ over $Q(D)$.

It follows easily that $\{x_1I_{ns},\cdots,x_tI_{ns}\}$ is a basis for $M_{ns}(F')$ over $M_n(Q(D))=\widehat{Q}$.\\

\noindent Now, $c\in I$ and $c\neq 0$, so $c=x_1c_1+\cdots+x_tc_t$ for some $c_i\in\widehat{Q}$, not all zero, and $c\tau\delta(y)=0$ for all $y\in P$.\\

\noindent Therefore $0=c\tau\delta(y)=x_1(c_1\tau\delta(y))+x_2(c_2\tau\delta_2(y))+\cdots+x_t(c_t\tau\delta(y))$, so it follows from $\widehat{Q}$-linear independence of $x_1I_{ns},\cdots,x_tI_{ns}$ that $c_i\tau\delta(y)=0$ for all $i$, and hence $c_i\in I\cap\widehat{Q}$.\\

\noindent So choose $i$ such that $c_i\neq 0$, and since $c_i\in I\cap\widehat{Q}$, we have that $I\neq 0$ as required.\end{proof}

\begin{theorem}\label{linear-dep}

Let $\delta_1,\cdots,\delta_r:kG\to kG$ be $k$-linear derivations of $kG$, and suppose that there exist matrices $a,q,d_1,\cdots,d_r\in M_{ns}(F')$ such that $a$ is invertible, the $d_i$ are diagonal of value $\lambda$, $\rho(q)>\lambda$ and for all $y\in P$: 

\begin{center}
$0=d_1^{p^m}a\tau\delta_1(y)+d_2^{p^m}a\tau\delta_2(y)+\cdots+d_r^{p^m}a\tau\delta_r(y)+O(aq^{p^m})$
\end{center}

\noindent Suppose further that $d_1,\cdots,d_r$ are $\mathbb{F}_p$-linearly independent modulo $\lambda^+$, then $\tau\delta_i(P)=0$ for all $i$.

\end{theorem}

\begin{proof}

We will use induction on $r$. First suppose that $r=1$.\\

\noindent Then since $0=d_1^{p^m}a\tau\delta_1(y)+O(aq^{p^m})$, it follows immediately from Proposition \ref{linear-indep} that $e_ja\tau\delta_1(y)=0$ for any $j=1,\cdots,ns$ such that $v(d_{1,j})=\lambda$, and this holds for all $y\in P$.

Since $a$ is a unit, $e_ja\neq 0$, so using Lemma \ref{artinian}, we see that $\tau\delta_1(P)=0$ as required.\\

\noindent Now suppose, for induction, that the result holds for $r-1$:\\

\noindent Assume first that there exists $j=1,\cdots,ns$ such that $d_{1,j},\cdots,d_{r,j}$ are $\mathbb{F}_p$-linearly independent modulo $\lambda^+$. Then using Proposition \ref{linear-indep} and Lemma \ref{artinian} again, we see that $e_ja\tau\delta_i(y)=0$ for all $i=1,\cdots,r$, $y\in P$, and hence $\tau\delta_i(P)=0$ for all $i$ as required.\\

\noindent Hence we may assume that all the corresponding entries of $d_1,\cdots,d_r$ are $\mathbb{F}_p$-linearly dependent modulo $\lambda^+$, i.e. given $j=1,\cdots,ns$, we can find $\beta_1,\cdots,\beta_r\in\mathbb{F}_p$ such that $v(\beta_1d_{1,j}+\cdots+\beta_rd_{r,j})>\lambda$. We can of course choose $j$ such that $v(d_{i,j})=\lambda$ for some $i$.

Assume without loss of generality that $v(d_{r,j})=\lambda$ and that $\beta_r\neq 0$, so after rescaling we may assume that $\beta_r=-1$.\\

\noindent It follows immediately that $v'(e_j\beta_1d_1+\cdots+e_j\beta_{r-1}d_{r-1}-e_jd_r)>\lambda$, so set 

\noindent $\epsilon:=e_j\beta_1d_{1,j}+\cdots+e_j\beta_{r-1}d_{r-1,j}-e_jd_{r,j}$ for convenience.\\

\noindent Multiplying our expression by $e_j$ gives:\\

\noindent $0=e_jd_1^{p^m}a\tau\delta_1(y)+\cdots+e_jd_{r-1}^{p^m}a\tau\delta_{r-1}(y)+e_jd_r^{p^m}a\tau\delta_r(y)+O(e_jaq^{p^m})$\\

\noindent $=e_jd_1^{p^m}a\tau\delta_1(y)+\cdots+e_jd_{r-1}^{p^m}a\tau\delta_{r-1}(y)-e_j(\beta_1d_1+\cdots+\beta_{r-1}d_{r-1})^{p^m}a\tau\delta_r(y)+\epsilon^{p^m}a\tau\delta_r(y)+O(aq^{p^m})$\\

\noindent $=(e_jd_1)^{p^m}a(\tau\delta_1-\beta_1\tau\delta_r)(y)+\cdots+(e_jd_{r-1})^{p^m}a(\tau\delta_{r-1}-\beta_{r-1}\tau\delta_r)(y)+\epsilon^{p^m}a\tau\delta_r(y)+O(aq^{p^m})$.\\

\noindent Now, set $\delta_i':=\delta_i-\beta_i\delta_r$, $d_i':=e_jd_i$ for each $i=1,\cdots,r-1$. Then the $\delta_i'$ are $k$-linear derivations of $kG$, and since $\epsilon$ is diagonal and $v'(\epsilon)>\lambda$, it follows that $\rho(\epsilon)>\lambda$, and so $\epsilon^{p^m}a\tau\delta_r(y)+O(aq^{p^m})=O(aq'^{p^m})$ for some $q'$ with $\rho(q')>\lambda$. Hence:

\begin{center}
$0=d_1'^{p^m}a\tau\delta_1'(y)+d_2'^{p^m}a\tau\delta_2'(y)+\cdots+d_{r-1}'^{p^m}a\tau\delta_{r-1}'(y)+O(aq'^{p^m})$
\end{center}

\noindent So it follows from induction that $\tau\delta_i'(P)=0$ for all $i$, i.e. for all $y\in P$, $\tau\delta_i(y)=\beta_i\tau\delta_r(y)$, and:\\

$0=d_1^{p^m}a\tau\delta_1(y)+\cdots+d_r^{p^m}a\tau\delta_r(y)+O(aq^{p^m})$\\

$=d_1^{p^m}a(\beta_1\tau\delta_r)(y)+\cdots+d_{r-1}^{p^m}a(\beta_{r-1}\tau\delta_{r})(y)+d_r^{p^m}a\tau\delta_r(y)+O(aq^{p^m})$\\

$=(\beta_1d_1+\beta_2d_2+\cdots+\beta_{r-1}d_{r-1}+d_r)^{p^m}a\tau\delta_r(y)+O(aq^{p^m})$.\\

\noindent But since $d_1,\cdots,d_r$ are $\mathbb{F}_p$-linearly independent modulo $\lambda^+$, it follows that 

\noindent $v'(\beta_1q_1+\cdots+\beta_{r-1}d_{r-1}+d_r)=\lambda$, so just applying Proposition \ref{linear-indep} and Lemma \ref{artinian} again gives that $\tau\delta_r(P)=0$.\\

\noindent So, we have $0=d_1^{p^m}a\tau\delta_1(y)+\cdots+d_{r-1}^{p^m}a\tau\delta_{r-1}(y)+O(aq^{p^m})$. Therefore $\tau\delta_i(P)=0$ for all $i<r$ by induction.\end{proof}

\begin{corollary}\label{C}

Let $G$ be a $p$-valuable, abelian-by-procyclic group, $P\in Spec^f(kG)$ such that $Q(\frac{kG}{P})$ is a CSA. Then $P$ is controlled by a proper open subgroup of $G$.

\end{corollary}

\begin{proof}

We know that $t_1,\cdots,t_r$ are $\mathbb{F}_p$-linearly independent modulo $\lambda^+$ by Lemma \ref{technical}, so applying Theorem \ref{linear-dep} with $\delta_i=\partial_i$ and $d_i=t_i$, it follows that $\tau\partial_i(P)=0$ for all $i=1,\cdots,r$.

Hence $P$ is controlled by a proper open subgroup of $G$ by Proposition \ref{Frattini}.\end{proof}

\noindent So, from now on, we may assume that $Q(\frac{kG}{P})$ is not a CSA.

\section{The Extended Commutator Subgroup}

Again, fix a non-abelian, $p$-valuable, abelian-by-procyclic group $G$, with principal subgroup $H$, procyclic element $X$. We will assume further that $G$ has \emph{split-centre}, i.e. $1\to Z(G)\to G\to\frac{G}{Z(G)}\to 1$ is split exact.

\subsection{Uniform groups}

Assume for now that $G$ is \emph{uniform}, i.e. that $(G,G)\subseteq G^{p^\epsilon}$, where $\epsilon=2$ if $p=2$ and $\epsilon=1$ if $p>2$. Note that  uniform group $G$ is $p$-saturable using the $p$-valuation $\omega(g)=\max\{n\in\mathbb{N}:g\in G^{p^{n-\epsilon}}\}$ (see \cite[Chapter 4]{DDMS} for full details).\\

\noindent Recall from \cite[Chapter 9]{DDMS}, that a free $\mathbb{Z}_p$-Lie algebra $\mathfrak{g}$ of finite rank is \emph{powerful} if $[\mathfrak{g},\mathfrak{g}]\subseteq p^{\epsilon}\mathfrak{g}$. It follows from \cite[Theorem 9.10]{DDMS} that a $p$-saturable group $G$ is uniform if and only if $\log(G)$ is powerful.

Let $\mathfrak{g}=\log(G)$, then using Lie theory, we see that $\mathfrak{g}=\mathfrak{h}$ $\rtimes$ Span$_{\mathbb{Z}_p}\{x\}$, where $\mathfrak{h}=\log(H)$, $x=\log(X)$, and $[\mathfrak{g},\mathfrak{g}]=[x,\mathfrak{h}]\subseteq p\mathfrak{h}$.\\

\noindent Recall that a map $w:\mathfrak{g}\to\mathbb{R}\cup\{\infty\}$ is a \emph{valuation} if for all $u,v\in\mathfrak{g}$, $\alpha\in\mathbb{Z}_p$:

\begin{itemize}
\item $w(u+v)\geq\min\{w(u),w(v)\}$, \item $w([u,v])\geq w(u)+w(v)$, \item $w(\alpha u)=v_p(\alpha)+w(u)$, \item $w(u)=\infty$ if and only if $u=0$, \item $w(u)>\frac{1}{p-1}$.
\end{itemize}

\noindent Also recall from \cite[Proposition 32.6]{Schneider} that if $w$ is a valuation on $\mathfrak{g}$, then $w$ corresponds to a $p$-valuation $\omega$ on $G$ defined by $\omega(g):=w(\log(g))$.

\begin{proposition}\label{p-valuation}

Let $G$ be a non-abelian, uniform, abelian-by-procyclic group with split-centre, let $\mathfrak{g}:=\log(G)$, and let $V:=\exp([\mathfrak{g},\mathfrak{g}])\subseteq H^p$. Then there exists a basis $\{h_1,\cdots,h_d\}$ for $H$, $r\leq d$ such that $\{h_{r+1},\cdots,h_{d}\}$ is a basis for $Z(G)$ and $\{h_1^{p^{t_1}},\cdots,h_r^{p^{t_r}}\}$ is a basis for $V$ for some $t_i\geq 1$.\\

\noindent Moreover, there exists an abelian $p$-valuation $\omega$ on $G$ such that $(i)$ $\{h_1,\cdots,h_d,X\}$ is an ordered basis for $(G,\omega)$, and $(ii)$ $\omega(h_1^{p^{t_1}})=\omega(h_2^{p^{t_2}})=\cdots=\omega(h_r^{p^{t_r}})>\omega(X)$.

\end{proposition}

\begin{proof}

\noindent First, note that since $G$ has split centre, we have that $G\cong Z(G)\times\frac{G}{Z(G)}$. In fact, since $Z(G)\subseteq H$, we have that $H\cong Z(G)\times H'$ for some $H'\leq H$, normal and isolated in $G$.

It follows that $\mathfrak{h}=Z(\mathfrak{g})\oplus\mathfrak{h}'$, where $\mathfrak{h}':=\log(H')$, and clearly $[\mathfrak{g},\mathfrak{g}]=[x,\mathfrak{h}]=[x,\mathfrak{h}']$.\\

\noindent By the Elementary Divisors Theorem, there exists a basis $\{v_1,\cdots,v_{r}\}$ for $\mathfrak{h}'$ such that $\{p^{t_1}v_1,....,p^{t_{r}}v_{r}\}$ is a basis for $[x,\mathfrak{h}']$ for some $t_i\geq 0$. And since $\mathfrak{g}$ is powerful, we have in fact that $t_i\geq\epsilon$ for each $i$.\\

\noindent Let $\{v_{r+1},\cdots,v_d\}$ be any basis for $Z(\mathfrak{g})$, and set $h_i:=\exp(v_i)$ for each $i=1,\cdots d$. It follows that $\{h_1^{p^{t_1}},.....,h_{r}^{p^{{t_r}}}\}$ is a basis for $V$, and that $\{h_{r+1},\cdots,h_d\}$ is a basis for $Z(G)$ as required.\\

\noindent Now, the proof of \cite[Lemma 26.13]{Schneider} shows that if $\omega$ is any $p$-valuation on $G$ and we choose $c>0$ with $\omega(g)>c+\frac{1}{p-1}$ for all $g\in G$, then $\omega_c(g):=\omega(g)-c$ defines a new $p$-valuation on $G$ satisfying $\omega_c((g,h))>\omega_c(g)+\omega_c(h)$, which preserves ordered bases.\\

\noindent So if $\omega$ is an \emph{integer valued} $p$-valuation satisfying $i$ and $ii$, then take $c:=\frac{1}{e}$ for any integer $e\geq 2$ and $\omega_c$ will also satisfy $i$ and $ii$. Also $\omega_c(G)\subseteq\frac{1}{e}\mathbb{Z}$ and $\omega_c((g,h))>\omega_c(g)+\omega_c(h)$ for all $g,h\in G$, i.e. $\omega_c$ is abelian.\\

\noindent Therefore, it remains to show that we can define an integer valued $p$-valuation on $G$ satisfying $i$ and $ii$.\\

\noindent Assume without loss of generality that $t_1\geq t_i$ for all $i=1,\cdots,r$. Choose $a\in\mathbb{Z}$ with $a>\epsilon$, and set $a_i:=a+t_1-t_i$ for each $i$, so that $a_i+t_i=a_j+t_j$ for all $i,j=1,\cdots,r$.\\

\noindent For convenience, set $v_{d+1}:=x$, and for each $i>r$, set $a_{i}=\epsilon$. Then define:

\begin{center}
$w:\mathfrak{g}\to\mathbb{Z}\cup\{\infty\}, \underset{i=1,...,d}{\sum}{\alpha_i v_i}\mapsto \min\{v_p(\alpha_i)+a_i:i=1,....,d\}$.
\end{center}

\noindent We will prove that $w$ is a valuation on $\mathfrak{g}$, and that $w(p^{t_i}v_i)=w(p^{t_j}v_j)>w(x)$ for all $i,j\leq r$. Then by defining $\omega$ on $G$ by $\omega(g)=w(\log(g))$, the result will follow.\\

\noindent Firstly, the property that $w(p^{t_i}u_i)=w(p^{t_j}u_j)>w(x)$ is clear, since $w(p^{t_i}u_i)=t_i+a_i=t_j+a_j=w(p^{t_j}u_j)$ for all $i,j<d$, and $a_i+t_i=a+t_1\geq a>1=w(x)$.\\

\noindent It is also clear from the definition of $w$ that $w(u+v)\geq\min\{w(u),w(v)\}$, $w(\alpha u)=v_p(\alpha)+w(u)$, $w(u)=\infty$ if and only if $u=0$, and $w(u)>\frac{1}{p-1}$ for all $u,v\in\mathfrak{g}$, $\alpha\in\mathbb{Z}_p$.\\

\noindent Therefore it remains to prove that $w([u,v])\geq w(u)+w(v)$, and it is straightforward to show that it suffices to prove this for basis elements. 

So since $v_{r+1},\cdots,v_d$ are central, we need only to show that $w([x,v_i])\geq w(x)+w(v_i)$ for all $i\leq r$.\\

\noindent We have that $[x,v_i]=\alpha_{i,1}p^{t_1}v_1+....+\alpha_{i,r}p^{t_{r}}v_{r}$ for some $\alpha_{i,j}\in\mathbb{Z}_p$, so:\\

\noindent $w([x,v_i])=\underset{j=1,...,r}{\min}\{v_p(\alpha_{i,j})+t_j+a_j\}=\underset{j=1,....,r}{\min}\{v_p(\alpha_{i,j})\}+t_i+a_i\geq a_i+t_i\geq a_i+\epsilon=w(v_i)+w(x)$.\end{proof}

\noindent\textbf{\underline{Remark:}} This result does not hold in general if $G$ is not uniform. For example, if $p>2$ and $\mathfrak{g}=$ Span$_{\mathbb{Z}_p}\{x,y,z\}$ with $[y,z]=0$, $[x,y]=py$, $[x,z]=y+pz$, then $\mathfrak{g}$ is not powerful, and there is no valuation $w$ on $\mathfrak{g}$ that equates the values of basis elements for $[\mathfrak{g},\mathfrak{g}]$.

\vspace{0.2in}

\noindent Now suppose that $G$ is any non-abelian, $p$-valuable, abelian-by-procyclic group with split centre, principal subgroup $H$.

\begin{lemma}\label{transport}

For any $h\in H$, $g\in G$, set $v=\log(h)$, $u=\log(g)$, then $(g,h)=\exp(\underset{n\geq 1}{\sum}{\frac{1}{n!}(ad(u))^n(v)})$

\end{lemma}

\begin{proof}

$ghg^{-1}=g\exp(v)g^{-1}=\underset{n\geq 0}{\sum}{\frac{1}{n!}(gvg^{-1})^n}=\exp(gvg^{-1})$.\\

\noindent Let $l_x,r_x$ be left and right multiplication by $x$, then note that $l_{\exp(x)}=\exp(l_x)$, same for $r_x$.\\

\noindent Then $gvg^{-1}=\exp(u)v\exp(u)^{-1}=\exp(u)v\exp(-u)=(l_{\exp(u)}r_{\exp(-u)})(v)$\\

\noindent $=(\exp(l_u)\exp(r_{-u}))(v)=\exp(l_u-r_u)(v)=\exp(ad(u))(v)=\sum_{n\geq 0}{\frac{1}{n!}(ad(u))^n(v)}$.\\

\noindent Therefore $ghg^{-1}=\exp(gvg^{-1})=\exp(\underset{n\geq 0}{\sum}{\frac{1}{n!}(ad(u))^n(v)})$.\\

\noindent Finally, $\log((g,h))=\log((ghg^{-1})h^{-1})=\log(ghg^{-1})-\log(h)$ since $h$ and $ghg^{-1}$ commute. Clearly this is equal to $\underset{n\geq 1}{\sum}{\frac{1}{n!}(ad(u))^n(v)}$ as required.\end{proof}

\noindent We know that $G=H\rtimes\langle X\rangle$, so for each $m\in\mathbb{N}$, define $G_m:=H\rtimes\langle X^{p^{m}}\rangle$, and it is immediate that $G_m$ is an open, normal subgroup of $G$, and that it is non-abelian, $p$-valuable, abelian-by-procyclic with principal subgroup $H$, procyclic element $X^{p^{m}}$ and split centre.

\begin{lemma}\label{uniform}

There exists $m\in\mathbb{N}$ such that $G_m$ is a uniform group.

\end{lemma}

\begin{proof}

Recall that $G$ is an open subgroup of the $p$-saturated group $Sat(G)$, i.e there exists $t\in\mathbb{N}$ with $Sat(G)^{p^t}\subseteq G$. Choose any such $t$ and let $m:=t+\epsilon$.\\

\noindent Given $h\in H$, by Lemma \ref{transport} we have that $(X^{p^{m}},h)=\exp(\underset{n\geq 1}{\sum}{\frac{1}{n!}(ad(p^{m}x)^n(v)})$ where $x=\log(X)$ and $v=\log(h)$ lie in $\log(Sat(G))$.\\

\noindent We want to prove that $(X^{p^m},h)\in H^{p^{\epsilon}}=G'^{p^{\epsilon}}\cap H$, so since $Sat(G)^{p^m}=Sat(G)^{p^{t+\epsilon}}\subseteq G^{p^{\epsilon}}$, it suffices to prove that $\frac{1}{n!}ad(p^{m}x)^n(v)\in p^{m}\log(Sat(G))$ for all $n\geq 1$.\\

\noindent Clearly, for each $n$, $ad(p^{m}x)^n(v)=[p^{m}x,ad(p^{m}x)^{n-1}(u)]$, so we only need to prove that $ad(p^{m}x)^{n-1}(u)\in p^{v_p(n!)}\log(Sat(G))$, in which case:\\

\noindent $\frac{1}{n!}ad(p^{m}x)^n(v)=\frac{p^{m}}{n!}[x,ad(p^{m}x)^{n-1}(u)]\in \frac{p^{v_p(n!)+m}}{n!}\log(Sat(G))\subseteq p^{m}\log(Sat(G))$.\\

\noindent Let $w$ be a saturated valuation on $\log(Sat(G))$, i.e. if $w(x)>\frac{1}{p-1}+1$ then $x=py$ for some $y\in\log(Sat(G))$.\\ 

\noindent Then since $w(ad(p^{m}x)^{n-1}(u))\geq (n-1)w(p^{m}x)+w(u)>\frac{n-1}{p-1}+\frac{1}{p-1}$, it follows that $ad(p^{m}x)^{n-1}(u)=p^kv$ for some $v\in\log(Sat(G))$, $k\geq\frac{n-1}{p-1}$.\\

\noindent We will show that $k\geq v_p(n!)$, and it will follow that $ad(p^{m}x)^{n-1}(u)=p^kv\in p^{v_p(n!)}\log(Sat(G))$.\\

\noindent If $n=a_0+a_1p+\cdots+a_rp^r$ for some $0\leq a_i<p$, then let $s(n)=a_0+a_1+\cdots+a_r$. We know from \cite[\rom{1} 2.2.3]{Lazard} that $v_p(n!)=\frac{n-s(n)}{p-1}$.

Suppose that $v_p(n!)>\frac{n-1}{p-1}$, i.e. $\frac{n-s(n)}{p-1}>\frac{n-1}{p-1}$, and hence $s(n)<1$. This means that $s(n)=0$ and hence $n=0$ -- contradiction.\\

\noindent Therefore $k\geq \frac{n-1}{p-1}\geq v_p(n!)$ as required.\end{proof}

\subsection{Extension of Filtration}

From now on, fix $c\in\mathbb{N}$ minimal such that $G_c$ is uniform, we know that this exists by Lemma \ref{uniform}. Let $\mathfrak{g}:=\log(G_c)$ -- a powerful $\mathbb{Z}_p$-subalgebra of $\log(Sat(G))$.

\begin{definition}\label{c(G)}

Define the \emph{extended commutator subgroup} of $G$ to be 

\begin{center}
$c(G):=\left(Z(G)\times\exp([\mathfrak{g},\mathfrak{g}])\right)\rtimes\langle X^{p^c}\rangle\subseteq G_c$.
\end{center}

\end{definition}

\begin{proposition}\label{extended-commutator}

If $G$ is any $p$-valuable, abelian-by-procyclic group with split centre, then:\\

 i. $c(G)$ is an open normal subgroup of $G$.\\

 ii. There exists a basis $\{k_1,k_2,\cdots,k_{d}\}$ for $H$ such that $\{k_{r+1},\cdots,k_d\}$ is a basis for $Z(G)$ and 

$\{u_{c}(k_1),u_{c}(k_2),\cdots,u_{c}(k_{r}),k_{r+1},\cdots,k_d,X^{p^c}\}$ is a basis for $c(G)$.\\

 iii. We may choose this basis $\{k_1,\cdots,k_d\}$ such that for each $i\leq r$, there exist $\alpha_{i,j}\in\mathbb{Z}_p$ with $p\mid\alpha_{i,j}$ for 
 
 $j<i$ and $\alpha_{i,i}=1$, such that $Xu_{c}(k_i)X^{-1}=u_{c}(k_1)^{\alpha_{i,1}}\cdots u_{c}(k_d)^{\alpha_{i,d}}$.\\

 iv. There is an abelian $p$-valuation $\omega$ on $c(G)$ such that $\{u_{c}(k_1),\cdots,u_{c}(k_r),k_{r+1},\cdots,k_d,X^{p^c}\}$ is an ordered 
 
 basis for $(c(G),\omega)$ and $\omega(u_{c}(k_i))=\omega(u_{c}(k_j))>\omega(X^{p^c})$ for all $i,j\leq r$.

\end{proposition}

\begin{proof}

\noindent Let $V=\exp([\mathfrak{g},\mathfrak{g}])$, and let $x=\log(X)\in\log(Sat(G))$.\\

\noindent If $h\in V$ then $h=\exp([p^c x,u])$ for some $u\in\log(H)$, i.e. $u=\log(k)$, and so:

\begin{center}
$h=\exp([\log(X^{p^c}),\log(k)])=\underset{n\rightarrow\infty}{\lim}{(X^{p^{n+c}}k^{p^n}X^{-p^{n+c}}k^{-p^n})^{p^{-2n}}}$ by \cite[\rom{4}. 3.2.3]{Lazard}.
\end{center}

\noindent Thus $XhX^{-1}=\underset{n\rightarrow\infty}{\lim}{(X^{p^{n+c}}(XkX^{-1})^{p^n}X^{-p^{n+c}}(XkX^{-1})^{-p^n})^{p^{-2n}}}=\exp([\log(X^{p^c}),\log(XkX^{-1})])$.\\

\noindent Clearly this lies in $V$, and hence $V$ is normal in $G$.\\

\noindent Using Lemma \ref{transport}, it is straightforward to show that for all $h\in H$, $(X^{p^{c}},h)\in V$, therefore:

\begin{center}
$hX^{p^{c}}h^{-1}=(X^{p^{c}},h)^{-1}X^{p^{c}}\in V\rtimes\langle X^{p^c}\rangle$
\end{center}

\noindent It follows that $c(G)=Z(G)\times V\rtimes\big\langle X^{p^c}\big\rangle$ is normal in $G$.\\

\noindent Using Proposition \ref{p-valuation}, we may choose a basis $\{h_1,\cdots,h_d\}$ for $H$ such that $\{h_{r+1},\cdots,h_d\}$ is a basis for $Z(G)$ and $\{h_1^{p^{t_1}},\cdots,h_r^{p^{t_r}}\}$ is a basis for $V$.

Therefore $c(G)$ has basis $\{h_1^{p^{t_1}},\cdots,h_r^{p^{t_r}},h_{r+1},\cdots,h_d,X^{p^{c}}\}$, and hence it is open in $G$ as required.\\

\noindent Now, for each $i=1,\cdots,d$, let $u_i=\log(h_i)$, then $\{u_1,\cdots,u_d\}$ is a $\mathbb{Z}_p$-basis for $\log(H)$, and $\{p^{t_1}u_1,\cdots,p^{t_r}u_r\}$ is a basis for $[\mathfrak{g},\mathfrak{g}]$.\\

\noindent Therefore, for each $i$, $p^{t_i}u_i=[p^c x,v_i]$ for some $v_i\in\log(H)$, in fact we may assume that $v_i\in$ Span$_{\mathbb{Z}_p}\{u_1,\cdots,u_r\}$, and it follows that $\{v_1,\cdots,v_r\}$ forms a basis for Span$_{\mathbb{Z}_p}\{u_1,\cdots,u_r\}$.\\

\noindent Let $k_i:=\exp(v_i)$ for each $i=1,\cdots r$, and for $i>r$ set $k_i=h_i$. Then we know that 

\begin{center}
$u_{c}(k_i)=\exp([p^{c}\log(X),\log(k_i)])=\exp([p^{c}x,v_i])=\exp(p^{t_i}u_i)=h_i^{p^{t_i}}$ for each $i\leq r$
\end{center}

\noindent It follows that $\{u_{c}(k_1),\cdots,u_{c}(k_r),k_{r+1},\cdots,k_d,X^{p^c}\}$ is a basis for $c(G)$, thus giving part $ii$.\\

\noindent Now, $V$ is normal in $G$, and clearly $V^p$ is also normal, so consider the action $\psi$ of $X$ on the $r$-dimensional $\mathbb{F}_p$-vector space $\frac{V}{V^p}$, i.e. $\psi(hV^p)=(XhX^{-1})V^p$. It is clear that this action $\psi$ is $\mathbb{F}_p$-linear.\\

\noindent Furthermore, since $G_c$ is uniform and $X^{p^c}\in G_c$, we have that $\psi^{p^c}=id$, i.e. $(\psi-id)^{p^c}=0$. Therefore $\psi$ has a $1$-eigenvector in $\frac{V}{V^p}$.

It follows that we may choose a basis for $\frac{V}{V^p}$ such that $\psi$ is represented by an upper-triangular matrix, with $1$'s on the diagonal.\\

\noindent This basis is obtained by transforming $\{u_{c}(k_1),\cdots,u_{c}(k_r)\}=\{h_1^{p^{t_1}},\cdots,h_r^{p^{t_r}}\}$ by an invertible matrix over $\mathbb{Z}_p$. The new basis will also have the same form $\{u_{c}(k_1'),\cdots,u_{c}(k_r')\}=\{h_1'^{p^{t_1}},\cdots,h_r'^{p^{t_r}}\}$ as described by $ii$, and it will satisfy $iii$ as required.\\

\noindent Finally, using Proposition \ref{p-valuation}, we see that there is an abelian $p$-valuation $\omega$ on the uniform group $G_c$ such that $\omega(u_{c}(k_i))=\omega(u_{c}(k_j))>\omega(X^{p^c})$ for all $i,j\leq r$, and of course $\omega$ restricts to $c(G)$, which gives us part $iv$.\end{proof}

\noindent We call a basis $\{k_1,\cdots,k_d\}$ for $H$ satisfying these conditions a \emph{$k$-basis} for $H$.\\

\noindent This result gives us a $p$-valuation that equates the values of $u(k_i)$ and $u(k_j)$ for each $i,j$. Unfortunately, with the standard Lazard filtration, this does not mean that $w(u(k_i)-1)=w(u(k_j)-1)$.

\begin{theorem}\label{equalising}

Let $G$ be a non-abelian, $p$-valuable, abelian-by-procyclic group with split centre. Let $c(G)$ be the extended commutator subgroup, and let $\{k_1,\cdots,k_{d}\}$ be a $k$-basis for $H$.

Then there exists a complete, Zariskian filtration $w:kG\to\mathbb{Z}\cup\{\infty\}$ such that:\\ 

\noindent i. For all $i,j=1,\cdots,r$, $w(u_{c}(k_i)-1)=w(u_{c}(k_j)-1)=\theta$ for some integer $\theta>0$.\\

\noindent ii. The associated graded $gr$ $kG\cong k[T_1,\cdots,T_{d+1}]\ast \frac{G}{c(G)}$, where $T_i=$ gr$(u_{c}(k_i)-1)$ for $i\leq r$, $T_i=$ gr$(k_i-1)$ for $r+1\leq i\leq d$ and $T_{d+1}=$ gr$(X^{p^{c}}-1)$. Each $T_i$ has positive degree, and $deg(T_i)=\theta$ for $i=1,\cdots,r$.\\

\noindent iii. Set $\bar{X}:=$ gr$(X)$. Then $T_r$ is central, and for each $i< r$, $\bar{X}T_i\bar{X}^{-1}=T_i+D_i$ for some $D_i\in$ Span$_{\mathbb{F}_p}\{T_{i+1},\cdots,T_r\}$.\\

\noindent Let $A:=(k[T_1,\cdots,T_{d+1}])^{\frac{G}{c(G)}}$ be the ring of invariants, then $A$ is Noetherian, central in gr $kG$, and $k[T_1,\cdots,T_{d+1}]$ is finitely generated over $A$.

\end{theorem}

\begin{proof}

Set $U=c(G)=\langle u_{c}(k_1),\cdots,u_{c}(k_r),k_{r+1},\cdots,k_d,X^{p^{c}}\rangle$. Then $U$ is an open, normal subgroup of $G$ by Proposition \ref{extended-commutator}($i$), and hence $kG\cong kU\ast\frac{G}{U}$.\\

\noindent Using Proposition \ref{extended-commutator}($iv$), we choose an abelian $p$-valuation $\omega$ on $U$ such that $\frac{1}{e}\theta:=\omega(u_{c}(k_i))=\omega(u_{c}(k_j))>\omega(X^{p^{c}})$ for all $i,j\leq r$, where $\theta>0$ is an integer. Then we can define the Lazard valuation $w$ on $kU$ with respect to $\omega$.

Since $\{u_{c}(k_1),\cdots,u_{c}(k_{r}),k_{r+1},\cdots,k_d,X^{p^{c}}\}$ is an ordered basis for $(U,\omega)$, it follows from the definition of $w$ that:

\begin{center}
$w(u_{c}(k_j)-1)=e\omega(u_{c}(k_j))=e\omega(u_{c}(k_i))=w(u_{c}(k_i)-1)=\theta$ for all $i,j\leq r$.
\end{center}

\noindent Furthermore, we have that if $V=\exp([\mathfrak{g},\mathfrak{g}])\subseteq U$ and $r\in kV$, then $w(r)\geq \theta$.\\

\noindent We want to apply Proposition \ref{crossed product} and extend $w$ to $kG\cong kU\ast\frac{G}{U}$. So we only need to verify that for all $g\in G$, $r\in kU$, $w(grg^{-1})=w(r)$, and it suffices to verify this property for $r=u_{c}(k_1)-1,\cdots,u_{c}(k_{r})-1,$

\noindent $k_{r+1}-1,\cdots,k_d-1,X^{p^{c}}-1$, since they form a topological generating set for $kU$.\\

\noindent Since $k_{r+1},\cdots,k_d\in Z(G)$, it is obvious that $w(g(k_l-1)g^{-1})=w(k_l-1)$ for each $r+1\leq l\leq d$, $g\in G$.\\

\noindent For each $j\leq r$, $gu_{c}(k_j)g^{-1}\in V$, thus:\\

\noindent $w(gu_{c}(k_j)g^{-1}-1)\geq \theta=w(u_{c}(k_j)-1)$ and it follows easily that equality holds.\\

\noindent Finally, $g=hX^{\beta}$ for some $h\in H$, $\beta\in\mathbb{Z}_p$, so

\begin{center}
$gX^{p^{c}}g^{-1}-1=hX^{p^{c}}h^{-1}-1=((h,X^{p^{c}})-1)(X^{p^{c}}-1)+((h,X^{p^{c}})-1)+(X^{p^{c}}-1)$.
\end{center}

\noindent Hence $w(g(X^{p^{c}}-1)g^{-1})\geq\min\{w((h,X^{p^{c}})-1),w(X^{p^{c}}-1)\}$, with equality if $w((h,X^{p^{c}})-1)\neq w(X^{p^{c}}-1)$. But since $(h,X^{p^{c}})\in V$, we have that 

\begin{center}
$w((h,X^{p^{c}})-1)\geq \theta=e\omega(u_{c}(k_i))>e\omega(X^{p^{c}})=w(X^{p^{c}}-1)$
\end{center}

\noindent and hence $w(gX^{p^{c}}g^{-1}-1)=w(X^{p^{c}}-1)$ as required. Note that it is here that we need the fact that $\omega(u_{c}(k_i))>\omega(X^{p^{c}})$.\\

\noindent Therefore we can apply Proposition \ref{crossed product}, and extend $w$ to $kG$ so that gr$_{w}$ $kG\cong ($gr$_w$ $kU)\ast\frac{G}{U}$, and we have that gr$_w$ $kU\cong k[T_1,....,T_{d+1}]$ as usual, where $T_i=$ gr$(u_{c}(k_i)-1)$ has degree $\theta$ for $i\leq r$, $T_i=$ gr$(k_i-1)$ for $r+1\leq i\leq d$ and $T_{d+1}=$ gr$(X^{p^{c}}-1)$ as required.\\

\noindent Using Proposition \ref{extended-commutator}($iii$), we see that $Xu_{c}(k_i)X^{-1}=u_{c}(k_1)^{\alpha_{1,i}}\cdots u_{c}(k_r)^{\alpha_{1,i}}$ where $p\mid\alpha_j$ for $j<i$ and $\alpha_{i,i}=1$.

Hence $\bar{X}T_i\bar{X}^{-1}=\overline{\alpha}_{i,1}T_1+\cdots+\overline{\alpha}_{i,r}T_r=T_i+\overline{\alpha}_{i+1,1}T_{i+1}+\cdots+\overline{\alpha}_{i,r}T_r$ for each $i\leq r$, thus giving part $iii$.\\ 

\noindent Now, every element $u\in $ gr $kU$ is a root of the polynomial $\underset{g\in\frac{G}{c(G)}}{\prod}{(s-gug^{-1})}\in A[s]$, hence gr $kU$ is integral over $A$.\\

\noindent So since $k\subseteq A$ and gr $kU\cong k[T_1,....,T_{d+1}]$ is a finitely generated $k$-algebra, we have that gr $kU$ is a finitely generated $A$-algebra, and hence finitely generated as an $A$-module by the integral property.\\

\noindent So it follows that gr$_{w}$ $kG$ is finitely generated as a right $A$-module. Furthermore, since gr $kU$ is Noetherian and commutative, it follows from \cite[Theorem 2]{Enkin} that $A$ is Noetherian.\\

\noindent Furthermore, it is easy to show that the twist $\frac{G}{U}\times\frac{G}{U}\to ($gr $kU)^{\times}$ of the crossed product is trivial, so it follows that if $r\in $ gr $kU$ is invariant under the action of $\frac{G}{U}$ then it is central. Hence $A$ is central in gr $kG$.\end{proof}

\noindent\textbf{\underline{Note:}} For any $h\in H$, we have that $(u_{c}(h)-1)+F_{\theta+1}kG\in$ Span$_{\mathbb{F}_p}\{T_1,\cdots,T_r\}$.

\vspace{0.2in}

\noindent Now, let $P$ be a faithful prime ideal of $kG$, and let $w$ be a filtration on $kG$ satisfying the conditions of the proposition. Then $w$ induces the quotient filtration $\overline{w}$ of $\frac{kG}{P}$.\\

\noindent By \cite[Ch.\rom{2} Corollary 2.1.5]{LVO}, $P$ is closed in $kG$, and hence $\frac{kG}{P}$ is complete, and gr$_{\overline{w}}$ $\frac{kG}{P}\cong\frac{\text{gr }P}{\text{gr }P}$ is Noetherian. Therefore $\overline{w}$ is Zariskian by \cite[Ch.\rom{2} Theorem 2.1.2]{LVO}, and clearly $\tau:kG\to\frac{kG}{P}$ is continuous.\\

\noindent For convenience, set $\overline{T}:=T+$ gr $P\in$ gr $\frac{kG}{P}$ for all $T\in $ gr $kc(G)=k[T_1,\cdots,T_{d+1}]$.\\

\noindent Let $\overline{A}:=\frac{A+\text{gr }P}{\text{gr }P}$ be the image of $A$ in gr $\frac{kG}{P}$, and let $A':=(k[\overline{T}_1,\cdots,\overline{T}_{d+1}])^{\frac{G}{c(G)}}$ be the ring of $\frac{G}{c(G)}$-invariants in $k[\overline{T}_1,\cdots,\overline{T}_{d+1}]=\frac{k[T_1,\cdots,T_{d+1}]+\text{gr }P}{\text{gr } P}$.\\

\noindent Since $\frac{G}{c(G)}$-invariant elements in gr $kG$ are $\frac{G}{c(G)}$-invariant modulo gr $P$, it is clear that $\overline{A}\subseteq A'\subseteq\frac{k[T_1,\cdots,T_{d+1}]+\text{gr }P}{\text{gr } P}$.\\

\noindent Then since $k[T_1,\cdots,T_{d+1}]$ is finitely generated over $A$ by Theorem \ref{equalising}, it follows that $A'$ is finitely generated over the Noetherian ring $\overline{A}$, hence $A'$ is Noetherian.\\

\noindent Therefore, $\frac{kG}{P}$ is a prime ring with a Zariskian filtration $\overline{w}$ such that gr $\frac{kG}{P}$ is finitely generated over a central, Noetherian subring $A'$. Hence we may apply Theorem \ref{filtration} to produce a non-commutative valuation on $Q(\frac{kG}{P})$.

\subsection{A special case}

In the next section, we will prove Theorem \ref{B} in full generality, but first we need to deal with a special case:\\

\noindent Fix a $k$-basis $\{k_1,\cdots,k_s\}$ for $H$, and a Zariskian filtration $w$ on $kG$ satisfying the conditions of Theorem \ref{equalising}. Then we have that $T_r\in A$ and $\bar{X}T_i\bar{X}^{-1}=T_i+D_i$ for some $D_i\in$ Span$_{\mathbb{F}_p}\{T_{i+1},\cdots,T_{r}\}$ for all $i<r$.\\

\noindent We will now suppose that for each $i<r$, $D_i$ is nilpotent modulo gr $P$.\\

\noindent Then for sufficiently high $m$, $\bar{X}T_i^{p^m}\bar{X}^{-1}\equiv (T_i+D_i)^{p^m}=T_i^{p^m}+D_i^{p^m}\equiv T_i^{p^m}$ (mod gr $P$), i.e. $\overline{T}_i^{p^m}\in A'$\\

\noindent Fix an integer $m_0$ such that $\overline{T}_i^{p^{m_0}}\in A'$ for all $i\leq r$.

\begin{proposition}\label{CSA}

Suppose that for each $i=1,\cdots,r$, $T_i$ is nilpotent modulo gr $P$, i.e. $\overline{T}_i$ is nilpotent. Then $Q(\frac{kG}{P})$ is a central simple algebra.

\end{proposition}

\begin{proof}

Using Theorem \ref{equalising}($ii$), every element of gr $kG$ has the form

\begin{center}
$\underset{g\in\frac{G}{c(G)}}{\sum}(\underset{\alpha\in\mathbb{N}^{d+1}}{\sum}{\lambda_{\alpha}T_1^{\alpha_1}\cdots T_{d+1}^{\alpha_{d+1}}})g$
\end{center} 

\noindent where $\lambda_{\alpha}=0$ for all but finitely many $\alpha$.\\

\noindent Therefore, it follows immediately from nilpotence of $\overline{T}_1,\cdots,\overline{T}_r$ that $\frac{\text{gr }kG}{\text{gr }P}$ is finitely generated over $\frac{k[T_{r+1},\cdots,T_{d+1}]+\text{gr }P}{\text{gr }P}$.\\

\noindent But since $Z(G)=\langle k_{r+1},\cdots,k_d\rangle$ by Proposition \ref{extended-commutator}($ii$), it follows that under the quotient filtration, gr $\frac{k(Z(G)\times \langle X^{p^{c}}\rangle)+P}{P}\cong\frac{k[T_{r+1},\cdots,T_{d+1}]+\text{gr }P}{\text{gr }P}$.\\

\noindent So since gr $\frac{kG}{P}$ is finitely generated over gr $\frac{k(Z(G)\times\langle X^{p^{c}}\rangle) +P}{P}$, and $\frac{k(Z(G)\times\langle X^{p^{c}}\rangle) +P}{P}$ is closed in $\frac{kG}{P}$, it follows from \cite[Ch.\rom{1} Theorem 5.7]{LVO} that $\frac{kG}{P}$ is finitely generated over $\frac{k(Z(G)\times\langle X^{p^{c}}\rangle) +P}{P}$.\\

\noindent But $k(Z(G)\times\langle X^{p^{c}}\rangle)$ is commutative, so $\frac{kG}{P}$ is  finitely generated as a right module over a commutative subring. Therefore, by \cite[Corollary 13.1.14(iii)]{McConnell}, $\frac{kG}{P}$ satisfies a polynomial identity.

So since $\frac{kG}{P}$ is prime, it follows from Posner's theorem \cite[Theorem 13.6.5]{McConnell}, that $Q(\frac{kG}{P})$ is a central simple algebra. \end{proof}

\noindent\textbf{\underline{Note:}} This proof relies on the split centre property, without which we would not be able to argue that $\frac{k[T_{r+1},\cdots,T_{d+1}]+\text{gr }P}{\text{gr }P}$ arises as the associated graded of some commutative subring of $\frac{kG}{P}$.\\

\vspace{0.1in}

\noindent So let us assume that $Q(\frac{kG}{P})$ is not a CSA. Then by the proposition, we know that there exists $s\leq r$ such that $\overline{T}_s$ is not nilpotent.\\

\noindent Since we know that $\overline{T}_s^{p^{m_0}}\in A'$, it follows there exists a minimal prime ideal $\mathfrak{q}$ of $A'$ such that $\overline{T}_s^{p^{m_0}}\notin\mathfrak{q}$. Using Theorem \ref{filtration}, we let $v=v_{\mathfrak{q}}$ be the non-commutative valuation on $Q(\frac{kG}{P})$ corresponding to $\mathfrak{q}$.\\

\noindent So, $\overline{T}_s^{p^{m_0}}=$ gr$\tau(u_{c}(k_s)-1)^{p^{m_0}}\in A'\backslash\mathfrak{q}$, and hence using Theorem \ref{comparison}, we see that $\tau(u_{c}(k_s)-1)^{p^k}$ is $v$-regular for some $k\geq m_0$, and hence $\rho(\tau(u_{c}(k_s)-1)^{p^k})=v(\tau(u_{c}(k_s)-1)^{p^k})$.\\

\noindent Recall that $\lambda=\inf\{\rho(\tau(u(g)-1)):g\in G\}<\infty$ by Lemma \ref{faithful}.

\begin{lemma}\label{minimal}

Let $h\in H$ such that $\rho(\tau(u(h)-1))=\lambda$. Then $\tau(u(h)-1)^{p^m}$ is $v$-regular for sufficiently high $m$.

\end{lemma}

\begin{proof}

It is clear that $w(u_c(h)-1)\geq\theta=w(u_c(k_s)-1)$, so let $T(h):=u_c(h)-1+F_{\theta+1}kG\in$ Span$_{\mathbb{F}_p}\{T_1,\cdots,T_r\}$. Then $T(h)=$ gr$(u_c(h)-1)$ if and only if $w(u_c(h)-1)=\theta$, otherwise $T(h)=0$.\\

\noindent We know that gr$(u_{c}(k_s)-1)=T_s\notin$ gr $P$, and hence $\overline{w}(\tau(u_{c}(k_s)-1))=w(u_{c}(k_s)-1)$, giving that $\overline{w}(\tau(u_{c}(h)-1))\geq w(u_{c}(h)-1)\geq w(u_{c}(k_s)-1)=\overline{w}(\tau(u_{c}(k_s)-1))$.\\

\noindent Also, $T(h)^{p^{m_0}}+$ gr $P\in A'$, and if $T(h)^{p^{m_0}}+$ gr $P\notin\mathfrak{q}$ then $T(h)^{p^{m_0}}+$ gr $P=$ gr$_{\overline{w}}(\tau(u_c(h)-1)^{p^{m_0}})\in A'\backslash\mathfrak{q}$, so it follows from Theorem \ref{comparison} that $\tau(u(h)-1)^{p^m}$ is $v$-regular for $m>>0$.

So, suppose for contradiction that $T(h)^{p^{m_0}}+$ gr $P\in\mathfrak{q}$:\\

\noindent If $T(h)^{p^{m_0}}+$ gr $P=0$ then $\overline{w}(\tau(u_c(h)-1)^{p^{m_0}})>p^{m_0}\theta=\overline{w}(\tau(u_c(k_s)-1))$, and if $T(h)^{p^{m_0}}\neq 0$ then $T(h)^{p^{m_0}}+$ gr $P=$ gr$_{\overline{w}}(\tau(u_c(h)-1)^{p^{m_0}})\in\mathfrak{q}$. In either case, using Theorem \ref{comparison}, it follows that for $m$ sufficiently high:

\begin{center}
$v(\tau(u_{c}(h)-1)^{p^m})>v(\tau(u_{c}(k_s)-1)^{p^m})$.
\end{center}

\noindent Therefore, $v(\tau(u(h)-1))>v(\tau(u(k_s)-1))$, so since $\tau(u(k_s)-1)^{p^k}$ is $v$-regular:\\

\noindent $\rho(\tau(u(h)-1))\geq \frac{1}{p^k}v(\tau(u(h)-1)^{p^k})>\frac{1}{p^k}v(\tau(u(k_s)-1)^{p^k})=\rho(\tau(u(k_s)-1))\geq\lambda$ -- contradiction.\end{proof}

\noindent Recall the definition of a growth preserving polynomial (GPP) from Section 2.4, and recall that the identity map is a non-trivial GPP.

\begin{proposition}\label{special-case}

Suppose that $Q(\frac{kG}{P})$ is not a CSA, and that $D_i$ is nilpotent mod gr $P$ for all $i<r$. Then $id:\tau(kH)\to\tau(kH)$ is a special GPP with respect to some non-commutative valuation on $Q(\frac{kG}{P})$.

\end{proposition}

\begin{proof}

This is immediate from Definition \ref{GPP} and Lemma \ref{minimal}.\end{proof}

\noindent Therefore, we will assume from now on that $D_i$ is not nilpotent mod gr $P$ for some $i$.\\

\noindent\textbf{\underline{Remark:}} If $G$ is uniform then $D_i=0$ for all $i$, but in general we cannot assume this. For example, if $p>2$ and $G=\langle X,Y,Z\rangle$ where $Y$ and $Z$ commute, $XYX^{-1}=Y^r$, $XZX^{-1}=(YZ)^r$, $r=e^p\in\mathbb{Z}_p$, then $G$ is non-uniform, $c=1$, and $\{Z,Y^{p-1}Z^p\}$ is a $k$-basis for $H=\langle Y,Z\rangle$. In this case, $\bar{X}T_2\bar{X}^{-1}=T_2$, $\bar{X}T_1\bar{X}^{-1}=T_1+T_2$.

\section{Construction of Growth Preserving Polynomials}

\noindent As in the previous section, we will take $G$ to be a $p$-valuable, abelian-by-procyclic group with split centre, $P$ a faithful prime ideal of $kG$, and $w$ a Zariskian filtration on $kG$ satisfying the conditions of Theorem \ref{equalising}.\\

\noindent We fix a $k$-basis $\{k_1,\cdots,k_d\}$ for $H$, and let $T_i:=$ gr$(u_{c}(k_i)-1)$ for each $i\leq r$, $T_i:=$ gr$(k_i-1)$ for $i>r$. We know that $T_r,\cdots,T_d$ are central, and $D_i:=\bar{X}T_i\bar{X}^{-1}-T_i\in$ Span$_{\mathbb{F}_p}\{T_{i+1},\cdots,T_r\}$ for each $i<r$.\\

\noindent We now assume that not all the $D_i$ are nilpotent modulo gr $P$, so let $s<r$ be maximal such that $D_s$ is not nilpotent, i.e. for all $i>s$, $\overline{T}_i^{p^m}\in A'$ for sufficiently high $m$.

Throughout, we will fix $m_0$ such that $\overline{T}_i^{p^{m_0}}\in A'$ for all $i>s$, and we may assume that $m_1\geq m_0$.\\

\noindent By definition, we know that $D_s\in$ Span$_{\mathbb{F}_p}\{T_{s+1},\cdots,T_r\}$, and hence $\overline{D}_s^{p^{m_0}}\in A'$. So since $D_s$ is not nilpotent mod gr $P$, we can fix a minimal prime ideal $\mathfrak{q}$ of $A'$ such that $\overline{D}_s^{p^{m_0}}\in A'\backslash\mathfrak{q}$.\\

\noindent Let $v=v_{\mathfrak{q}}$ be the corresponding non-commutative valuation given by Theorem \ref{filtration}.

\subsection{The Reduction Coefficients}

\noindent Define a function $L$ of commuting variables $x$ and $y$ by:

\begin{equation}
L(x,y):=x^p-xy^{p-1}
\end{equation}

\noindent Moreover, for commuting variables $x,y_1,y_2,\cdots,y_n$, define the iterated function 

\begin{equation}
L^{(n)}(x,y_1,y_2,\cdots,y_n):=L(L(L(\cdots(L(x,y_1),y_2),\cdots),y_n)
\end{equation}

\noindent For $n=0$, we define $L^{(n)}(x,y_1,\cdots,y_n):=x$.\\

\noindent We can readily see that for any commutative $\mathbb{F}_p$-algebra $S$, $y_1,\cdots,y_n\in S$, $L(-,y_1,\cdots,y_n)$ is $\mathbb{F}_p$-linear.

\begin{lemma}\label{polynomial}

Let $S$ be an $\mathbb{F}_p$-algebra, and let $y_1,\cdots,y_n\in S$ commute. Then there exist $a_0,a_1,\cdots,a_{n-1}\in S$ such that $L^{(n)}(x,y_1,\cdots,y_n)=a_0x+a_1x^p+\cdots+a_{n-1}x^{p^{n-1}}+x^{p^n}$ for all $x$ commuting with $y_1,\cdots,y_n$.

\end{lemma}

\begin{proof}

Both statements are trivially true for $n=0$, so assume that they hold for some $n\geq 0$ and proceed by induction on $n$:\\

\noindent So $L^{(n)}(x,y_1,\cdots,y_n)=a_0x+a_1x^p+\cdots+a_{n-1}x^{p^{n-1}}+x^{p^n}$, and:\\

\noindent $L^{(n+1)}(x,y_1,\cdots,y_{n+1})=L^{(n)}(x,y_1,\cdots,y_n)^p-L^{(n)}(x,y_1,\cdots,y_n)y_{n+1}^{p-1}$\\

\noindent $=(a_0x+\cdots+a_{n-1}x^{p^{n-1}}+x^{p^n})^p-(a_0x+\cdots+a_{n-1}x^{p^{n-1}}+x^{p^n})y_{n+1}^{p-1}$\\

\noindent $=(-a_0y_{n+1}^{p-1})x+(a_0^p-a_1y_{n+1}^{p-1})x^p+\cdots+(a_{n-2}^p-a_{n-1}y_{n+1}^{p-1})x^{p^{n-1}}+(a_{n-1}^p-y_{n-1}^{p-1})x^{p^n}+x^{p^{n+1}}$.\\

\noindent So setting $b_0=(-a_0y_{n+1}^{p-1})$, $b_i=(a_{i-1}^p-a_iy_{n+1}^{p-1})$ for $1\leq i\leq n$ (taking $a_n:=1$), we have that $L^{(n+1)}(x,y_1,\cdots,y_{n+1})=b_0x+b_1x^p+\cdots+b_nx^{p^n}+x^{p^{n+1}}$ as required.\end{proof}

\vspace{0.1in}

\noindent Now, let $B_s:=D_s$, and for each $1\leq i<s$, let $B_i:=L^{(s-i)}(D_i,B_s,\cdots,B_{i+1})$.

\begin{lemma}\label{centralising}

For each $i\leq s$, $L^{(s-i+1)}(\overline{T}_i,\overline{B}_s,\cdots,\overline{B}_i)^{p^{m_0}}\in A'$, so in particular $\overline{B}_i^{p^{m_0}}\in A'$.

\end{lemma}

\begin{proof}

\noindent Note that $C\in k[T_1,\cdots,T_d]$ is central if and only if it is invariant under the action of $\bar{X}$.\\ 

\noindent Also, $L(C,C)=C^p-CC^{p-1}=0$, and if $D$ is $\bar{X}$-invariant, then $\bar{X}L(C,D)\bar{X}^{-1}=L(\bar{X}C\bar{X}^{-1},D)$\\

\noindent\textbf{\underline{Notation:}} Let $Y':=\overline{Y}^{p^{m_0}}$ for any $Y\in$ gr $kG$.\\

\noindent We will proceed by downwards induction on $i$, starting with $i=s$. Clearly $B_s'=D_s'\in$ Span$_{\mathbb{F}_p}\{T_{s+1}',\cdots,T_{r}'\}$ is invariant under the action of $\bar{X}$, so:\\

\noindent $\bar{X}L(T_s',B_s')\bar{X}^{-1}=L(\bar{X}T_s'\bar{X}^{-1},B_s')=L(T_s'+B_s',B_s')=L(T_s',B_s')+L(B_s,B_s)=L(T_s',B_s')$.\\

\noindent Therefore $L(T_s',B_s')$ is $\bar{X}$-invariant and the result holds.\\

\noindent Suppose we have the result for all $s\geq j>i$.\\

\noindent Then $B_i'=L^{(s-i)}(D_i',B_s',\cdots,B_{i+1}')$, and $D_i'\in$ Span$_{\mathbb{F}_p}\{T_{i+1}',\cdots,T_{r}'\}$.\\ 

\noindent Using linearity of $L(-,y)$ we have that 

\begin{center}
$B_i'\in$ Span$_{\mathbb{F}_p}\{L^{(s-i)}(T_{j}',B_s',\cdots,B_{i+1}'):j=i+1,\cdots,r\}$
\end{center} 

\noindent therefore $B_i'$ is $\bar{X}$-invariant by the inductive hypothesis.\\

\noindent Also, since $B_s',\cdots,B_{i+1}'$ are $\bar{X}$-invariant, we have that:\\

\noindent $\bar{X}L^{(s-i)}(T_i',B_s',\cdots,B_{i+1}')\bar{X}^{-1}=L^{(s-i)}(\bar{X}T_i'\bar{X}^{-1},B_s',\cdots,B_{i+1}')$\\

\noindent $=L^{(s-i)}(T_i'+\bar{X}T_i'\bar{X}^{-1}-T_i',B_s',\cdots,B_{i+1}')=L^{(s-i)}(T_i',B_s',\cdots,B_{i+1}')+B_i$\\ 

\noindent The final equality follows from linearity of $L(-,B_s',\cdots,B_{i+1}')$ and the fact that $D_i'=\bar{X}T_i'\bar{X}^{-1}-T_i'$.\\

\noindent Set $C:=L^{(s-i)}(T_i',B_s',\cdots,B_{i+1}')$, so that 

\begin{center}
$L^{(s-i+1)}(T_i',B_s',\cdots,B_i')=L(A,B_i')=C^p-CB_i'^{p-1}$, and $\bar{X}C\bar{X}^{-1}=C+B_i'$
\end{center} 

\noindent Then:\\

\noindent $\bar{X}L(C,B_i')\bar{X}^{-1}=L(\bar{X}C\bar{X}^{-1},B_i')=L(C+B_i',B_i')=L(C,B_i')+L(B_i',B_i')=L(C,B_i')$\\

\noindent Hence $L(C,B_i')=L^{(s-i+1)}(T_i',B_s',\cdots,B_i')$ is $\bar{X}$-invariant as required.\end{proof}

\noindent It follows immediately from this Lemma that $L^{(s)}(\overline{T},\overline{B}_s,\cdots,\overline{B}_1)^{p^{m_0}}\in A'$ for all $T\in$ Span$_{\mathbb{F}_p}\{T_1,\cdots,T_{r}\}$ (i.e. for all $T=u_{c}(h)-1+F_{\theta+1}kG$).\\

\vspace{0.1in}

\noindent Now, for each $i\leq s$, $D_i\in$ Span$_{\mathbb{F}_p}\{T_{i+1},\cdots,T_r\}$, so either $D_i=0$ or $D_i=u_c(f_i)-1+F_{\theta+1}kG$ for some $f_i\in H$ with $w(u_c(f_i)-1)=\theta$.

\begin{definition}\label{reduction}

Define $y_s:=u_{c}(f_s)-1$, and for each $1\leq i<s$, define $y_i\in kH$ inductively by:\\

\noindent $y_i:=\begin{cases} 
       L^{(s-i)}(u_c(f_i)-1,y_s,\cdots,y_{i+1}) & D_i\neq 0 \\
      0 & D_i=0
   \end{cases}
$

\noindent And define $b_i:=\tau(y_i)^{p^{m_1-c}}\in Q(\frac{kG}{P})$, we call these $b_i$ the \emph{reduction coefficients}.

\end{definition}

\noindent For convenience, we will replace $m_1$ by $m_1+c$, so that $\tau(u(g)-1)=\tau(u_c(g)-1)^{p^{m_1}}$ for all $g\in G$, and $b_i=\tau(y_i)^{p^{m_1}}$. Since $B_s\notin$ gr $P$, it is clear that gr$_{\overline{w}}(b_s)=\overline{B}_s^{p^{m_1}}$.\\

\noindent Since gr$_{\overline{w}}(b_s)=\overline{B}_s^{p^{m_1}}\in A'\backslash\mathfrak{q}$, it follows from Theorem \ref{comparison} that $b_s^{p^k}$ is $v$-regular for some $k\in\mathbb{N}$. After replacing $m_1$ by $m_1+k$, we may also assume that $b_s$ is $v$-regular.

\begin{lemma}\label{reduction-value}

For all $h\in H$, $i\leq s$, let $t:=u_c(h)-1$, $T:=t+F_{\theta+1}kG\in$ gr $kG$.\\ 

\noindent Then $w(L^{(s-i)}(t,y_s,\cdots,y_{i+1}))\geq p^{s-i}\theta$, with equality if and only if gr$(L^{(s-i)}(t,y_s,\cdots,y_{i+1}))=L^{(s-i)}(T,B_s,\cdots B_{i+1})$, otherwise $L^{(s-i)}(T,B_s,\cdots B_{i+1})=0$.

In particular, for $y_i\neq 0$, $w(y_i)\geq p^{s-i}\theta$, with equality if and only if $B_i=$ gr$(y_i)$, otherwise $B_i=0$.

\end{lemma}

\begin{proof}

We will use downwards induction on $i$, with $i=s$ as the base case:\\

\noindent Since $i=s$, $L^{(s-i)}(t,y_s,\cdots,y_{i+1})=t$, and clearly $w(t)\geq\theta=p^{s-s}\theta$, and equality holds if and only if gr$(t)=T$, and otherwise $T=0$ as required.\\

\noindent Now suppose the result holds for some $i\leq s$, so let $c:=L^{(s-i)}(t,y_s,\cdots,y_{i+1})$, $C:=L^{(s-i)}(T,B_s,\cdots,B_{i+1})$. Then by induction, $w(c)\geq p^{s-i}\theta$, with equality if and only if $C=$ gr$(c)$, and $w(y_i)\geq p^{s-i}\theta$.\\

\noindent So $w(L^{(s-i+1)}(t,y_s,\cdots,y_i))=w(c^p-cy_i^{p-1})\geq\min\{w(c^p),w(cy_i^{p-1})\}\geq\min\{pw(c),w(c)+(p-1)w(y_s)\}\geq p^{s-i+1}\theta$ as required.

In particular, this argument shows that if $w(c)>p^{s-i}\theta$ then $w(c^p-cy_i^{p-1})>p^{s-i+1}\theta$.\\

\noindent Therefore, if $w(L^{(s-i+1)}(t,y_s,\cdots,y_i))=p^{s-i+1}\theta$ then $w(c)=p^{s-i}\theta$, and so $C=$ gr$(c)=c+F_{p^{s-i}\theta} kG$ by induction. In this case, gr$(L^{(s-i+1)}(t,y_s,\cdots,y_i))=(c^p-cy_i^{p-1})+F_{p^{s-i+1}\theta} kG$. \\

\noindent Also, since $c,y_i\in kc(G)$ and $w$ is a valuation on $kc(G)$, we have that $w(c^p)=pw(c)=p^{s-i+1}\theta$ and $w(cy_i^{p-1})=w(c)+(p-1)w(y_i)\geq p^{s-i+1}\theta$.\\

\noindent If $w(y_i)>p^{s-i}\theta$ then $c^p-cy_i^{p-1}+F_{\theta+1}kG=c^p+ F_{\theta+1}kG=C^p$. But since $B_i=0$ by induction, this means that $c^p-cy_i^{p-1}+F_{p^{s-i+1}\theta+1}=C^p-CB_i^{p-1}$.

Whereas, if $w(y_i)=p^{s-i}\theta$ then $B_i=$ gr$(y_i)=y_i+F_{p^{s-i}\theta+1}kG$ by assumption, so $c^p-cy_i^{p-1}+F_{p^{s-i+1}\theta+1}kG=C^p-CB_i^{p-1}$ as required.\\

\noindent Finally, if $w(L^{(s-i+1)}(t,y_s,\cdots,y_i))>p^{s-i+1}\theta$ then $w(c^p-cy_i^{p-1})>p^{s-i+1}\theta$. Clearly if $C=0$ then $L^{(s-i+1)}(T,B_s,\cdots,B_i)=C^p-CB_i^{p-1}=0$, so we may assume that $C\neq 0$, and hence $C=$ gr$(c)=c+F_{p^{s-i}\theta+1}kG$.\\

\noindent So since $w(c^p-cy_i^{p-1})>p^{s-i+1}\theta$, it follows that $C^p-CB_i^{p-1}=0$ as required.\end{proof}

\noindent Using this Lemma, we see that $\overline{w}(b_i)\geq p^{s-i}\theta$, with equality if and only if gr$_{\overline{w}}(b_i)=\overline{B}_i^{p^{m_1}}$.\\

\noindent\textbf{\underline{Notation:}} For each $0\leq i\leq s$, $q\in Q(\frac{kG}{P})$ commuting with $b_1,\cdots,b_s$, define $L_i(q):=L^{(s-i)}(q,b_s,\cdots,b_{i+1})$.\\ 

\noindent e.g. $L_s(q)=q$, $L_{s-1}(q)=q^p-qb_s^{p-1}$, $L_{s-2}(q)=q^{p^2}-q^p(b_s^{p^2-p}-b_{s-1}^{p-1})+qb_s^{p-1}b_{s-1}^{p-1}$.

\subsection{Polynomials}

Again, we have $\lambda=\inf\{\rho(\tau(u(g)-1)):g\in G\}<\infty$.

\begin{lemma}\label{inequality}

For each $i\leq s$, $\rho(b_i)\geq p^{s-i}\lambda$, and it follows that $\rho(L_i(\tau(u(h)-1))\geq p^{s-i}\lambda$ for all $h\in H$.

Moreover, if $\rho(b_i)=p^{s-i}\lambda$ then $b_i^{p^m}$ is $v$-regular for $m$ sufficiently high.

\end{lemma}

\begin{proof}

For $i=s$ the result is clear, because $b_s=\tau(u(f_s)-1)$, so $\rho(b_s)\geq\lambda=p^{s-s}\lambda$ by definition. So we will proceed again by downwards induction on $i$.\\ 

\noindent The inductive hypothesis states that $\rho(b_{i+1})\geq p^{s-i-1}\lambda$, and $\rho(L_{i+1}(\tau(u(h)-1)))\geq p^{s-i-1}\lambda$ for all $h\in H$.\\

\noindent Thus $\rho(L_i(\tau(u(h)-1)))=\rho(L_{i+1}(\tau(u(h)-1))^p-L_{i+1}(\tau(u(h)-1))b_{i+1}^{p-1})$\\

\noindent $\geq\min\{p\cdot p^{s-i-1}\lambda,p^{s-i-1}\lambda+(p-1)p^{s-i-1}\lambda\}=p^{s-i}\lambda$ for all $h$.\\

\noindent By definition, $b_i=L^{(s-i)}(\tau(u(f_i)-1),b_s,\cdots,b_{i+1})=L_i(\tau(u(f_i)-1))$, so 

\begin{center}
$\rho(b_i)=\rho(L_i(\tau(u(f_i)-1)))\geq p^{s-i}\lambda$, and the first statement follows.
\end{center}

\noindent Finally, suppose that $\rho(b_i)=p^{s-i}\lambda$:\\

\noindent Then if $\overline{w}(b_i)>p^{s-i+m_1}\theta=\overline{w}(b_s^{p^{s-i}})$, then $v(b_i^{p^m})>v(b_s^{p^{s-i+m}})$ for $m>>0$ by Theorem \ref{comparison}. So using $v$-regularity of $b_s$, we see that $\rho(b_i)>\rho(b_s^{p^{s-i}})\geq p^{s-i}\lambda$ -- contradiction.\\

\noindent Therefore, by Lemma \ref{reduction-value}, we see that $\overline{w}(b_i)=p^{s-i}\theta$ and gr$_{\overline{w}}(b_i)=\overline{B}_i^{p^{m_1}}$.\\

\noindent We know that $\overline{B}_i^{p^{m_1}}\in A'$, so suppose that $\overline{B}_i^{p^{m_1}}\in\mathfrak{q}$. Then since $\overline{w}(b_i)=p^{s-i}\theta=\overline{w}(b_s^{p^{s-i}})$, it follows again from Theorem \ref{comparison} that $v(b_i^{p^m})>v(b_s^{p^{s-i+m}})$ for $m>>0$, and hence $\rho(b_i)>p^{s-i}\rho(b_s)\geq p^{s-i}\lambda$ -- contradiction.\\

\noindent Hence $\overline{B}_i^{p^{m_0}}=$ gr$_{\overline{w}}(b_i)\in A'\backslash\mathfrak{q}$, so $b_i^{p^k}$ is $v$-regular for some $k\in\mathbb{N}$ by Theorem \ref{comparison}.\end{proof}

\vspace{0.1in}

\noindent Now, using Lemma \ref{polynomial}, we see that $L_i(x)=L^{(s-i)}(x,b_s,\cdots,b_{i+1})=a_0x+a_1x^p+\cdots+a_{s-i-1}x^{p^{s-i-1}}+x^{p^{s-i}}$ for some $a_j\in\tau(kH)$.\\

\begin{proposition}\label{growth-preserving}

For each $i\leq s$, $L_i$ is a growth preserving polynomial of $p$-degree $s-i$, and $L_s$ is not trivial.

\end{proposition}

\begin{proof}

Firstly, it is clear that $L_s=id$, and so $L_s$ is a non-trivial GPP.\\ 

\noindent We first want to prove that for all $q\in\tau(kH)$, if $\rho(q)\geq\lambda$ then $\rho(L_i(q))\geq p^{s-i}\lambda$, with strict inequality if $\rho(q)>\lambda$. We know that this holds for $i=s$, so as in the proof of Lemma \ref{inequality}, we will use downwards induction on $i$.\\

\noindent So suppose that $\rho(L_{i+1}(q))\geq p^{s-i-1}\lambda$, with strict inequality if $\rho(q)>\lambda$. Then:\\

\noindent $L_i(q)=L_{i+1}(q)^p-L_{i+1}(q)b_{i+1}^{p-1}$, so $\rho(L_i(q))\geq\min\{\rho(L_{i+1}(q)^p),\rho(L_{i+1}(q)b_{i+1}^{p-1})\}$.\\

\noindent But $\rho(L_{i+1}(q)^p)\geq p\cdot p^{s-i-1}\lambda=p^{s-i}\lambda$, and since $\rho(b_{i+1})\geq p^{s-i}\lambda$ by Lemma \ref{inequality}, $\rho(L_{i+1}(q)b_{i+1}^{p-1})\geq\rho(L_{i+1}(q))+(p-1)\rho(b_{i+1})\geq p^{s-j-1}\lambda+(p-1)p^{s-i-1}\lambda=p^{s-i}\lambda$.\\

\noindent By the inductive hypothesis, both these inequalities are strict if $\rho(q)>\lambda$, and thus $L_i$ is a GPP as required.\end{proof}

\noindent So all that remains is to prove that one of the $L_i$ is special.

\subsection{Trivial growth preserving polynomials}

Let us first suppose that for some $j$, $L_j$ is trivial, i.e. $L_j(\tau(u(h)-1))>p^{s-j}\lambda$ for all $h\in H$. \\

\noindent We know that $L_s$ is not trivial, so we can fix $j\leq s$ such that $L_j$ is non-trivial and $L_{j-1}$ is trivial. We will need the following results:

\begin{lemma}\label{limit}

Let $A$ be a $k$-algebra, with filtration $w$ such that $A$ is complete with respect to $w$. Suppose $a\in A$ and $w(a^p-a)>0$, then $a^{p^m}\rightarrow b\in A$ with $b^p=b$ as $m\rightarrow\infty$.

\end{lemma}

\begin{proof}

Let $\varepsilon:=a^p-a$, then $w(\varepsilon)>0$, $a$ commutes with $\varepsilon$, and $a^p=a+\varepsilon$.\\

\noindent Therefore, since $char(k)=p$, $a^{p^2}=a^p+\varepsilon^p=a+\varepsilon+\varepsilon^p$, and it follows from induction that for all $m\in\mathbb{N}$, $a^{p^{m+1}}=a+\varepsilon+\varepsilon^p+\cdots+\varepsilon^{p^m}$.\\

\noindent But $\varepsilon^{p^m}\rightarrow 0$ as $m\rightarrow\infty$ since $w(\varepsilon)>0$, so since $A$ is complete, the sum $\underset{m\geq 0}{\sum}{\varepsilon^{p^m}}$ converges in $A$, and hence $a^{p^m}\rightarrow a+\underset{m\geq 0}{\sum}{\varepsilon^{p^m}}\in A$.\\

\noindent So let $b:=a+\underset{m\geq 0}{\sum}{\varepsilon^{p^m}}$, then $b^p=a^p+(\underset{m\geq 0}{\sum}{\varepsilon^{p^m}})^p=a+\varepsilon+\underset{m\geq 1}{\sum}{\varepsilon^{p^m}}=a+\underset{m\geq 0}{\sum}{\varepsilon^{p^m}}=b$ as required.\end{proof}

\begin{proposition}\label{divided-powers}

Let $Q=\widehat{Q(\frac{kG}{P})}$, and let $\delta_1,\cdots,\delta_r:kG\to kG$ be derivations such that $\tau\delta_i(P)\neq 0$ for all $i$. Set $N:=\{(a_1,\cdots,a_r)\in Q^r:(a_1\tau\delta_1+\cdots+a_r\tau\delta_r)(P)=0\}$.\\

\noindent Then $N$ is a $Q$-bisubmodule of $Q^r$, and either $N=0$ or there exist $\alpha_1,\cdots,\alpha_r\in Z(Q)$, not all zero, such that for all $(a_1,\cdots,a_r)\in N$, $\alpha_1a_1+\cdots+\alpha_ra_r=0$.

\end{proposition}

\begin{proof}

Since $v$ is a non-commutative valuation, we have that $Q$ is simple and artinian, and the proof that $N$ is a $Q$-bisubmodule of $Q^r$ is similar to the proof of Lemma \ref{artinian}. For the second statement, we will proceed using induction on $r$.\\

\noindent First suppose that $r=1$, then $N$ is a two sided ideal of the simple ring $Q$, so it is either $0$ or $Q$. But if $N=Q$ then $1\in Q$ so $\tau\delta_1(P)=0$ -- contradiction. Hence $N=0$.\\

\noindent Now suppose that the result holds for $r-1$ for some $r>1$. If $N\neq 0$ then there exists $(a_1,\cdots,a_r)\in N$ with $a_i\neq 0$ for some $i$, and we may assume without loss of generality that $i=1$.\\

\noindent So, let $A:=\{a\in Q:(a,a_2,\cdots, a_r)\in N$ for some $a_i\in Q\}$, then clearly $A$ is a two-sided ideal of $Q$, so $A=0$ or $A=Q$. But $A\neq 0$ since $a_1\in A$ and $a_1\neq 0$.

Therefore $A=Q$, and hence we have that for all $b\in Q$, $(b,b_2,\cdots,b_r)\in N$ for some $b_i\in Q$.\\

\noindent Let $N'=\{(a_2,\cdots,a_r)\in Q^{r-1}:(a_2\tau\delta_2+\cdots+a_r\tau\delta_r)(P)=0\}$. Suppose first that $N'=0$.\\

\noindent Then if for some $q\in Q$, $(q,x_2,\cdots,x_r),(q,x_2',\cdots,x_r')\in N$ for $x_i,x_i'\in Q$, we have that $(x_2-x_2',\cdots,x_r-x_r')\in N'=0$, and hence $x_i=x_i'$ for all $i$.

Hence there is a unique $(1,\beta_2,\cdots,\beta_r)\in N$.\\

\noindent Given $x\in Q$, $(x,x\beta_2,\cdots,x\beta_r)$, $(x,\beta_2x,\cdots,\beta_rx)\in N$, and so $([x,\beta_2],\cdots,[x,\beta_r])\in N'$. Hence $[x,\beta_i]=0$ for all $i$, so $\beta_i\in Z(Q)$.\\

\noindent Moreover, if $(a_1,\cdots,a_r)\in N$, then since $(a_1,a_1\beta_2,\cdots,a_1\beta_r)\in N$, it follows that $a_i=\beta_ia_1$ for all $i>1$, and since $\tau\delta_1(P)\neq 0$, it is clear that $\beta_i\neq 0$ for some $i$, thus giving the result.\\

\noindent So from now on, we may assume that $N'\neq 0$, so by the inductive hypothesis, this means that there exist $\alpha_2,\cdots,\alpha_r\in Z(Q)$, not all zero, such that for all $(a_2,\cdots,a_r)\in N'$, $\alpha_2a_2+\cdots+\alpha_ra_r=0$.\\

\noindent Again, suppose we have that $(a,x_2,\cdots,x_r),(a,x_2',\cdots,x_r')\in N$ for some $a,x_i,x_i'\in Q$. Then clearly $(x_2-x_2',\cdots,x_r-x_r')\in N'$, and hence $\alpha_2(x_2-x_2')+\cdots+\alpha_r(x_r-x_r')=0$, i.e. $\alpha_2x_2+\cdots+\alpha_rx_r=\alpha_2x_2'+\cdots+\alpha_rx_r'$.\\

\noindent So, given $q\in Q$, $(1,x_2,\cdots,x_r)\in N$, we have that $(q,qx_2,\cdots,qx_r),(q,x_2q,\cdots,x_rq)\in N$, and hence 

\begin{center}
$\alpha_2qx_2+\cdots+\alpha_rqx_r=\alpha_2x_2q+\cdots+\alpha_rx_rq$
\end{center} 

\noindent i.e $[q,\alpha_2x_2+\cdots+\alpha_rx_r]=0$.\\

\noindent Since this holds for all $q\in Q$, it follows that $\alpha_2x_2+\cdots+\alpha_rx_r\in Z(Q)$, so let $-\alpha_1$ be this value.\\

\noindent In fact, for any such $(1,x_2',\cdots,x_r')\in N$, $\alpha_2x_2'+\cdots+\alpha_rx_r'=\alpha_2x_2+\cdots+\alpha_rx_r=-\alpha_1$, so $-\alpha_1\in Z(Q)$ is unchanged, regardless of our choice of $x_i$.\\

\noindent Finally, suppose that $(a_1,\cdots,a_r),(1,x_2,\cdots,x_r)\in N$, then $(a_1,a_1x_2,\cdots,a_1x_r)\in N$, and hence $(a_2-a_1x_2,\cdots,a_r-a_1x_r)\in N'$. Thus $\alpha_2(a_2-a_1x_2)+\cdots+\alpha_r(a_r-a_1x_r)=0$, i.e. 

\begin{center}
$\alpha_2a_2+\cdots+\alpha_ra_r=a_1(\alpha_2x_2+\cdots+\alpha_rx_r)=-\alpha_1a_1$.
\end{center}

\noindent Therefore $\alpha_1a_1+\alpha_2a_2+\cdots+\alpha_ra_r=0$, and $\alpha_i\in Z(Q)$ as required.\end{proof}

\noindent Recall from Lemma \ref{sub} that for any GPP $f$ of $p$-degree $r$, $K_f:=\{h\in H:\rho(f(\tau(u(h)-1)))>p^r\lambda\}$ is an open subgroup of $H$ containing $H^p$, and that it is proper in $H$ if $f$ is non-trivial. For each $i\leq s$, define $K_i:=K_{L_i}$.\\

\noindent Then since $L_{j-1}$ is trivial and $L_{j}$ is not, we know that $K_{j-1}=H$ and $K_{j}$ is a proper subgroup of $H$.

\begin{lemma}\label{convergence}

There exists $k\in\mathbb{N}$ such that $b_{j}^{p^k}$ is $v$-regular of value $p^{k+s-j}\lambda$, and for any $h\in H\backslash K_{j}$, $(L_{j}(\tau(u(h)-1)b_{j}^{-1})^{p^m}\rightarrow c\in \widehat{Q(\frac{kG}{P})}$ with $c\neq 0$ and $c^p=c$.

\end{lemma}

\begin{proof}

Since $L_{j-1}$ is trivial, we know that for each $h\in H$, $\rho(L_{j-1}(\tau(u(h)-1)))>p^{s-j+1}\lambda$.\\ 

\noindent Choose $h\in H\backslash K_{j}$, i.e. $\rho(L_{j}(\tau(u(h)-1))=p^{s-j}\lambda$. Setting $q:=\tau(u(h)-1)$ for convenience, we have:

\begin{center}
$\rho(L_{j-1}(q))=\rho(L_{j}(q)^p-L_{j}(q)b_{j}^{p-1})>p^{s-j+1}\lambda$
\end{center}

\noindent But $\rho(L_{j}(q)^p-L_{j}(q)b_{j}^{p-1})\geq\min\{\rho(L_{j}(q)^p),\rho(L_{j}(q)b_{j}^{p-1})\}$, with equality if $\rho(L_{j}(q)^p)\neq\rho(L_{j}(q)b_{j}^{p-1})$.\\

\noindent So if $\rho(b_{j})> p^{s-j}\lambda$, then we have that: 

\begin{center}
$\rho(L_{j}(q)b_{j}^{p-1})>\rho(L_{j}(q))+(p-1)p^{s-j}\lambda=p^{s-j}\lambda$.
\end{center}

\noindent But $\rho(L_{j}(q)^p)=p\rho(L_{j}(q))=p^{s-j+1}\lambda$, and hence $\rho(L_{j-1}(\tau(u(h)-1)))=\min\{\rho(L_{j}(q)^p),\rho(L_{j}(q)b_{j}^{p-1})\}$

\noindent $=p^{s-j+1}\lambda$ -- contradiction.\\

\noindent Therefore, $\rho(b_{j})\leq p^{s-j}\lambda$, so using Lemma \ref{inequality}, we see that $\rho(b_{j})=p^{s-j}\lambda$, and $b_{j}^{p^k}$ is $v$-regular for some $k$, and thus $v(b_{j}^{p^k})=\rho(b_{j}^{p^k})=p^{k+s-j}\lambda$.\\

\noindent Now, $\rho((L_{j}(q)^{p^k}b_{j}^{-p^k})^{p}-(L_{j}(q)^{p^k}b_{j}^{-p^k}))=\rho(b_{j}^{-p^{k+1}}(L_{j}(q)^p-L_{j}(q)b_{j}^{p-1})^{p^k})$\\

\noindent $=\rho(L_{j-1}(q)^{p^k})-pv(b_{j}^{p^k})>p^{s-j+k+1}\lambda-p^{s-j+k+1}\lambda=0$\\

\noindent This means that $v(((L_{j}(q)b_{j}^{-1})^{p}-(L_{j}(q)b_{j}^{-1}))^{p^m})>0$ for $m>>0$, so it follows from Lemma \ref{limit} that $L_{j}(\tau(u(h)-1))b_{j}^{-1})^{p^m}$ converges to $c\in\widehat{Q(\frac{kG}{P})}$ with $c^p=c$.\\ 

\noindent Finally, since $\rho(L_{j}(\tau(u(h)-1))^{p^k}b_{j}^{-p^k})=0$, it follows that $c\neq 0$.\end{proof}

\begin{theorem}\label{maximal}

If $L_{j-1}$ is trivial and $L_{j}$ is not trivial, then $P$ is controlled by a proper open subgroup of $G$.

\end{theorem}

\begin{proof}

\noindent Since $K_{j}$ is a proper subgroup of $H$ containing $H^p$, we can choose an ordered basis $\{h_1,\cdots,h_d\}$ for $H$ such that $\{h_1^p,\cdots,h_t^p,h_{t+1},\cdots,h_d\}$ is an ordered basis for $K_{j}$.\\

\noindent Consider our Mahler expression (\ref{MahlerPol}), taking $f=L_{j}$, $q_i=\tau(u(h_i)-1)$:

\begin{equation}\label{Mahlerx}
0=L_{j}(q_1)^{p^m}\tau\partial_1(y)+\cdots+L_{j}(q_{d})^{p^m}\tau\partial_{d}(y)+O(L_{j}(q)^{p^m})
\end{equation}

\noindent Where $\rho(q)>\lambda$, and hence $\rho(L_{j}(q))>p^{s-j}\lambda$. Note that we also have:

\begin{center}
$\rho(L_{j}(q_i))=p^{s-j}\lambda$ for all $i\leq t$, and $\rho(L_{j}(q_i))>p^{s-j}\lambda$ for all $i>t$.
\end{center}

\noindent Using Lemma \ref{convergence}, we see that $b_{j}^{p^k}$ is $v$-regular of value $p^{k+s-j}$ for some $k$, and for each $i\leq t$, $(L_{j}(q_i)^{p^k}b_{j}^{-p^k})^{p^m}\rightarrow c_i\neq 0$ as $m\rightarrow\infty$, with $c_i^p=c_i$. Clearly $c_1,\cdots,c_r$ commute.\\

\noindent So, divide out our expression (\ref{Mahlerx}) by $b_{j}^{p^m}$, which is $v$-regular of value $p^{m+s-j}\lambda$ to obtain:

\begin{equation}
0=(b_{j}^{-1}L_{j}(q_1))^{p^m}\tau\partial_1(y)+\cdots+(b_{j}^{-1}L_{j}(q_{d-1}))^{p^m}\tau\partial_{d-1}(y)+O((b_{j}^{-1}L_{j}(q))^{p^m})
\end{equation}

\noindent Take the limit as $m\rightarrow\infty$ and the higher order terms will converge to zero. Hence the expression converges to $c_1\tau\partial_1(y)+\cdots+c_r\tau\partial_r(y)$.

Therefore $(c_1\tau\partial_1+\cdots+c_r\tau\partial_r)(P)=0$.\\

\noindent Now, using Proposition \ref{Frattini}, we know that if $\tau\partial_i(P)=0$ for some $i\leq r$ then $P$ is controlled by a proper open subgroup of $G$. So we will suppose, for contradiction, that $\tau\partial_i(P)\neq 0$ for all $i\leq r$.\\

\noindent Let $N:=\{(q_1,\cdots,q_r)\in\widehat{Q(\frac{kG}{P})}:(q_1\tau\partial_1+\cdots+q_r\tau\partial_r)(P)=0\}$, then $0\neq (c_1,\cdots,c_r)\in N$, so $N\neq 0$.

Therefore, using Proposition \ref{divided-powers}, we see that $c_1,\cdots,c_r$ are $Z(Q)$-linearly dependent.\\

\noindent So, we can find some $1<t\leq r$ such that $c_1,\cdots,c_t$ are $Z(Q)$-linearly dependent, but no proper subset of $\{c_1,\cdots,c_t\}$ is $Z(Q)$-linearly dependent.

It follows that we can find $\alpha_2,\cdots,\alpha_t\in Z(Q)\backslash\{0\}$ such that $c_1+\alpha_2c_2+\cdots+\alpha_tc_t=0$.\\

\noindent Therefore, since $c_i^p=c_i$ for all $i$, we also have that:

\begin{center}
$c_1+\alpha_2^pc_2+\cdots+\alpha_t^pc_t=(c_1+\alpha_2c_2+\cdots+\alpha_tc_t)^p=0$
\end{center} 

\noindent Hence $(\alpha_2^p-\alpha_2)c_2+\cdots+(\alpha_t^p-\alpha_t)c_t=0$.\\

\noindent So using minimality of $\{c_1,\cdots,c_t\}$, this means that $\alpha_i^p=\alpha_i$ for each $i$, and it follows that  $\alpha_i\in\mathbb{F}_p$ for each $i$, i.e. $c_1,\cdots,c_r$ are $\mathbb{F}_p$-linearly dependent.\\

\noindent So, we can find $\beta_1,\cdots,\beta_r\in\mathbb{F}_p$, not all zero, such that $\beta_1c_1+\cdots+\beta_rc_r=0$, or in other words:\\

\begin{center}
$\underset{m\rightarrow\infty}{\lim}{(L_{j}(\beta_1q_1+\cdots+\beta_rq_r)^{p^k}b_{j}^{-p^k})^{p^m}=0}$
\end{center}

\noindent Therefore, $\rho(L_{j}(\beta_1q_1+\cdots+\beta_rq_r)^{p^k}b_{j}^{-p^k})>0$, and hence $\rho(L_{j}(\beta_1q_1+\cdots+\beta_rq_r))>v(b_{j})=p^{s-j}\lambda$.\\

\noindent But $\beta_1q_1+\cdots+\beta_rq_r=\tau(u(h_1^{\beta_1}\cdots h_r^{\beta_r})-1)+\varepsilon$, where $\rho(\varepsilon)>\lambda$, and we know that $\rho(L_{j}(\tau(u(h_1^{\beta_1}\cdots h_r^{\beta_r})-1)))=p^{s-j}\lambda$ by the definition of $K_{j}$, and $\rho(L_{j}(\varepsilon))>p^{s-j}\lambda$ since $L_{j}$ is a GPP.

Hence $\rho(L_{j}(\beta_1q_1+\cdots+\beta_rq_r))=\rho(L_{j}(\tau(u(h_1^{\beta_1}\cdots h_r^{\beta_r})-1)))=p^{s-j}\lambda$ -- contradiction.\\

\noindent Therefore $P$ is controlled by a proper open subgroup of $G$.\end{proof}

\subsection{Control Theorem}

We may now suppose that $L_i$ is not trivial for all $i\leq s$, in particular, $L_0$ is a non-trivial GPP of $p$-degree $s$.

\begin{proposition}\label{L-special}

$L_0$ is a special growth preserving polynomial.

\end{proposition}

\begin{proof}

Given $h\in H$ such that $\rho(L_0(\tau(u(h)-1)))=p^s\lambda$, we want to prove that $L_0(\tau(u(h)-1))^{p^k}$ is $v$-regular for some $k$. Let $T=u_{c}(h)-1+F_{\theta+1}kG\in$ Span$_{\mathbb{F}_p}\{T_1,\cdots,T_r\}$.\\

\noindent We know that $L^{(s)}(\overline{T},\overline{B}_s,\cdots,\overline{B}_1)^{p^{m_1}}$ lies in $A'$ by Lemma \ref{centralising}. If $L^{(s)}(\overline{T},\overline{B}_s,\cdots,\overline{B}_1)^{p^{m_0}}\in\mathfrak{q}$ then we may assume that $L^{(s)}(\overline{T},\overline{B}_s,\cdots,\overline{B}_1)^{p^{m_1}}=0$, so using Lemma \ref{reduction-value} we see that $\overline{w}(L^{(s)}(\tau(u(h)-1),b_s,\cdots,b_1))>p^{s+m_1}\theta=\overline{w}(b_s^{p^s})$.\\

\noindent So again, since gr$_{\overline{w}}(b_s)=\overline{B}_s\in A'\backslash\mathfrak{q}$, it follows from Theorem \ref{comparison} that 

\begin{center}
$v(L^{(s)}(\tau(u(h)-1),b_s,\cdots,b_1)^{p^m})>v(b_s^{p^{m+s}})$ for $m>>0$
\end{center} 

\noindent Hence $\rho(L_0(\tau(u(h)-1)))=\rho(L^{(s)}(\tau(u(h)-1),b_s,\cdots,b_1))>\rho(b_s^{p^s})\geq p^s\lambda$ -- contradiction.\\

\noindent Therefore, we have that $L^{(s)}(\overline{T},\overline{B}_s,\cdots,\overline{B}_1)^{p^{m_1}}\in A'\backslash{q}$, and hence it is equal to 

\noindent gr$_{\overline{w}}(L^{(s)}(\tau(u(h)-1),b_s,\cdots,b_1))$ by Lemma \ref{reduction-value}.\\

\noindent It follows from Theorem \ref{comparison} that for $m>>0$, $L_0(\tau(u(h)-1))^{p^m}=L^{(s)}(\tau(u(h)-1),b_s,\cdots,b_1)^{p^m}$ is $v$-regular as required.\end{proof}

\noindent Now we can finally prove our main control theorem in all cases. But we first need the following technical result.

\begin{lemma}\label{split-centre}

Let $G=H\rtimes\langle X\rangle$ be an abelian-by-procyclic group. Then $G$ has split centre if and only if $(G,G)\cap Z(G)=1$.

\end{lemma}

\begin{proof}

It is clear that if $G$ has split centre then $(G,G)\cap Z(G)=1$. Conversely, suppose that $(G,G)\cap Z(G)=1$, and consider the $\mathbb{Z}_p$-module homomorphism $H\to H,h\mapsto (X,h)$.\\

\noindent The kernel of this map is precisely $Z(G)$, therefore $(X,H)\cong\frac{H}{Z(G)}$. So since $Z(G)\cap (X,H)=1$, it follows that $Z(G)\times (X,H)$ has the same rank as $H$, hence it is open in $H$.\\

\noindent Recall from \cite[Definition 1.6]{Billy} the definition of the \emph{isolator} $i_G(N)$ of a closed, normal subgroup $N$ of $G$, and recall from \cite[Proposition 1.7, Lemma 1.8]{Billy} that it is a closed, isolated normal subgroup of $G$, and that $N$ is open in $i_G(N)$.\\ 

\noindent Let $C=i_G((X,H))\leq H$, then it is clear that $Z(G)\cap C=1$ and that $Z(G)\times C$ is isolated.

Therefore, since $Z(G)\times (X,H)$ is open in $H$, it follows that $Z(G)\times C=H$, and hence $G=Z(G)\times C\rtimes\langle X\rangle$, and $G$ has split centre.\end{proof}

\vspace{0.3in}

\noindent\emph{Proof of Theorem \ref{B}.} Let $Z_1(G):=\{g\in G:(g,G)\subseteq Z(G)\}$, this is a closed subgroup of $G$ containing $Z(G)$ and contained in $H$. Suppose first that $Z_1(G)\neq Z(G)$.\\

\noindent Then choose $h\in Z_1(G)\backslash Z(G)$, then $(h,G)\subseteq Z(G)$, so if we take $\psi\in Inn(G)$ to be conjugation by $h$, then $\psi$ is trivial mod centre, and clearly $\psi(P)=P$. So it follows that $P$ is controlled by a proper, open subgroup of $G$ by \cite[Theorem B]{nilpotent}.\\

\noindent So from now on, we may assume that $Z_1(G)=Z(G)$.\\

\noindent Suppose that $(X,h)\in Z(G)$ for some $h\in H$, then clearly $(h,G)\subseteq Z(G)$, so $h\in Z_1(G)=Z(G)$, giving that $(X,h)=1$. It follows that $Z(G)\cap (G,G)=1$, and hence $G$ has split centre by Lemma \ref{split-centre}.\\

\noindent Therefore, using Theorem \ref{equalising}, we can choose a $k$-basis $\{k_1,\cdots,k_d\}$ for $H$ and a filtration $w$ on $kG$ such that gr$_w$ $kG\cong k[T_1,\cdots,T_{d+1}]\ast\frac{G}{c(G)}$, where $T_i=$ gr$(u_{c}(k_i)-1)$ for $i\leq r$, $T_i=$ gr$(k_i-1)$ for $i>r$, $T_r,\cdots,T_{d}$ are central and $\bar{X}T_i\bar{X}^{-1}=T_i+D_i$ for some $D_i\in$ Span$_{\mathbb{F}_p}\{T_{i+1},\cdots,T_r\}$ for all $i<r$.\\

\noindent If $Q(\frac{kG}{P})$ is a CSA, then the result follows from Corollary \ref{C}, so we may assume that $Q(\frac{kG}{P})$ is not a CSA.\\ 

\noindent Hence if each $D_i$ is nilpotent mod gr $P$, then $id:\tau(kH)\to\tau(kH)$ is a special GPP with respect to some non-commutative valuation by Proposition \ref{special-case}.

Therefore, by Theorem \ref{special-GPP}, $P$ is controlled by a proper open subgroup of $G$ as required.\\

\noindent If $D_s$ is not nilpotent mod gr $P$ for some $s<r$, then we can construct GPP's $L_s,\cdots,L_0$ with respect to some non-commutative valuation using Proposition \ref{growth-preserving}, and $L_s$ is non-trivial.\\

\noindent If $L_{j-1}$ is trivial and $L_{j}$ is non-trivial for some $0< j\leq s$, then the result follows from Theorem \ref{maximal}. Whereas if all the $L_i$ are non-trivial, then $L_0$ is a special GPP by Proposition \ref{L-special}, and the result follows again from Theorem \ref{special-GPP}.\qed

\bibliographystyle{abbrv}

\end{document}